\documentclass[]{amsart}
\usepackage{amsmath}
\usepackage{amssymb}
\usepackage{amsthm}
\usepackage{tikz}
\usepackage{tikz-cd}
\usepackage{graphicx}
\usetikzlibrary{arrows}
\usepackage{hyperref}
\usepackage{url}
\usepackage{braket}
\usepackage{nccmath}
\usepackage{stmaryrd}
\DeclareMathAlphabet\mathbfcal{OMS}{cmsy}{b}{n}

\hypersetup{
 colorlinks=true,       
    linkcolor=cyan,          
    citecolor=green,        
    filecolor=magenta,      
   urlcolor=cyan           
   }

\numberwithin{equation}{section}
\numberwithin{figure}{section}

\allowdisplaybreaks
\theoremstyle{plain}
\newtheorem{lemma}{Lemma}[section]
\newtheorem{proposition}[lemma]{Proposition}
\newtheorem{corollary}[lemma]{Corollary}

\newtheorem{theorem}[lemma]{Theorem}

\theoremstyle{definition}
\newtheorem{definition}[lemma]{Definition}
\newtheorem{remark}[lemma]{Remark}

\newtheorem{example}[lemma]{Example}

\newcommand{\R}{\mathbb{R}}
\newcommand{\Z}{\mathbb{Z}}
\newcommand{\N}{\mathbb{N}}
\newcommand{\Q}{\mathbb{Q}}

\DeclareSymbolFont{bbold}{U}{bbold}{m}{n}
\DeclareSymbolFontAlphabet{\mathbbold}{bbold}
\newcommand{\ind}{\mathbbold{1}}

\renewcommand{\P}{\mathbb{P}}
\newcommand{\E}{\mathbb{E}}

\newcommand{\F}{\mathcal{F}}

\newcommand{\supp}{\mathrm{supp}}

\newcommand{\todo}[1]{\textcolor{red}{(TODO: #1)}}

\renewcommand{\mathbf}{\boldsymbol}

\renewcommand{\circle}{\mathbb{T}}

\newcommand{\chain}{\mathcal{C}}
\newcommand{\vchain}{\mathbfcal{C}}

\newcommand{\perm}{\mathcal{S}}
\newcommand{\vperm}{\mathbfcal{S}}

\newcommand{\vinitstate}{\mathbfcal{I}}

\newcommand{\vdist}{\mathbfcal{D}}

\newcommand{\transmat}{\mathcal{K}}
\newcommand{\vtransmat}{\mathbfcal{K}}

\newcommand{\uniform}{\text{U}}
\newcommand{\vuniform}{\textbf{U}}

\newcommand{\ex}[1]{\mathrm{ex}(#1)}

\newcommand{\identity}{\textbf{I}}

\newcommand{\birkhoff}{\mathcal{B}}
\newcommand{\vbirkhoff}{\mathbfcal{B}}

\begin{document}
	
\title{Virtual Markov Chains}

\author{Steven N. Evans}
\address{Department of Statistics\\
	University  of California\\ 
	367 Evans Hall \#3860\\
	Berkeley, CA 94720-3860 \\
	U.S.A.}
\email{evans@stat.berkeley.edu}

\author{Adam Q. Jaffe}
\address{Department of Statistics\\
	University  of California\\ 
	367 Evans Hall \#3860\\
	Berkeley, CA 94720-3860 \\
	U.S.A.}
\email{aqjaffe@berkeley.edu}
\thanks{This material is based upon work for which AQJ was supported by the National Science Foundation Graduate Research Fellowship under Grant No. DGE 1752814.}

\subjclass[2010]{Primary: 60J10; Secondary: 46A55, 46M40}

\keywords{projective limit, inverse limit, virtual permutation, continuous-time Markov process, balayage, compact convex space of measures, extreme points}

\date{\today}

\begin{abstract}
We introduce the space of virtual Markov chains (VMCs) as a projective limit of the spaces of all finite state space Markov chains (MCs), in the same way that the space of virtual permutations is the projective limit of the spaces of all permutations of finite sets.
We introduce the notions of virtual initial distribution (VID) and a virtual transition matrix (VTM), and we show that the law of any VMC is uniquely characterized by a pair of a VID and VTM which have to satisfy a certain compatibility condition.
Lastly, we study various properties of compact convex sets associated to the theory of VMCs, including that the Birkhoff-von Neumann theorem fails in the virtual setting.
\end{abstract}

\maketitle


\section{Introduction}\label{sec:intro}

For each $N\in\N:=\{1,2,\ldots\}$, write $\perm_N$ for the group of all permutations of the finite set $\{1,2,\ldots, N\}$.
A \textit{virtual permutation} is a sequence $\boldsymbol{\sigma} = \{\sigma_N\}_{N\in\N}\in\prod_{N\in\N}\perm_N$ with the extra projectivity condition that for each $N\in\N$ the cycle structure of $\sigma_N$ is exactly the cycle structure of $\sigma_{N+1}$ but with the element $N+1$ removed.
The space of virtual permutations $\vperm$ was introduced in the landmark paper \cite{VirtualPermutations} as a method for studying the representation theory of the infinite symmetric group $\perm_{<\infty}:=\bigcup_{N=1}^{\infty}\perm_N$, and subsequent research has focused both on furthering this direction \cite{NewRepTheorySymGroup,SymRepTheoryPfaffian,FourierInfiniteSymmetricGroup} and on using similar ideas to study the representation theory of other ``large'' non-abelian groups \cite{RepTheoryOrthoGrp,InfiniteUnitaryRepTheory}.

There are a few natural ways to enrich this picture by introducing an element of randomness.
To motivate this, we return to the case of permutations of a finite set, that is, $\perm_N$ for a fixed $N\in\N$.
One approach is to study a family of probability measures on $\perm_N$, and then to ask any number of questions about the resulting distributions, typically in terms of the asymptotics of various statistics as $N\to\infty$.
This study of \textit{random permutations} is an extensive area of active research in probability, combinatorics, and mathematical physics,
so we will not attempt the futile task of trying to summarize this research in just a few lines.
In the case of virtual permutations, the natural generalization of this program is to define a family of probability measures directly on $\vperm$ and to ask analogous questions about the resulting distributions; this line of research has been very fruitful, and there are a number of interesting probabilistic questions that have been proposed and answered \cite{UnitaryVirtualExtension,SphericalVirtualPermutations,VirtualPermStickBreaking,Tsilevich1}.

There is another simple way to introduce an aspect of randomness to this picture.
To see what this is, we return again to the case of $\perm_N$ for a fixed $N\in\N$; the approach is to view each fixed element of $\perm_N$ as a deterministic transition rule for some dynamics on $\{1,2,\ldots, N\}$, and to then expand our scope to include probabilistic transition rules.
This study of random dynamics is, of course, nothing other than the study of Markov chains (MCs) on finite state spaces, which has become its own subfield of modern probability theory.
In the case of virtual permutations, the natural generalization of this program is to define a family of probabilistic transition rules on $\{1,2,\ldots ,N\}$ simultaneously for all $N\in\N$ in such a way that the extra projectivity condition is, in a suitable sense, always maintained.
In this paper, we the take this second perspective as a starting point and hence initiate the study of the so-called \textit{virtual Markov chains (VMCs)}.

Roughly speaking, a VMC on some probability space is a sequence $\mathbf{X} = \{X_N\}_{N\in\N}$ of MCs with the extra condition that almost surely for each $N\in\N$ the sample path $X_N$ is exactly the sample path of $X_{N+1}$ but with all instances of the element $N+1$ removed.
In particular, a VMC is nothing more than a certain (rather strong) notion of coupling for a sequence of MCs onto the same underlying probability space.
We will later give several classes of examples (Subsection~\ref{subsec:Ex-1}) which show that VMCs are naturally encoded in the enumeration of the hitting locations of various continuous-time Markov processes.

Subsequently, much of the paper is dedicated to generalizing existing results for MCs to the setting of VMCs, with an emphasis on the elements which have behavior that is different from or more interesting than the analogous elements of the classical case.
For example, one of our main results (Theorem~\ref{thm:VMC-VID-VTM-bijection}), which generalizes the classical result stating that the law of a MC is uniquely described by a pair of an initial distribution (ID) and a transition matrix (TM), shows that the law of a VMC is uniquely characterized by a pair of a \textit{virtual initial distribution (VID)} and \textit{virtual transition matrix (VTM)}, but that these two data must satisfy an additional \textit{compatibility} condition.

Another collection of results is related to various compact convex spaces associated to the study of VMCs.
In the case of MCs, the TM is usually regarded as containing ``more interesting'' information than the ID; this intuition extends to the case of VMCs as well, hence we later focus our attention on VTMs and we let VIDs play a secondary role.
From this perspective it makes sense to fix a VTM and to consider the space of all VIDs which are compatible with it; it turns out (Subsection~\ref{subsec:compatibility-revisted}) that this is a compact convex space of measures whose extreme points have many nice properties that we explore.
In particular, this study of compatibility is equivalent to characterizing the solution set of an infinite system of balayage inverse problems, which generalize the one-step balayage inverse problems originally studied in \cite{Balayage1,Balayage2}.
We also  define the notions of equilibrium distribution and stationary distribution, and show (Subsection~\ref{subsec:equi-stat}) that they are distinct, unlike in the case of classical MCs where they coincide.

Finally, we give (Subsection~\ref{subsec:Birkhoff-polytope}) a collection of results which connects the theory of VMCs back to theory of virtual permutations from whence it came.
We know in the classical setting that the space of doubly-stochastic TMs, called the \textit{Birkhoff polytope}, coincides with (closed) convex hull of the permutation matrices; in the virtual setting, we show that the space of doubly-stochastic VTMs, which we analogously call the \textit{virtual Birkhoff polytope} $\vbirkhoff$, is not even convex.
However, not all hope is lost in understading the convexity structure of $\vbirkhoff$:
We show that it exhibits an ``all-or-nothing-convexity'' property, namely, that for any two points in $\vbirkhoff$, the relative interior of the line segment joining them is either contained in or disjoint from $\vbirkhoff$.
Moreover, we show that there is a unique point in $\vbirkhoff$ with the property that the line segment joining it to any other point of $\vbirkhoff$ is contained in $\vbirkhoff$; this point is none other than the VTM of the virtual permutation corresponding to the identity; in Figure~\ref{fig:virtual-Birkhoff-space} we give a cartoon depiction of the virtual Birkhoff polytope.

\begin{figure}
	\begin{tikzpicture}
	
	\definecolor{BirkhoffColorExt}{HTML}{0235db}
	\definecolor{BirkhoffColorInt}{HTML}{8fa9ff}
	\definecolor{PermsColor}{HTML}{db0214}
	
	
	\filldraw[fill=BirkhoffColorInt,draw=BirkhoffColorExt,ultra thick] (0,0) -- (2,0) -- (1.5,0.5) -- (0,0);
	\filldraw[fill=BirkhoffColorInt,draw=BirkhoffColorExt,ultra thick] (0,0) -- (2.4,1);
	\filldraw[fill=BirkhoffColorInt,draw=BirkhoffColorExt,ultra thick] (0,0) -- (2,1);
	\filldraw[fill=BirkhoffColorInt,draw=BirkhoffColorExt,ultra thick] (0,0) -- (1.8,1.3);
	\filldraw[fill=BirkhoffColorInt,draw=BirkhoffColorExt,ultra thick] (0,0) -- (0.8,1.2) -- (2,2) -- (0,0);
	\filldraw[fill=BirkhoffColorInt,draw=BirkhoffColorExt,ultra thick] (0,0) -- (0.8,1.8);
	\filldraw[fill=BirkhoffColorInt,draw=BirkhoffColorExt,ultra thick] (0,0) -- (0.3,1.4) -- (-0.3,1.4) -- (0,0);
	\filldraw[fill=BirkhoffColorInt,draw=BirkhoffColorExt,ultra thick] (0,0) -- (-0.7,1.8);
	\filldraw[fill=BirkhoffColorInt,draw=BirkhoffColorExt,ultra thick] (0,0) -- (-0.4,0.6);
	\filldraw[fill=BirkhoffColorInt,draw=BirkhoffColorExt,ultra thick] (0,0) -- (-0.7,0.6);
	\filldraw[fill=BirkhoffColorInt,draw=BirkhoffColorExt,ultra thick] (0,0) -- (-2.1,0.9);
	\filldraw[fill=BirkhoffColorInt,draw=BirkhoffColorExt,ultra thick] (0,0) -- (-0.9,0.2) -- (-1.8,-0.2) -- (-0.9,-0.4) -- (0,0);
	\filldraw[fill=BirkhoffColorInt,draw=BirkhoffColorExt,ultra thick] (0,0) -- (-2,-1.2) -- (-0.2,-0.5) -- (0,0);
	\filldraw[fill=BirkhoffColorInt,draw=BirkhoffColorExt,ultra thick] (0,0) -- (-0.2,-1.5);
	\filldraw[fill=BirkhoffColorInt,draw=BirkhoffColorExt,ultra thick] (0,0) -- (0.1,-1.8) -- (0.4,-1.75) -- (0,0);
	\filldraw[fill=BirkhoffColorInt,draw=BirkhoffColorExt,ultra thick] (0,0) -- (0.8,-1.75);
	\filldraw[fill=BirkhoffColorInt,draw=BirkhoffColorExt,ultra thick] (0,0) -- (0.8,-1.2);
	\filldraw[fill=BirkhoffColorInt,draw=BirkhoffColorExt,ultra thick] (0,0) -- (1,-1.2);
	\filldraw[fill=BirkhoffColorInt,draw=BirkhoffColorExt,ultra thick] (0,0) -- (0.7,-0.1);
	\filldraw[fill=BirkhoffColorInt,draw=BirkhoffColorExt,ultra thick] (0,0) -- (2.1,-0.8);
	\filldraw[fill=BirkhoffColorInt,draw=BirkhoffColorExt,ultra thick] (0,0) -- (1,-0.8);
	
	\draw[fill,color=red] (2,0) circle [radius=0.05] node[left]{};
	\draw[fill,color=red] (1.5,0.5) circle [radius=0.05] node[left]{};
	\draw[fill,color=red] (2.4,1) circle [radius=0.05] node[left]{};
	\draw[fill,color=red] (2,1) circle [radius=0.05] node[left]{};
	\draw[fill,color=red] (1.8,1.3) circle [radius=0.05] node[left]{};
	\draw[fill,color=red] (0.8,1.2) circle [radius=0.05] node[left]{};
	\draw[fill,color=red] (2,2) circle [radius=0.05] node[left]{};
	\draw[fill,color=red] (0.8,1.8) circle [radius=0.05] node[left]{};
	\draw[fill,color=red] (0.3,1.4) circle [radius=0.05] node[left]{};
	\draw[fill,color=red] (-0.3,1.4) circle [radius=0.05] node[left]{};
	\draw[fill,color=red] (-0.7,1.8) circle [radius=0.05] node[left]{};
	\draw[fill,color=red] (-0.4,0.6) circle [radius=0.05] node[left]{};
	\draw[fill,color=red] (-0.7,0.6) circle [radius=0.05] node[left]{};
	\draw[fill,color=red] (-2.1,0.9) circle [radius=0.05] node[left]{};
	\draw[fill,color=red] (-0.9,0.2) circle [radius=0.05] node[left]{};
	\draw[fill,color=red] (-1.8,-0.2) circle [radius=0.05] node[left]{};
	\draw[fill,color=red] (-0.9,-0.4) circle [radius=0.05] node[left]{};
	\draw[fill,color=red] (-2,-1.2) circle [radius=0.05] node[left]{};
	\draw[fill,color=red] (-0.2,-0.5) circle [radius=0.05] node[left]{};
	\draw[fill,color=red] (-0.2,-1.5) circle [radius=0.05] node[left]{};
	\draw[fill,color=red] (0.1,-1.8) circle [radius=0.05] node[left]{};
	\draw[fill,color=red] (0.4,-1.75) circle [radius=0.05] node[left]{};
	\draw[fill,color=red] (0.8,-1.75) circle [radius=0.05] node[left]{};
	\draw[fill,color=red] (0.8,-1.2) circle [radius=0.05] node[left]{};
	\draw[fill,color=red] (1,-1.2) circle [radius=0.05] node[left]{};
	\draw[fill,color=red] (0.7,-0.1) circle [radius=0.05] node[left]{};
	\draw[fill,color=red] (2.1,-0.8) circle [radius=0.05] node[left]{};
	\draw[fill,color=red] (1,-0.8) circle [radius=0.05] node[left]{};
	
	\node[blue] at (1.3,2) {$\vbirkhoff$};
	\node[red] at (-1,1) {$\vperm$};
	
	\draw[fill,color=PermsColor] (0.03,0) circle [radius=0.1] node[left]{};
	\end{tikzpicture}
	\label{fig:virtual-Birkhoff-space}
	\caption{A cartoon of the virtual Birkhoff polytope $\vbirkhoff$ (blue) and the virtual permutation matrices $\vperm$ within it (red).}
\end{figure}

The remainder of this paper is structured as follows.
In the last part of this section, we outline some important notational conventions, some topological preliminaries, and some remarks about the categorical aspects of the constructions herein.
In Section~\ref{sec:basic} we give the basic definitions and properties of our theory, and we outline several classes of examples of VMCs which can be built from sufficiently regular continuous-time Markov processes.
In Section~\ref{sec:canonical-data} we state and prove the fundamental representation theorem which gives the correspondence between VMCs and compatible pairs of VIDs and VTMs.
In Section~\ref{sec:convexity}, we study a few problems of convexity in infinite-dimensional vector spaces that naturally arise in the theory we have developed.
Many examples are given in every section.

\subsection*{Notation}

We use calligraphic characters to denote spaces of objects, and we use boldface characters to denote virtual objects; in \LaTeX, these correspond to the \texttt{mathcal} and \texttt{mathbf} environments, respectively.
Correspondingly we use calligraphic boldface characters to denote spaces of virtual objects.
For example, $\perm$ denotes the space of all permutations of $\N$, a generic element of which is denoted $\sigma$; likewise, $\vperm$ denotes the space of all virtual permutations, a generic element of which is denoted $\boldsymbol{\sigma}$.

We also overload notation and use $\{P_N\}_{N\in\N}$ to denote several different families of projections operations and $\iota$ to denote several different natural injections.
This overloading should cause no confusion since the argument always disambiguates which map is which.

\subsection*{Topology}
Unless otherwise stated, all subsets of topological spaces are endowed with the relative topology, all cartesian products of topological spaces are endowed with the product topology, and all spaces of probability measures on a topological space are endowed with the topology of weak convergence.

\subsection*{Category}
While we will not make the categorical aspects of our work rigorous in this paper, we make a few remarks for the interested reader.

In the case of \cite{VirtualPermutations}, one works in the category of topological groups and considers the sequence $\{\perm_N\}_{N\in\N}$ of permutation groups.
Then, the infinite symmetric group $S_{<\infty} = \bigcup_{N\in\N}$ is an inductive limit of this sequence when endowed with the natural collection of injective homomorphisms $\iota_{NM}:\perm_N\to \perm_M$ which, for each $N,M\in\N$ with $N\le M$, send each permutation on $\{1,2,\ldots N\}$ to the unique permutation on $\{1,2,\ldots M\}$ which has the same cycle structure but with each element of $\{N+1,N+2,\ldots M-1,M\}$ placed into its own cycle.
Dually, the space of virtual permutations $\vperm$ is a projective limit of this sequence when endowed with the natural collection of surjective homomorphisms $P_{NM}:\perm_M\to \perm_N$ which, for each $N,M\in\N$ with $N\le M$, send each permutation on $\{1,2,\ldots M\}$ to the permutation on $\{1,2,\ldots N\}$ which has the same cycle structure but with each element of $\{N+1,N+2,\ldots M-1,M\}$ removed and with the gaps ``stitched up''.

In the present paper, one works in a certain ``category of Markov chains'', which, to the best of our knowledge, has not been studied.
Roughly speaking, the objects in this category are the Markov chains defined on a common probability space $(\Omega,\F,\P)$, and the morphisms in this category are sample-path transformations that preserve the Markov property.
For each $N\in\N$, write $\mathcal{M}_N$ for the collection of all Markov chains on $\{1,2,\ldots, N\}$ defined on some fixed $(\Omega,\F,\P)$, and consider the sequence $\{\mathcal{M}_N\}_{N\in\N}$.
Then, one can form an inductive limit of this sequence when it is endowed with the natural collection of injective morphisms $\iota_{NM}:\mathcal{M}_{N}\to \mathcal{M}_M$ which, for each $N,M\in\N$ with $N\le M$, send a sample path on $\{1,2,\ldots N\}$ to the same sample path but viewed as a path on $\{1,2,\ldots M\}$;
the resulting objects is the collection all the finite state space Markov chains on $(\Omega,\F,\P)$.
Dually, one can form a projective limit when endowing this sequence with the natural collection of surjective morphisms $P_{NM}:\perm_M\to \perm_N$ which, for each $N,M\in\N$ with $N\le M$, send each sample path on $\{1,2,\ldots M\}$ to the sample path on $\{1,2,\ldots N\}$ which results from removing all instances of $\{N+1,N+2,\ldots M-1,M\}$ and with the resulting gaps ``stitched up''.
(Note that there is some care to be taken in the case that the sample path eventually leaves $\{1,2,\ldots,N\}$ and never returns.)
The resulting object is the collection of all virtual Markov chains.

\section{Basic Theory}\label{sec:basic}

In this section we deveop the basic theory of the main objects of interest in the paper.
In Subsection~\ref{subsec:virtual-paths} we define the notion of \textit{virtual path space}, and in Subsection~\ref{subsec:VMC} we define \textit{virtual Markov chains (VMCs)} as nothing other than probability measures on virtual path space with some nice properties.
Then in Subsection~\ref{subsec:Ex-1} we build various examples of VMCs, mostly from nice continuous-time Markov processes, and we explore some different properties.

\subsection{Virtual Path Space}\label{subsec:virtual-paths}

Set $\N = \{1,2,\ldots\}$ and $\N_0 = \N\cup\{0\}$.
For $a,b\in \N_0$, write $\llbracket a,b\rrbracket = \{a,a+1,\ldots, b-1,b\}$, which may be empty.
\begin{definition}
The set
\begin{equation*}
\chain:=\left\{x\in \N_0^{\N_0}: \begin{matrix}
\text{ if } x(i) = 0 \text{ for some } i\in\N_0, \\
\text{then } x(j) = 0 \text{ for } j\in \llbracket i,\infty \llbracket
\end{matrix}\right\}
\end{equation*}
is called the space of \textit{paths}.
\end{definition}

In the above, the elements of $\N$ serve as possible states of paths, and 0 serves as a distinguished state, called the \textit{cemetery}, since any chain which visits 0 gets trapped at 0 forever after.
(Customary symbols for a cemetery state are $\partial,\dagger$, or $\Delta$, but we choose to use 0 since it fits in well with the ordering on $\N_0$.)
For $N\in\N$, we also define an analogous space of paths on the state space $\llbracket 0,N\rrbracket$, via
\begin{equation*}
\chain_N :=\left\{x\in \llbracket 0,N\rrbracket^{\N_0}: \begin{matrix}
\text{ if } x(i) = 0 \text{ for some } i\in\N_0, \\
\text{then } x(j) = 0 \text{ for } j\in \llbracket i,\infty \llbracket
\end{matrix}\right\}.
\end{equation*}
Observe in particular that for $M,N\in\N$ with $M\ge N$ we have $\chain_N\subseteq \chain_M \subseteq \chain$.
We endow $\chain$ with the relative topology induced by the product topology on $\N_0^{\N_0}$, which is metrizable; thus, compactness is equivalent to sequential compactness, continuity is equivalent to sequential continuity, etc.

\begin{lemma}
The space $\chain$ is a Polish space, and the space $\chain_{N}$ is a compact Polish space for each $N\in\N$.
\end{lemma}

\begin{proof}
Since $\N_0^{\N_0}$ is Polish, the first claim follows if we can show that $\chain$ is closed in $\N_0^{\N_0}$.
This is immediate by writing
\begin{equation*}
\chain = \bigcap_{i\in\N_0}\left(\left\{x\in\N_0^{\N_0}: x(i) \neq 0 \right\}\cup \bigcap_{j=i}^{\infty}\left\{x\in\N_0^{\N_0}: x(j) = 0 \right\}\right)
\end{equation*}
and noting that $\{x\in\N_0^{\N_0}: x(k) = 0 \}$ is clopen in $\N_0^{\N_0}$ for all $k\in\N_0$.

Now let $N\in\N$ be arbitrary.
Since $\llbracket 0,N\rrbracket^{\N_0}$ is a compact Polish space, the second claim follows if we can show that $\chain_{N}$ is closed in $\llbracket 0,N\rrbracket^{\N_0}$, and this follows by the same argument as above.
\end{proof}

For $x\in \chain$ and $N\in\N$, set $I_{x,N}(0) = \inf\{i\in\N_0: x(i)\in\llbracket 0,N\rrbracket \}$ and recursively $I_{x,N}(j+1) = \inf\{i\in \N_0: i> I_{x,N}(j), x(i)\in\llbracket 0,N\rrbracket \}$ for $j\in\N$; we use the standard convention that $\inf\emptyset = \infty$.
Observe that $\{I_{x,N}(j)\}_{j\in\N}$ enumerates the indices at which the path $x$ visits states in the set $\llbracket 0,N\rrbracket$.
In particular, $\{I_{x,N}(j)\}_{j\in\N}$ is non-decreasing and is possibly eventually equal to infinity.

Then, for $x\in\chain$ and $N\in\N$, we define $P_N(x)\in\chain_N$ via
\begin{equation*}
(P_N(x))(j) = \begin{cases}
x(I_{x,N}(j)), &\text{ if } I_{x,N}(j) < \infty, \\
0, &\text{ if } I_{x,N}(j) = \infty.
\end{cases}
\end{equation*}
Since $\chain_M\subseteq \chain$ for all $M\in\N$, we can also think of the map $P_N:\chain_M\to \chain_N$ for any $M,N\in\N$.
Intuitively, $x\mapsto P_N(x)$ is the operation which removes from $x$ all of its excursions that visit states higher than $N$, including a possible final excursion of infinite length; in the case that $x$ has an infinite excursion, we pad the path $P_N(x)$ with $0$s.
Also note that, for all $N\in\N$, we have the fundamental projectivity property $P_{N}\circ P_{N+1} = P_N$ on $\chain$.

\begin{lemma}\label{lem:proj-cts}
	For each $N\in\N$, the map $P_N:\chain\to\chain_N$ is continuous.
\end{lemma}

\begin{proof}
	An arbitrary cylinder set $C\subseteq\chain_N$ is of the form
	\begin{equation*}
	C = \{x\in \chain_N: x(i_1) \in B_1,\ldots, x(i_{\ell}) \in B_{\ell} \}
	\end{equation*}
	for some $\ell\in\N_0$, distinct $i_1,\ldots, i_{\ell}\in \N_0$, and $B_1,\ldots, B_{\ell}\subseteq \llbracket 0,N\rrbracket$.
	So, simply set $I = \max\{i_1,\ldots, i_{\ell}\}$ and define
	\begin{equation*}
	C' = \left\{x\in \chain: \begin{matrix}
	x(i) \in B_i \text{ for all } i\in\{i_1,\ldots, i_{\ell}\}, \text{ and } \\
	x(i)\in \llbracket 0,N\rrbracket \text{ for all } i\in \llbracket 0,I\rrbracket\setminus\{i_1,\ldots, i_{\ell}\}
	\end{matrix}\right\}.
	\end{equation*}
	Notice that $C'$ is a cylinder set in $\chain$ and that $x\in C'$ clearly implies $P_N(x)\in C$, hence $P_N$ is continuous.
\end{proof}

\begin{lemma}\label{lem:proj-approximate-identity}
	For each $x\in\chain$, we have $P_N(x)\to x$ as $N\to\infty$.
\end{lemma}

\begin{proof}
	An arbitrary cylinder set $C\subseteq\chain$ is of the form
	\begin{equation*}
	C = \{x\in \chain: x(i_1) \in B_1,\ldots, x(i_{\ell}) \in B_{\ell} \}
	\end{equation*}
	for some $\ell\in\N_0$, distinct $i_1,\ldots, i_{\ell}\in \N_0$, and $B_1,\ldots, B_{\ell}\subseteq \N_0$.
	Then set $I = \max\{i_1,\ldots, i_{\ell}\}$ and $M = \max\{x(k): 0\le k \le I\}$.
	Thus, if $x\in C$, then $N\ge M$ implies that the restrictions of the infinite sequences $P_N(x)$ and $x$ to the interval of indices $\llbracket 0,I \rrbracket$ must agree, hence $P_N(x)\in C$, as needed.
\end{proof}

\begin{definition} 
The set
\begin{equation*}
\vchain := \left\{\{x_N\}_{N\in\N} \in \prod_{N\in\N}\chain_N: P_{N}(x_{N+1}) = x_N\text{ for all } N\in\N \right\}
\end{equation*}
is called the space of \textit{virtual paths}.
\end{definition}

\begin{lemma}
The space $\vchain$ is a compact Polish space.
\end{lemma}

\begin{proof}
Since each $\chain_N$ is a compact Polish space, the countable product $\prod_{N\in\N}\chain_N$ is a compact Polish space.
Thus, it suffices to show that $\vchain$ is closed in $\prod_{N\in\N}\chain_N$.
To see this, just write
\begin{equation*}
\vchain = \bigcap_{N\in\N}\left\{\{x_N\}_{N\in\N} \in \prod_{N\in\N}\chain_N: P_{N}(x_{N+1}) = x_N\right\}
\end{equation*}
and note that each of these sets in the countable intersection is a closed subset of $\prod_{N\in\N}\chain_N$ by virtue of Lemma~\ref{lem:proj-cts}.
\end{proof}

\begin{lemma}\label{lem:C-VC-inclusion}
The map $\iota:\chain\to\vchain$ defined in the natural way via $\iota(x) = \{P_N(x)\}_{N\in\N}$ is well-defined, continuous, and injective.
\end{lemma}

\begin{proof}
That this map is well-defined follows from the projectivity property of the projections $\{P_N\}_{N\in\N}$, and that it is continuous follows from Lemma~\ref{lem:proj-cts}.
To see that it is injective, suppose that $x,y\in\chain$ have $\iota(x) = \iota(y)$.
Then for each $I\in\N_0$, set $N = \max\{\max\{x(k),y(k): 0\le k \le I\} \}$.
The restrictions of $P_N(x)$ and $x$ to $\llbracket 0,I\rrbracket$ must agree, and the restrictions of $P_N(y)$ and $y$ to the indices $\llbracket 0,I\rrbracket$ must also agree.
Thus $P_N(x) = P_N(y)$ implies that the restrictions of $x$ and $y$ to $\llbracket 0,I\rrbracket$ must agree.
As $I\in\N_0$ was arbitrary, this implies $x=y$.
\end{proof}

\begin{lemma}
The space $\iota(\chain)$ is dense in $\vchain$.
\end{lemma}

\begin{proof}
Note that a basis for the topology of $\vchain$ is given by all sets of the form
\begin{equation*}
\{\{x_N\}_{N\in\N}\in \vchain: x_{N_1}\in B_{1},\ldots, x_{N_\ell}\in B_{\ell} \}
\end{equation*}
for some $\ell\in\N_0$, an increasing sequence $N_1,\ldots N_{\ell}\in\N$, and some non-empty open sets $B_{i}\subseteq \chain_{N_i}$ for each $i=1,\ldots \ell$.
For any such set $U$ which is non-empty, there is some $\{x_N\}_{N\in\N}\in U$, and it follows that the element $x_{N_{\ell}}\in\chain_{N_{\ell}}\subseteq \chain$ has $\iota(x_{N_{\ell}})\in U$.
Thus, $\iota(\chain)$ intersects every non-empty open set in $\vchain$, so the result follows.
\end{proof}

In the previous two results we showed that $\vchain$ contains a ``copy'' of $\chain$, and that this copy is dense in $\vchain$.
Since we also showed that $\vchain$ is compact, it is useful to regard $\vchain$ as a certain kind of compactification of the space $\chain$.
Our final result completes the picture by giving a simple characterization of the elements of $\vchain\setminus \iota(\chain)$.

\begin{lemma}\label{lem:VC-C-characterization}
	A virtual path $\mathbf{x} = \{x_N\}_{N\in\N}\in\vchain$ lies in $\iota(\chain)$ if and only if $\lim_{N\to\infty}x_N$ exists in $\chain$ and $\mathbf{x} = \iota(\lim_{N\to\infty}x_N)$.
\end{lemma}

\begin{proof}
	If $\lim_{N\to\infty}x_N$ exists and $\mathbf{x} = \iota(\lim_{N\to\infty}x_N)$, then $\mathbf{x}\in\iota(\chain)$ trivially.
	Conversely, if we have $\mathbf{x} = \iota(x)$ for some $x\in\chain$, then the result follows from Lemma~\ref{lem:proj-approximate-identity}.
\end{proof}

The preceding result states that $\lim_{N\to\infty}x_N$ is the only candidate for the preimage of $\mathbf{x} = \{x_N\}_{N\in\N}$ by $\iota$.
Thus, a virtual path fails to be identified with an actual path whenever either this candidate $x$ does not exist, or it exists but it does not satisfy the identity $\mathbf{x}= \iota(x)$.

Now, let's see some simple examples of virtual paths.
Our first three examples illustrate the ways in the hypotheses of Lemma~\ref{lem:VC-C-characterization} come into effect, and hence highlight the relationship between $\vchain$ and $\chain$.

\begin{example}\label{ex:virtual-path-1}
	For each $N\in\N$, define the path $x_N\in\chain_N$ via
	\begin{equation*}
	x_N(i) = \begin{cases}
	i + 1, & \text{ if } i < N, \\
	0, & \text{ if } i\ge N, \\
	\end{cases}
	\end{equation*}
	for $i\in\N_0$; see Figure~\ref{fig:virtual-paths} (top) for a depiction.	
	Then, $\mathbf{x} = \{x_N\}_{N\in\N}$ is a virtual path.
	Also define the path $x\in \chain$ via $x(i) = i+1$ for $i\in\N_0$, and note that we have $\mathbf{x} = \iota(x)$.
\end{example}

\begin{example}\label{ex:virtual-path-2}
	For each $N\in\N$, define the path $x_N\in\chain_N$ via
	\begin{equation*}
	x_N(i) = \begin{cases}
	i + 1, & \text{ if } i < N, \\
	2N - i - 1, & \text{ if } N \le i < 2N - 1, \\
	1, & \text{ if } i \ge 2N - 1, \\
	\end{cases}
	\end{equation*}
	for $i\in\N_0$, and see Figure~\ref{fig:virtual-paths} (middle) for an illustration.
	Then, $\mathbf{x} = \{x_N\}_{N\in\N}$ is a virtual path.
	Intuitively, $\mathbf{x}$ is the virtual path which ``goes to infinity and back'' and then stays forever at the state 1.
	Observe that we have $\lim_{N\to\infty}x_N = x$, where $x\in \chain$ is defined in the previous example, but that $\iota(x) \neq \mathbf{x}$.
	Thus, by Lemma~\ref{lem:VC-C-characterization}, there is no path which corresponds to this virtual path.
\end{example}

\begin{example}\label{ex:virtual-path-3}
	For each $N\in\N$, define the path $x_N \in \chain_N$ via
	\begin{equation*}
	x_N(i) = \begin{cases}
	N - i, & \text{ if } i < N, \\
	1, & \text{ if } i \ge N, \\
	\end{cases}
	\end{equation*}
	for $i\in\N_0$, and see Figure~\ref{fig:virtual-paths} (middle).
	Then, $\mathbf{x} = \{x_N\}_{N\in\N}$ is a virtual path.
	Intuitively, $\mathbf{x}$ is the virtual path which ``comes down from infinity'' and then stays forever at the state 1.
	Observe that $x_N(0) = N$, so $\lim_{N\to\infty}x_N$ cannot exist in $\chain$.
	Thus, by Lemma~\ref{lem:VC-C-characterization}, we see that there is no path which corresponds to this virtual path.
\end{example}

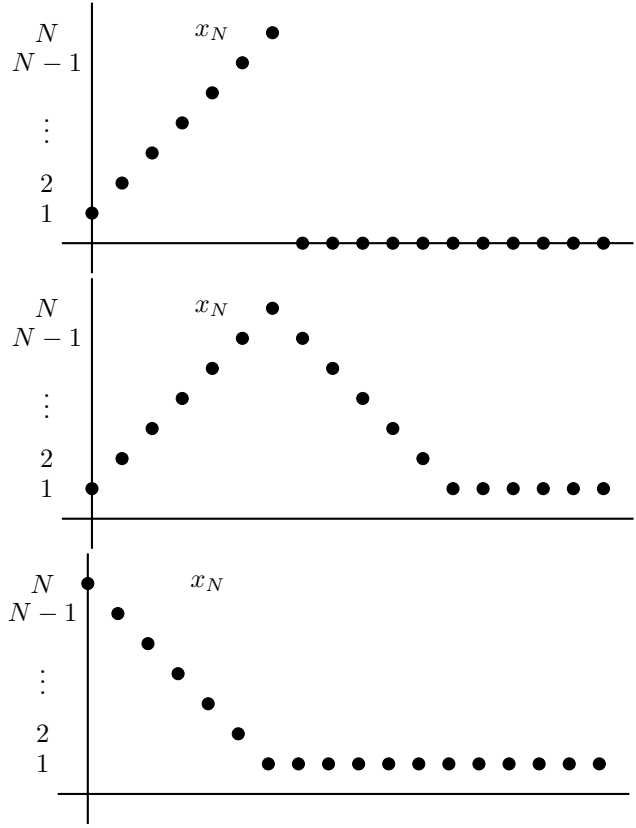
\begin{figure}
	\begin{tikzpicture}[scale=0.4]
	\draw[-, thick] (0, -1) to (0,8);
	\draw[-, thick] (-1,0) to (18,0);
	
	\draw [fill] (0,1) circle [radius=.2] node[left]{};
	\draw [fill] (1,2) circle [radius=.2] node[left]{};
	\draw [fill] (2,3) circle [radius=.2] node[left]{};
	\draw [fill] (3,4) circle [radius=.2] node[left]{};
	\draw [fill] (4,5) circle [radius=.2] node[left]{};
	\draw [fill] (5,6) circle [radius=.2] node[left]{};
	\draw [fill] (6,7) circle [radius=.2] node[left]{};
	\draw [fill] (7,0) circle [radius=.2] node[left]{};
	\draw [fill] (8,0) circle [radius=.2] node[left]{};
	\draw [fill] (9,0) circle [radius=.2] node[left]{};
	\draw [fill] (10,0) circle [radius=.2] node[left]{};
	\draw [fill] (11,0) circle [radius=.2] node[left]{};
	\draw [fill] (12,0) circle [radius=.2] node[left]{};
	\draw [fill] (13,0) circle [radius=.2] node[left]{};
	\draw [fill] (14,0) circle [radius=.2] node[left]{};
	\draw [fill] (15,0) circle [radius=.2] node[left]{};
	\draw [fill] (16,0) circle [radius=.2] node[left]{};
	\draw [fill] (17,0) circle [radius=.2] node[left]{};
	
	\draw (-1.5,1) node[]{$1$};
	\draw (-1.5,2) node[]{$2$};
	\draw (-1.5,4) node[]{$\vdots$};
	\draw (-1.5,6) node[]{$N-1$};
	\draw (-1.5,7) node[]{$N$};
	
	\draw (4,7) node[]{$x_N$};
	\end{tikzpicture}
	\begin{tikzpicture}[scale=0.4]
	\draw[-, thick] (0, -1) to (0,8);
	\draw[-, thick] (-1,0) to (18,0);
	
	\draw [fill] (0,1) circle [radius=.2] node[left]{};
	\draw [fill] (1,2) circle [radius=.2] node[left]{};
	\draw [fill] (2,3) circle [radius=.2] node[left]{};
	\draw [fill] (3,4) circle [radius=.2] node[left]{};
	\draw [fill] (4,5) circle [radius=.2] node[left]{};
	\draw [fill] (5,6) circle [radius=.2] node[left]{};
	\draw [fill] (6,7) circle [radius=.2] node[left]{};
	\draw [fill] (7,6) circle [radius=.2] node[left]{};
	\draw [fill] (8,5) circle [radius=.2] node[left]{};
	\draw [fill] (9,4) circle [radius=.2] node[left]{};
	\draw [fill] (10,3) circle [radius=.2] node[left]{};
	\draw [fill] (11,2) circle [radius=.2] node[left]{};
	\draw [fill] (12,1) circle [radius=.2] node[left]{};
	\draw [fill] (13,1) circle [radius=.2] node[left]{};
	\draw [fill] (14,1) circle [radius=.2] node[left]{};
	\draw [fill] (15,1) circle [radius=.2] node[left]{};
	\draw [fill] (16,1) circle [radius=.2] node[left]{};
	\draw [fill] (17,1) circle [radius=.2] node[left]{};
	
	\draw (-1.5,1) node[]{$1$};
	\draw (-1.5,2) node[]{$2$};
	\draw (-1.5,4) node[]{$\vdots$};
	\draw (-1.5,6) node[]{$N-1$};
	\draw (-1.5,7) node[]{$N$};
	
	\draw (4,7) node[]{$x_N$};
	\end{tikzpicture}
	\begin{tikzpicture}[scale=0.4]
	\draw[-, thick] (0, -1) to (0,8);
	\draw[-, thick] (-1,0) to (18,0);
	
	\draw [fill] (0,7) circle [radius=.2] node[left]{};
	\draw [fill] (1,6) circle [radius=.2] node[left]{};
	\draw [fill] (2,5) circle [radius=.2] node[left]{};
	\draw [fill] (3,4) circle [radius=.2] node[left]{};
	
	\draw [fill] (4,3) circle [radius=.2] node[left]{};
	\draw [fill] (5,2) circle [radius=.2] node[left]{};
	\draw [fill] (6,1) circle [radius=.2] node[left]{};
	\draw [fill] (7,1) circle [radius=.2] node[left]{};
	\draw [fill] (8,1) circle [radius=.2] node[left]{};
	\draw [fill] (9,1) circle [radius=.2] node[left]{};
	\draw [fill] (10,1) circle [radius=.2] node[left]{};
	\draw [fill] (11,1) circle [radius=.2] node[left]{};
	\draw [fill] (12,1) circle [radius=.2] node[left]{};
	\draw [fill] (13,1) circle [radius=.2] node[left]{};
	\draw [fill] (14,1) circle [radius=.2] node[left]{};
	\draw [fill] (15,1) circle [radius=.2] node[left]{};
	\draw [fill] (16,1) circle [radius=.2] node[left]{};
	\draw [fill] (17,1) circle [radius=.2] node[left]{};
	
	\draw (-1.5,1) node[]{$1$};
	\draw (-1.5,2) node[]{$2$};
	\draw (-1.5,4) node[]{$\vdots$};
	\draw (-1.5,6) node[]{$N-1$};
	\draw (-1.5,7) node[]{$N$};
	
	\draw (4,7) node[]{$x_N$};
	\end{tikzpicture}
	\label{fig:virtual-paths}
	\caption{The virtual paths of Example~\ref{ex:virtual-path-1} (top), Example~\ref{ex:virtual-path-2} (middle), and Example~\ref{ex:virtual-path-3} (bottom).}
\end{figure}

For our last class of examples, we connect the ideas herein to the concept of virtual permutations, as constructed in \cite{VirtualPermutations}.
Recall that a \textit{virtual permutation} is a collection $\boldsymbol{\sigma} = \{\sigma_N\}_{N\in\N}\in\vperm$ where $\sigma_N$ is a permutation of $\llbracket 1,N \rrbracket$ for each $N\in\N$, such that the cycle structure of $\sigma_{N}$ is given by removing the element $N+1$ from the cycle structure of $\sigma_N$.

\begin{example}\label{ex:virtual-permutations}
Let $\boldsymbol{\sigma} =\{\sigma_{N}\}_{N\in\N}$ be a virtual permutation, and let $a\in\N$ be arbitrary.
Then define $\mathbf{x}_{\boldsymbol{\sigma},a} := \{x_N\}_{N\in\N}$ in $\vchain$ as follows:
For $N\ge a$, set $x_{N}(0) := a$ and recursively $x_{N}(j+1) := \sigma_N(x_N(j))$ for $j\in\N_0$.
For $N<a$, recursively set $x_{N} := P_{N}(x_{N+1})$.
Then $\mathbf{x}_{\boldsymbol{\sigma},a}$ represents the ``cycle'' of $\boldsymbol{\sigma}$ containing $a$.
Observe by Lemma~\ref{lem:VC-C-characterization} that $\mathbf{x}_{\boldsymbol{\sigma},a}\in\iota(\chain)$ iff this cycle is finite.
\end{example}

\subsection{Virtual Markov Chains}\label{subsec:VMC}

Analogously to how a stochastic process can be identified with its law on a suitable space of paths, we propose that a probability measure on $\vchain$ should be regarded as a \textit{virtual stochastic process}.
Our main object of study is the collection of virtual stochastic processes which are analogous to the case of Markov chains (MCs) in the classical sense.

For a topological space $S$, write $\mathcal{B}(S)$ for the Borel $\sigma$-algebra on $S$; whenever we view a topological space as a measurable space, we assume that it is endowed with its Borel $\sigma$-algebra, unless stated otherwise.
We also write $\mathcal{M}(S)$ for the space of signed Borel measures on $S$, and $\mathcal{M}_1(S)$ for the space of non-negative Borel probability measures on $S$.
Let us say that $X$ \textit{is a Markov chain in $\N$ with cemetery 0} if $X$ is a Markov chain in $\N_0$ such that 0 is an absorbing state.
Similarly, for any $N\in\N$, let us say that $X$ \textit{is a Markov chain in $\llbracket 1,N\rrbracket$ with cemetery 0} if $X$ is a Markov chain in $\llbracket 0,N\rrbracket$ such that 0 is an absorbing state.

\begin{definition}
A $\vchain$-valued random variable $\mathbf{X} = \{X_N\}_{N\in\N}$ is called a \textit{virtual Markov chain (VMC)} if, for each $N\in\N$, the random variable $X_N$ is a Markov chain in $\llbracket 1,N\rrbracket$ with cemetery 0.
\end{definition}

\begin{remark}\label{rem:VMC-name}
It would be slightly more precise to call such objects \textit{Markovian virtual paths}, but we stick with the slightly imprecise name because it sounds more natural.
\end{remark}

Before we develop further properties of VMCs, we should see that they indeed generalize classical MCs.
Towards this end, the following is an extension of Lemma~\ref{lem:VC-C-characterization} to the Markovian setting.

\begin{proposition}\label{prop:VMC-MC-characterization}
If $X$ is a MC in $\N$ with cemetery 0, then $\iota(X)$ is a VMC.
Conversely, if a VMC $\mathbf{X}$ satisfies $\mathbf{X}\in\iota(\chain)$ almost surely, then there is a Markov chain $X$ in $\N$ with cemetery 0 such that $\mathbf{X} = \iota(X)$ almost surely.
\end{proposition}

\begin{proof}
Suppose $X$ is a Markov chain in $\N$ with cemetery 0 on a probability space $(\Omega,\F,\P)$.
Then write $\iota(X) = \{X_N\}_{N\in\N}$, so that $X_N = P_N(X)$ for all $N\in\N$, and note that we only need to check that $X_N$ is a Markov chain in $\llbracket 1,N\rrbracket$ with cemetery 0 for all $N\in\N$.
Indeed, take any $N\in\N,\ell\in\N_0$ and $a_{0},\ldots, a_{\ell+1}\in \llbracket 0,N\rrbracket$, and let us show
\begin{align*}
&\P(X_N(\ell+1) = a_{\ell+1}\,|\,X_N(0) = a_{0},\ldots, X_N(\ell) = a_{\ell}) \\
&= \P(X_N(\ell+1) = a_{\ell+1}\,|\,X_N(\ell) = a_{\ell}).
\end{align*}
Note that, if $a_j =0$ for some $j\in\llbracket 0,\ell\rrbracket$, then this forces $a_{\ell}= 0$, hence both sides above are equal to $\ind\{a_{\ell+1} =0\}$.
Therefore, we can assume $a_{0},\ldots, a_{\ell}\in \llbracket 1,N\rrbracket$, and we recall that $\{I_{X,N}(j)\}_{j\in\N_0}$ is a sequence of stopping times.
So, if $a_{\ell+1} = 0$, then the strong Markov property implies
\begin{align*}
&\P(X_N(\ell+1) = 0\,|\,X_N(0) = a_{0},\ldots, X_N(\ell) = a_{\ell}) \\
&=\P(I_{X,N}(\ell+1)=\infty\,|\,X(I_{X,N}(0)) = a_{0},\ldots, X(I_{X,N}(\ell)) = a_{\ell}) \\
&=\P(I_{X,N}(\ell+1)=\infty\,|\,X(I_{X,N}(\ell)) = a_{j}) \\
&= \P(X_N(\ell+1) = a_{\ell+1}\,|\,X_N(\ell) = a_{\ell}),
\end{align*}
and, if $a_{\ell+1}\neq 0$, then the strong Markov property implies
\begin{align*}
&\P(X_N(\ell+1) = a_{\ell+1}\,|\,X_N(0) = a_{0},\ldots, X_N(\ell) = a_{\ell}) \\
&=\P(X(I_{X,N}(\ell+1)) = a_{\ell+1}\,|\,X(I_{X,N}(0)) = a_{0},\ldots, X(I_{X,N}(\ell)) = a_{\ell}) \\
&=\P(X(I_{X,N}(\ell+1)) = a_{\ell+1}\,|\,X(I_{X,N}(\ell)) = a_{j}) \\
&= \P(X_N(\ell+1) = a_{\ell+1}\,|\,X_N(\ell) = a_{\ell}).
\end{align*}
This shows that $X_N$ is a Markov chain in $\llbracket 1,N\rrbracket$ with cemetery 0 for each $N\in\N$, as needed.

Conversely, suppose that $\mathbf{X} = \{X_N\}_{N\in\N}$ is a VMC on a probability space $(\Omega,\F,\P)$ with $\P(\mathbf{X}\in \iota(\chain)) =1$.
Then by Lemma~\ref{lem:VC-C-characterization}, the limit $X= \lim_{N\to\infty}X_N$ exists in $\chain$ and we have $\mathbf{X} = \iota(X)$ almost surely.
In particular, $X$ is measurable since $X_N$ is measurable for each $N\in\N$.
Now we just need to show that $X$ is a Markov chain in $\N$ with cemetery 0.
To do this, take any $\ell\in\N_0$ and $a_0,\ldots, a_{\ell+1}\in\N_0$, and note that we have, by the Markov property of $X_N$ for each $N\in\N$:
\begin{align*}
&\P(X(\ell+1) = a_{\ell+1}\, |\, X(0) = a_0,\ldots, X(\ell) = a_{\ell}) \\
&= \lim_{N\to\infty}\P(X_N(\ell+1) = a_{\ell+1}\, |\, X_N(0) = a_0,\ldots, X_N(\ell) = a_{\ell}) \\
&= \lim_{N\to\infty}\P(X_N(\ell+1) = a_{\ell+1}\, |\, X_N(\ell) = a_{\ell}) \\
&= \P(X(\ell+1) = a_{\ell+1}\, |\, X(\ell) = a_{\ell}).
\end{align*}
This shows that $X$ is a Markov chain in $\N_0$, as needed.
To see that 0 is a cemetery for $X$, let $\ell\in\N_0$ and $a\in\N_0$ be arbitrary, and note that, since 0 is cemetery for $X_N$ for each $N\in\N$, we have:
\begin{align*}
&\P(X(\ell+1) = a\,|\,X(\ell) = 0) \\
&= \lim_{N\to\infty}\P(X_N(\ell+1) = a\, |\, X_N(\ell) = 0) \\
&= \ind\{a=0\}.
\end{align*}
This shows that $X$ is a Markov chain in $\N$ with cemetery 0, as needed.
\end{proof}

\subsection{Examples}\label{subsec:Ex-1}

Before we develop any further theory of VMCs, it will be instructive to build a wide collection of examples.
We have already seen in Example~\ref{ex:virtual-permutations} that virtual permutations give rise to VMCs, although they provide no interesting behavior probabilistically.
We have also seen in Proposition~\ref{prop:VMC-MC-characterization} that classical MCs give rise to VMCs, although they provide no interesting behavior ``at infinity''.
Our goal is thus to find examples of VMCs that are non-trivial in both of these senses.

The main tool for doing this will be a general result which allows us to construct a large class of VMCs from various continuous-time stochastic processes.
Before doing so, let us recall some definitions from \cite{Sharpe}.
When we say that
\begin{equation*}
\mathcal{Y} = (\Omega,\F,\{\F_t\}_{t\ge0},\{Y_t\}_{t\ge 0},\{\theta_t\}_{t\ge 0},\{\P_x\}_{x\in S})
\end{equation*}
is a \textit{Markov process} in a topological space $S$, we mean the following:
\begin{itemize}
\item $(\Omega,\F)$ is a measurable space,
\item $\{\F_t\}_{t\ge0}$ is a filtration of $\F$,
\item $\{Y_t\}_{t\ge 0}$ is a collection of $S$-valued random variables adapted to $\{\F_t\}_{t\ge0}$, where $S$ is endowed with its Borel $\sigma$-algebra, $\mathcal{B}(S)$,
\item $\{\theta_t\}_{t\ge0}$ is a semigroup of measurable maps from $\Omega$ to itself satisfying $Y_s\circ\theta_t = Y_{t+s}$ for all $s,t\ge 0$,
\item $\{\P_x\}_{x\in S}$ is a collection of probability measures on $S$ such that the map $P(t,x,\Gamma) = \P_x(Y_t\in\Gamma)$ is a measurable function of $x\in S$ for each $t\ge 0$ and $\Gamma\in\mathcal{S}$, and
\item for all $t,s\ge 0,\Gamma\in\mathcal{S}$, and $x\in S$, we have $\P_x(Y_{t+s}\in\Gamma\, |\, \F_t) = P(s,Y_{t},\Gamma)$ holding $\P_x$-almost surely.
\end{itemize}
In this setting, for any $\mu\in\mathcal{M}_1(S)$, one defines the probability meaure $\P_{\mu}$ on $S$ via $\P_{\mu}(A) = \int_{x\in S}\P_x(A)\, d\mu(x)$ for all $A\in\mathcal{B}(S)$.
If $S$ is Radon, then we say that $\mathcal{Y}$ is a \textit{right process} if we additionally have the following:
\begin{itemize}
\item $\{\F_t\}_{t\ge 0}$ is right-continuous, and we have $\F_t = \bigcap_{\mu\in\mathcal{M}_1(S)}(\F_t\vee \mathcal{N}^{\mu}(\F))$ for all $t\ge 0$, where $\mathcal{N}^{\mu}(\F)$ is the collection of all subsets of $\P_{\mu}$-null $\F$-measurable sets,
\item the sample path $\{Y_t\}_{t\ge0}$ is right-continuous $\P_x$-almost surely for all $x\in S$,
\item $P(0,x,\{x\}) = 1$ holds for all $x\in S$, and
\item two techinical conditions, usually denoted (HD1) and (HD2) which we will not describe, are satisfied.
\end{itemize}
See \cite[Chapter~1, Section~8]{Sharpe} for a precise description of (HD1) and (HD2), and for an authoritative account of the general theory of right processes.
In particular, we note that right processes have the strong Markov property in that, if $\tau$ is any $\{\F_t\}_{t\ge0}$-stopping time, then for any $s\ge 0,\Gamma\in\mathcal{S}$, and $x\in S$, we have $\P_x(Y_{\tau+s}\in\Gamma\, |\, \F_\tau) = P(s,Y_{\tau},\Gamma)$ holding $\P_x$-almost surely.
Moreover, the hitting times of right processes into Borel sets are always stopping times. We say that a point $x\in S$ is \textit{irregular for itself with respect to $\mathcal{Y}$} if we have $\P_x(\inf\{t > 0: Y_t = x\} > 0) =1$.

Usually we will be somewhat cavalier and refer to $Y = \{Y_t\}_{t\ge0}$ or $(\Omega,\F,\{Y_t\}_{t\ge0},\{\P_x\}_{x\in S})$ as the Markov process itself, even though the Markov process really refers to the entire ensemble of objects $\mathcal{Y}$.
For the following result, however, we stick with the full rigorous detail:

\begin{lemma}\label{lem:construct-VMC-from-MP}
	Let $\mathcal{Y} = (\Omega,\F,\{\F_t\}_{t\ge0},\{Y_t\}_{t\ge 0},\{\theta_t\}_{t\ge 0},\{\P_x\}_{x\in S})$ be a right process in a Radon space $S$, and let $\{L_N\}_{N\in\N}$ be a sequence of distinct elements of $S$ which are irregular for themsevles with respect to $\mathcal{Y}$.
	Then for $N\in\N$ define the stopping times $\tau^N_{0} := \inf\{t \ge 0: Y_t\in\{L_1,\ldots L_N\} \}$ and $\tau^N_{j+1} := \inf\{t > \tau^N_{j}: Y_t\in\{L_1,\ldots L_N\} \}$ for $j\in\N_0$.
	For $x\in S$, we have $\P_x$-almost surely that $\{\tau_j^N\}_{j\in\N_0}$ are non-decreasing, and strictly increasing whenever they are finite.
	Moreover, there is a well-defined path $X_N = \{X_N(j)\}_{j\in\N_0}$ given by
		\begin{equation}\label{eqn:VMC-from-MP}
		X_N(j) := \begin{cases}
		a, &\text{ if } \tau_j^N<\infty \text{ and }Y_{\tau^N_j} = L_a \text{ for } a \in \llbracket 1,N\rrbracket, \\
		0, &\text{ if } \tau_j^N=\infty,
		\end{cases}
		\end{equation}
		and $\mathbf{X} := \{X_N\}_{N\in\N}$ is a VMC on $(\Omega,\F,\P_x)$ for all $x\in S$.
\end{lemma}

\begin{proof}
Fix $x\in S$ and $N\in\N$, and write $\mathcal{L}_N = \{L_1,\ldots L_N\}$ for convenience.
Note that $\{\tau_{j}^N\}_{j\in\N_0}$ are obviously non-increasing, so our first step is to show that, $\P_x$-almost surely, we have $\tau_j^N < \tau_{j+1}^N$ whenever $\tau_{j}^{N}<\infty$. 
To do this, define the function $p_N:S\to[0,1]$ via
\begin{equation*}
p_N(y) = \P_y(\inf\{t> 0: Y_t \in \mathcal{L}_N\} > 0).
\end{equation*}
For arbitrary $a\in\llbracket 1,N\rrbracket$, note that every point being irregular for itself implies
\begin{align*}
p_N(L_a) &= \P_{L_a}(\inf\{t> 0: Y_t \in \mathcal{L}_N\} > 0)\\
&= \P_{L_a}(\inf\{t> 0: Y_t  = L_a \text{ or } Y_t\in\mathcal{L}_N\setminus\{L_a\}\} > 0) \\
&= \P_{L_a}(\inf\{t> 0: Y_t  = L_a \} > 0\text{ and } \inf\{t> 0: Y_t  \in \mathcal{L}_N\setminus\{L_a \}\} > 0) \\
&= \P_{L_a}(\inf\{t> 0: Y_t  \in \mathcal{L}_N\setminus\{L_a \}\} > 0).
\end{align*}
Also note that, on the event $\{\inf\{t> 0: Y_t \in \mathcal{L}_N\setminus\{L_a\}\} = 0\}$, we have $Y_{t_n} \in \mathcal{L}_N\setminus\{L_a\}$ for some sequence of times $\{t_n\}_{n\in\N}$ with $t_n\to 0$.
But $\mathcal{L}_N\setminus\{L_a\}$ is closed and $Y$ is right continuous, so this implies $Y_0\in\mathcal{L}_N\setminus\{L_a\}$, which is obviously a contradiction if $Y_0 = L_a$.
Combining this with the above shows $p_N(L_a) = 1$ for all $a\in\llbracket 1,N\rrbracket$.
Then let $j\in\N$ be arbitrary.
Note that the definition of $\{\tau_j^N\}_{j\in\N_0}$ implies $\ind\{\tau_{j}^N < \tau_{j+1}^N\}= \ind\{\inf\{t> 0: Y_t \in \mathcal{L}_N\} > 0\}\circ\theta_{\tau_{j}^N}$, so the strong Markov property gives
\begin{equation*}
\P_x\left(\tau_{j}^N < \tau_{j+1}^N\, \Big|\, \F_{\tau_j^N}\right) = p_N(Y_{\tau_{j}^N}) =1.
\end{equation*}
Therefore, we have
\begin{align*}
\P_x(\tau_{j}^N < \tau_{j+1}^N,\tau_{j}^N<\infty) &= \E_x\left[\ind\{\tau_{j}^N<\infty\}\P_x\left(\tau_{j}^N < \tau_{j+1}^N\, \Big|\, \F_{\tau_{j}^N}\right)\right] \\
&= \P_x\left(\tau_{j}^N<\infty\right).
\end{align*}
This shows $\P_x(\tau_{j}^N < \tau_{j+1}^N\, |\, \tau_{j}^N<\infty)=1$, so intersecting this statement over all $j\in\N_0$ gives the desired result.

To see that \eqref{eqn:VMC-from-MP} indeed constructs a well-defined element of $\chain_N$, note that $Y$ is right-continuous and $\mathcal{L}_N$ is closed, so we of course have $Y(\tau_{j}^N)\in \mathcal{L}_N$ holding $\P_x$-almost surely on the event $\{\tau_{j}^N<\infty\}$.
Then since $L_1,\ldots, L_N$ are distinct, there exists a unique $a\in\llbracket 1,N\rrbracket$ such that $Y(\tau_{j}^N) = L_a$.
The measurability of $X_N(j)$ follows from the facts that $X_N$ is progressively measurable and that $\tau_j^N$ is a stopping time.

It only remains to show that $X_N$ is a Markov chain in $\llbracket 1,N\rrbracket$ with cemetery 0 for each $N\in\N$.
Indeed, take any $N\in\N,\ell\in\N_0$, and $a_0,\ldots, a_{\ell+1}\in\llbracket 0,N\rrbracket$, and let us show
\begin{align*}
&\P_x(X_N(\ell+1)=a_{\ell+1}\,|\,X_N(0)=a_{0},\ldots X_N(\ell)=a_{\ell}) \\
&=\P_x(X_N(\ell+1)=a_{\ell+1}\,|\,X_N(\ell)=a_{\ell}).
\end{align*}
To do this, first suppose that $a_j = 0$ for some $j\in\llbracket 0,\ell\rrbracket$.
This implies $\tau_j^N = \infty$, hence $\tau_{\ell}^N = \infty$, so both sides above are equal to $\ind\{a_{\ell+1} = 0\}$.
Otherwise, we can assume $a_0,\ldots, a_{\ell}\in\llbracket 1,N\rrbracket$.
Now, if $a_{\ell+1} = 0$, we use the strong Markov property of $Y$ to get:
\begin{align*}
&\P_x(X_N(\ell+1)=0\, |\, X_N(0)=a_{0},\ldots, X_N(\ell)=a_{\ell}) \\
&=\P_x(\tau_{\ell+1}^N=\infty\, |\, Y_{\tau_{0}^N} = L_{a_{0}},\ldots, Y_{\tau_{\ell}^N} = L_{a_{\ell}}) \\
&=\P_x(\tau_{\ell+1}^N=\infty\, |\, Y_{\tau_{\ell}^N} = L_{a_{\ell}}) \\
&=\P_x(X_N(\ell+1)=0\, |\, X_N(\ell)=a_{\ell}),
\end{align*}
and, if $a_{\ell+1}\neq 0$, we use the strong Markov property again to get
\begin{align*}
&\P_x(X_N(\ell+1)=a_{\ell+1}\,|\,X_N(0)=a_{0},\ldots, X_N(\ell)=a_{\ell}) \\
&=\P_x(Y_{\tau_{\ell+1}^N} = L_{a_{\ell+1}}\,|\,Y_{\tau_{0}^N} = L_{a_{0}},\ldots, Y_{\tau_{\ell}^N} = L_{a_{\ell}}) \\
&=\P_x(Y_{\tau_{\ell+1}^N} = L_{a_{\ell+1}}\,|\, Y_{\tau_{\ell}^N} = L_{a_{\ell}}) \\
&=\P_x(X_N(\ell+1)=a_{\ell+1}\,|\,X_N(\ell)=a_{\ell}).
\end{align*}
This finishes the proof of the result.
\end{proof}

First let observe that the VMCs arising from virtual permutations via Example~\ref{ex:virtual-permutations} can also be constructed via Lemma~\ref{lem:construct-VMC-from-MP}.

\begin{example}\label{ex:VMC-from-virtual-perm}
	Let $\boldsymbol{\sigma} = \{\sigma_N\}_{N\in\N}\in\vperm$ be a virtual permutation, and set $S := \coprod_{c\in\N}\circle_c$ where $\circle_c$ is a copy of the circle $\R/\Z$ for each $c\in\N$.
	The goal will be to define $\{L_N\}_{N\in\N}$ in such a way that each ``cycle'' in $\boldsymbol{\sigma}$ corresponds to a unique circle $\circle_c$ for $c\in\N$.
	
	To do this, first set $L_1$ to be any point on $\circle_1$, and then proceed recursively.
	For $N\in\N$, consider two cases:
	If $N+1$ lies in its own cycle in $\sigma_{N+1}$, then we set $L_{N+1}$ to be any point in $\circle_c$, where $c\in\N$ is the smallest index such that $\coprod_{c'=c+1}^{\infty}\circle_c$ contains no elements of $\{L_{N'}\}_{N'=1}^{N}$.
	If $N+1$ lies in an existing cycle in $\sigma_{N+1}$, then there exist $a,b\in\llbracket 1,N\rrbracket$ satisfying $\sigma_{N+1}(a) = N+1$ and $\sigma_{N+1}(N+1) = b$, and, by construction, this means there is some $c\in\N$ such that $L_a$ and $L_b$ both lie in $\circle_c$, with the additional property that no element of $\{L_{N'}\}_{N'=1}^{N}$ lies in the clockwise arc between them.
	In this case, let $L_{N+1}$ be an arbitrary point on this clockwise arc.
	(To remove the choice from the construction above, we could arbitrarily decide that points in new cycles get sent to $0\text{ mod } 1$ and that points in old cycles get sent to the midpoints of the arcs they fall into.)
	
	Now consider the non-random process $Y$ which moves clockwise with rate one; let $\{\P_x\}_{x\in S}$ denote the (degenerate) probability measures governing this process.
	This clearly satisfies the hypotheses of Lemma~\ref{eqn:VMC-from-MP}, so we construct $\mathbf{X}$ by the result therein.
	Then, for arbitrary $a\in\N$, the VMC $\mathbf{X}$ on $(\Omega,\F,\P_{L_a})$ coincides with the virtual chain $\mathbf{x}_{\boldsymbol{\sigma},a}$ given by Example~\ref{ex:virtual-permutations}.
	To visualize this construction, see Figure~\ref{fig:vperm}.
	
	Observe that $\boldsymbol{\sigma}$ corresponds to a classical permutation if and only if this construction is such that $\circle_c\cap\{L_N\}_{N\in\N}$ is finite for each $c\in\N$.
	By compactness, this is equivalent to the statement that $\{L_N\}_{N\in\N}$ has no accumulation points.
\end{example}

\begin{figure}
	\begin{tikzpicture}[scale=0.4]
	\draw[thick] (0,0) circle [radius=2] {};
	\draw[thick] (6,0) circle [radius=2] {};
	\draw[thick] (12,0) circle [radius=2] {};
	\draw (16,0) node[]{$\cdots$};
	
	\draw[fill,color=blue] (0,2) circle [radius=0.2] node[above]{$L_1$};
	\draw[fill,color=blue] (6,2) circle [radius=0.2] node[above]{$L_2$};
	\draw[fill,color=blue] (0,-2) circle [radius=0.2] node[below]{$L_3$};
	\draw[fill,color=blue] (-2,0) circle [radius=0.2] node[left]{$L_4$};
	\draw[fill,color=blue] (12,2) circle [radius=0.2] node[above]{$L_5$};
	\draw[fill,color=blue] (-1.414,1.414) circle [radius=0.2] node[above left]{$L_6$};
	\draw[fill,color=blue] (6,-2) circle [radius=0.2] node[below]{$L_7$};
	\draw[fill,color=blue] (4,0) circle [radius=0.2] node[left]{$L_8$};
	\draw[fill,color=blue] (4.586,1.414) circle [radius=0.2] node[above left]{$L_9$};
	\draw[fill,color=blue] (5.2,1.85) circle [radius=0.2];
	\draw[fill,color=blue] (5.6,1.95) circle [radius=0.2];
	\draw[fill,color=blue] (5.75,1.975) circle [radius=0.2];x
	\draw[color=blue] (4.85,2.25) node[]{\tiny $\cdot^{\ \cdot^{\ \cdot}}$};
	
	\draw[thick,*->,color=red] (8,0) arc (0:-60:2);
	\draw[color=red] (7.5,-1) node[below right] {$Y$};
	
	\draw (0,0) node[]{$\circle_1$};
	\draw (6,0) node[]{$\circle_2$};
	\draw (12,0) node[]{$\circle_3$};
	
	\end{tikzpicture}
	\caption{As in Example~\ref{ex:VMC-from-virtual-perm}, we can construct the VMC corresponding to a virtual permutation via Lemma~\ref{lem:construct-VMC-from-MP}.}
	\label{fig:vperm}
\end{figure}
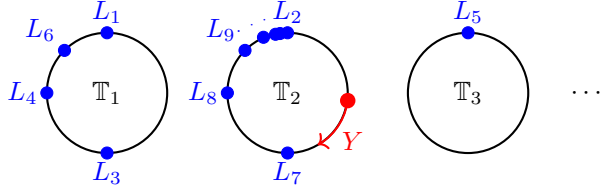

In many (but not all) of the examples of VMCs that we will see throughout the paper, the cemetery state of 0 is never visited.
In the setting of Lemma~\ref{lem:construct-VMC-from-MP}, a sufficient condition for this to occur is that the stopping times $\{\tau_{j}^N\}_{j\in\N_0,N\in\N}$ are all finite; this is in turn implied by the property that the process $Y$ is point-recurrent in the sense that
\begin{equation*}
\P_x(H_Y(x') \text{ is unbounded for all } x'\in S) =1
\end{equation*}
for all $x\in S$, where we have defined the set $H_Y(x') := \{t\ge0: Y_t = x' \}$.

We also give another interesting class of examples arising from the theory of regenerative sets.

\begin{example}\label{ex:down-from-infty}
Let $F$ the distribution function of a measure on $(0,\infty)$, and let $(\Omega,\F,\P)$ be a probability space on which is defined a sequence $\{Z_k\}_{k\in\N}$ of i.i.d. random variables with distribution $F$.
Then define the renewal (point) process $E$ via
\begin{equation*}
E := \left\{\sum_{k=1}^{K}Z_k: K\in\N_0\right\},
\end{equation*}
and note that $E$ is a regenerative random closed set on $(\Omega,\F,\P)$.
It also has empty interior $\P$-almost surely, so by \cite[Theorem~1]{RegenerativeSets}, it follows that the age process $Y = \{Y_t\}_{t\ge 0}$ of $E$, given by
\begin{equation*}
Y_t := t - \sup\{0 \le s \le t : s \in E\}
\end{equation*}
for $t\ge 0$ is a right process in the state space $S = [0,\infty)$.
Now let $\{L_N\}_{N\in\N}$ be a strictly decreasing sequence with $L_N\downarrow 0$ as $N\to\infty$.
Finally, let $\mathbf{X} = \{X_N\}_{N\in\N}$ be the VMC on $(\Omega,\F,\P)$ guaranteed by Lemma~\ref{lem:construct-VMC-from-MP} via \eqref{eqn:VMC-from-MP}.

Note that, $\P$-almost surely, we have $Z_1 > 0$, hence $X_N(0) = N$ for all $N > \frac{1}{Z_1}$.
Thus, $\P$-almost surely, the limit $\lim_{N\to\infty}X_N$ does not exist in $\chain$.
Therefore, by Lemma~\ref{lem:VC-C-characterization} we have $\P(\mathbf{X}\in \iota(\chain)) = 0$; see Figure~\ref{fig:down-from-infty} for an illustration.
Intuitively speaking, then, $\mathbf{X}$ is a VMC which tries to ``come down from infinity'' but ``jumps back to infinity'' at some random times.
\end{example}

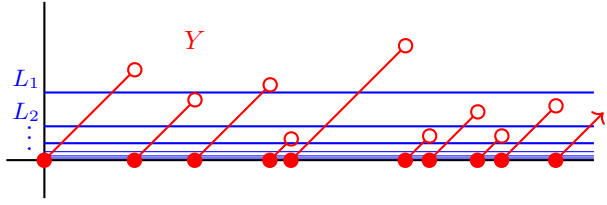
\begin{figure}
	\begin{tikzpicture}[scale=0.4,shorten >=-3pt,shorten <=-3pt]
	\draw[-, thick] (0, -1) to (0,5);
	\draw[-, thick] (-1,0) to (18,0);
	
	\draw[-, thick, color=blue] (0.25,2.25) to (18,2.25);
	\draw[-, thick, color=blue] (0.25,1.125) to (18,1.125);
	\draw[-, thick, color=blue] (0.25,0.5625) to (18,0.5625);
	\draw[-, color=blue] (0.25,0.28125) to (18,0.28125);
	\draw[-, color=blue] (0.25,0.140625) to (18,0.140625);
	\draw[-, color=blue] (0.25,0.0703125) to (18,0.0703125);
	
	\draw[*-o, thick,color=red] (0,0) to (3,3);
	\draw[*-o, thick,color=red] (3,0) to (5,2);
	\draw[*-o, thick,color=red] (5,0) to (7.5,2.5);
	\draw[*-o, thick,color=red] (7.5,0) to (8.2,0.7);
	\draw[*-o, thick,color=red] (8.2,0) to (12,3.8);
	\draw[*-o, thick,color=red] (12,0) to (12.8,0.8);
	\draw[*-o, thick,color=red] (12.8,0) to (14.4,1.6);
	\draw[*-o, thick,color=red] (14.4,0) to (15.2,0.8);
	\draw[*-o, thick,color=red] (15.2,0) to (17,1.8);
	\draw[*->, thick,color=red] (17,0) to (18.4,1.4);
	
	\draw[color=blue] (-0.6,2.7) node[]{\small $L_1$};
	\draw[color=blue] (-0.6,1.6) node[]{\small $L_2$};
	\draw[color=blue] (-0.5,1) node[]{$\vdots$};
	
	\draw[color=red] (5,4) node[]{$Y$};
	\end{tikzpicture}
	\label{fig:down-from-infty}
	\caption{A generic sample path of $Y$ and $\mathbf{X}$, when $Y$ is the age process of a regenerative set as given in Example~\ref{ex:down-from-infty}.}
\end{figure}

We also note that the construction above works for $E$ any regenerative set with empty interior almost surely.
In particular, it seems interesting to consider the case that $E$ is the zero set of a Brownian motion, in which case there are very many short excursions down from infinity.

\begin{example}\label{ex:splitting-VMC}
Consider the space $S = (-1,1)$ with the topology inherited form the usual topology on the real line.
Now let $(\Omega,\F,\{Y_t\}_{t\ge 0},\{\P_x\}_{x\in S})$ be the process which moves deterministically at rate one towards the closer of the two endpoints $\{-1,+1\}$, which teleports back to $x=0$ once it hits one of these endpoints, and which breaks the tie at $x=0$ by moving into either half of the domain with equal probability.

More concretely, the semigroup $\{P_t\}_{t\ge 0}$ of this process defined via $(P_tf)(x) := \E_x[f(Y_t)]$ for all bounded, measurable $f:S\to \R$ is exactly
\begin{equation*}
(P_tf)(x) = \begin{cases}
\frac{1}{2}f(t \text{ mod } 1)+\frac{1}{2}f(-(t \text{ mod } 1)), &\text{ if } x=0, \\
f(x+t) &\text{ if } x>0, \text{ and } |x|+t < 1, \\
f(x-t) &\text{ if } x< 0, \text{ and } |x|+t < 1, \\
\frac{1}{2}f(x+t \text{ mod } 1)+\frac{1}{2}f(-(x+t \text{ mod } 1)), &\text{ if } x> 0 \text{ and } |x|+t \ge 1, \\
\frac{1}{2}f(-x+t \text{ mod } 1)+\frac{1}{2}f(-(-x+t \text{ mod } 1)), &\text{ if } x< 0 \text{ and } |x|+t \ge 1. \\
\end{cases}
\end{equation*}
Of course, this process is point-recurrent and every point is irregular for itself.
However, it is easy to see that this is not a strong Markov process with respect to the right-continuous augmentation of its natural filtration $\{\F_t\}_{t\ge 0}$.
(For example, the strong Markov property fails when applied to the $\{\F_{t+}\}_{t\ge 0}$-stopping time $\tau :=\inf\{t \ge \tau_0: Y_t > 0 \}$, where $\tau_{0} := \inf\{t\ge 0: Y_t = 0 \}$.)

Now choose any sequence $\{\ell_N\}_{N\in\N}$ in $(0,1)$ with $\ell_N\downarrow 0$ as $N\to\infty$, and define the points $\{L_N\}_{N\in\N}$ via
\begin{equation*}
L_N = \begin{cases}
\ell_{(N+1)/2}, &\text{ if } N \text{ odd}, \\
-\ell_{N/2}, &\text{ if } N \text{ even}. \\
\end{cases}
\end{equation*}
An argument identical to that of Lemma~\ref{lem:construct-VMC-from-MP} still applies to show that $\mathbf{X}$ constructed therein is a VMC, since the strong Markov property is only needed at the hitting times of $\{L_N\}_{N\in\N}$.
We regard $\mathbf{X} = \{X_N\}_{N\in\N}$ as a VMC on the probability space $(\Omega,\F,\P_0)$.

We of course have $\lim_{N\to\infty}X_N(0) = \infty$ holding $\P_0$-almost surely, hence $\P_0(\mathbf{X} \in \iota(\chain)) =0$ by Proposition~\ref{prop:VMC-MC-characterization}.
See Figure~\ref{fig:splitting-VMC} for a generic realization of the sample path of $Y$ and $\mathbf{X}$ under $\P_0$.
Intuitively speaking, $\mathbf{X}$ is a VMC which has ``two different ways to come down from infinity''; it chooses one uniformly at random, traverses the (deterministic) path down to the lowest possible state, and then repeats the process over again.

This VMC has a non-trivial probabilistic component as well as non-trivial behavior ``at infinity'' as desired, but it is distinguished from out previous examples in that it also has non-trivial ``probabilistic behavior at infinity''.
To see this, consider the \textit{tail $\sigma$-algbra of $\mathbf{X}$}, defined via $\mathcal{T}(\mathbf{X}) = \bigcap_{N\in\N}\sigma(X_N,X_{N+1},\ldots)$.
Then, the event
\begin{equation*}
A = \{X_N(0) \text{ is even for infinitely many } N\in\N \}
\end{equation*}
is $\mathcal{T}(\mathbf{X})$-measurable, and satisfies $\P_0(A)= 1/2$.
This is closely related to the fact that the Blumenthal zero-one law fails for the constituent Markov process $Y$.
\end{example}

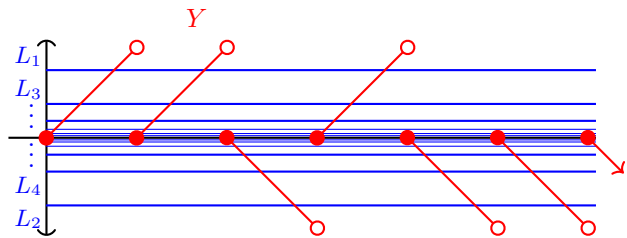
\begin{figure}
	\begin{tikzpicture}[scale=0.4,shorten >=-3pt,shorten <=-3pt]
	\draw[(-), thick] (0, -3) to (0,3);
	\draw[-, thick] (-1,0) to (18,0);
	
	\draw[-, thick, color=blue] (0.25,2.25) to (18,2.25);
	\draw[-, thick, color=blue] (0.25,1.125) to (18,1.125);
	\draw[-, thick, color=blue] (0.25,0.5625) to (18,0.5625);
	\draw[-, color=blue] (0.25,0.28125) to (18,0.28125);
	\draw[-, color=blue] (0.25,0.140625) to (18,0.140625);
	\draw[-, color=blue] (0.25,0.0703125) to (18,0.0703125);
	
	\draw[-, thick, color=blue] (0.25,-2.25) to (18,-2.25);
	\draw[-, thick, color=blue] (0.25,-1.125) to (18,-1.125);
	\draw[-, thick, color=blue] (0.25,-0.5625) to (18,-0.5625);
	\draw[-, color=blue] (0.25,-0.28125) to (18,-0.28125);
	\draw[-, color=blue] (0.25,-0.140625) to (18,-0.140625);
	\draw[-, color=blue] (0.25,-0.0703125) to (18,-0.0703125);
	
	\draw[*-o, thick,color=red] (0,0) to (3,3);
	\draw[*-o, thick,color=red] (3,0) to (6,3);
	\draw[*-o, thick,color=red] (6,0) to (9,-3);
	\draw[*-o, thick,color=red] (9,0) to (12,3);
	\draw[*-o, thick,color=red] (12,0) to (15,-3);
	\draw[*-o, thick,color=red] (15,0) to (18,-3);
	\draw[*->, thick,color=red] (18,0) to (19,-1);

	\draw[color=blue] (-0.6,2.7) node[]{\small $L_1$};
	\draw[color=blue] (-0.6,1.6) node[]{\small $L_3$};
	\draw[color=blue] (-0.5,1) node[]{$\vdots$};
	
	\draw[color=blue] (-0.6,-2.7) node[]{\small $L_2$};
	\draw[color=blue] (-0.6,-1.6) node[]{\small $L_4$};
	\draw[color=blue] (-0.5,-0.3) node[]{$\vdots$};
	
	\draw[color=red] (5,4) node[]{$Y$};
	\end{tikzpicture}
	\label{fig:splitting-VMC}
	\caption{A generic sample path of the process $Y$ described by Example~\ref{ex:splitting-VMC}.}
\end{figure}

\section{Canonical Data and the Representation Theorem}\label{sec:canonical-data}

Suppose $(\Omega,\F,\P)$ is some probability space on which is defined $X = \{X(i)\}_{i\in\N_0}$ a MC in $\N$ with cemetery 0.
Then, $X$ has an associated \textit{initial distribution (ID)} $\nu = \P\circ X(0)^{-1}$ and \textit{transition matrix (TM)} $K\in[0,1]^{\N_0\times\N_0}$ given by $K(a,b) = \P(X(1) = b\, |\, X(0) = a)$ for $a,b\in\N_0$, which also satisfies $K(0,a) = \ind\{a = 0\}$ for all $a\in\N_0$.
In this case, we say that $(\nu,K)$ are the \textit{canonical data} of $X$.
It is classical that for every pair of canonical data there is a probability space on which is defined a MC with these canonical data; moreover, this object is unique in the sense that any two MCs with the same canonical data must have the same law.
In other words, the collection of laws of MCs in $\N$ with cemetery 0 is in bijection with the collection of pairs of canonical data.

The goal of this section is show that an analogous statement is true for VMCs (Theorem~\ref{thm:VMC-VID-VTM-bijection}), where we suitably define the notions of \textit{virtual transition matrix (VTM}, Subsection~\ref{subsec:VTM}\textit{)} and \textit{virtual initial distribution (VID}, Subsection~\ref{subsec:VID}\textit{)}, and introduce an important notion of \textit{compatibility} (Subsection~\ref{subsec:compatibility}) which is not needed in the MC case.
In Subsection~\ref{subsec:Ex-2} we explore various examples.

\subsection{Virtual Transition Matrices}\label{subsec:VTM}

First we define the important notion of a virtual transition matrix, which will later be seen to encode most of the ``interesting'' information about VMCs.

\begin{definition}
Write $\transmat$ for the space of all transition matrices on $\N_0$ with $K(0,0) = 1$, which we endow with the topology of pointwise convergence.
\end{definition}

Moreover, for each $N\in\N$ write $\transmat_N$ for the space of all transition matrices on $\N_0$ with $K(a,a) = 1$ for all $a\in \N_0\setminus \llbracket 1,N\rrbracket$.
Observe of course that we have $\transmat_N\subseteq \transmat_{M}\subseteq \transmat$ for all $N,M\in\N$ with $N\le M$.

Next we define suitable projection maps for transition matrices from our existing projection maps for paths.
Specifically, for each $K\in\transmat$ and $N\in\N$, we define $P_N(K)$ to be the transition matrix of the Markov chain $P_N(X)$ where $X$ is a Markov chain with transition matrix $K$.
More specifically, for each $K \in\transmat$ and $a\in\N_0$ write $(\Omega,\F,\P_{K,a})$ for a probability space on which is defined a MC $X = \{X(i)\}_{i\in\N_0}$ with transition matrix $K$ and initial state $X(0) = a$.
Then, for each $N\in\N$ and $a,b\in\N_0$, set
\begin{equation}\label{eqn:K-proj}
(P_N(K))(a,b) = \begin{cases}
\P_{K,a}((P_N(X))(1) = b), &\text{ if } a \in\llbracket 1,N\rrbracket, \\
1, &\text{ else if } a = b, \\
0, &\text{ else if } a \neq b. \\
\end{cases}
\end{equation}
It is easily verified that indeed $P_N(K)\in\transmat_{N}$.

The next goal is to develop some properties of these projection maps which parallel (to the extent possible) the analogous results for the projection maps on spaces of paths.

\begin{lemma}\label{lem:K-proj-iterate}
	For $N\in\N$, we have $P_{N}\circ P_{N+1} = P_N$ on $\transmat$.
\end{lemma}

\begin{proof}
For any $K\in\transmat$ and $a,b\in \N_0$ we have, from the projectivity property $P_{N}\circ P_{N+1} = P_N$ on $\chain$,
\begin{align*}
(P_N(K))(a,b) &= \P_{K,a}((P_N(X))(1) = b) \\
&= \P_{K,a}((P_{N}(P_{N+1}(X)))(1) = b) \\
&= \P_{P_{N+1}(K),a}((P_{N}(X))(1) = b) \\
&= (P_N(P_{N+1}(K)))(a,b),
\end{align*}
as claimed.
\end{proof}

\begin{lemma}\label{lem:K-proj-approx-identity}
	For each $K\in\transmat$, we have $P_N(K)\to K$ as $N\to\infty$.
\end{lemma}

\begin{proof}
	Let $a,b\in\N_0$ be arbitrary.
	For $N\ge \max\{a,b\}$, we find ourselves in the first case of \eqref{eqn:K-proj}, hence we have
	\begin{equation*}
	\lim_{N\to\infty}K_N(a,b) = \lim_{N\to\infty}\P_{K,a}(X_N(1) = b)= \P_{K,a}(X(1) = b)= K(a,b),
	\end{equation*}
	as needed.
\end{proof}

\begin{definition}
	The set
	\begin{equation*}
	\vtransmat := \left\{\{K_N\}_{N\in\N}\in \prod_{N\in\N}\transmat_N: P_{N}(K_{N+1}) =K_N \text{ for all }N\in\N \right\}
	\end{equation*}
	is called the space of \textit{virtual transition matrices (VTMs)}.
\end{definition}

\begin{lemma}\label{lem:iota-K-injection}
	The map $\iota:\transmat\to\vtransmat$ defined in the natural way via $\iota(K) := \{P_N(K)\}_{N\in\N}$ is well-defined and injective.
\end{lemma}

\begin{proof}
	To see $\iota(K)\in\vtransmat$, simply apply Lemma~\ref{lem:K-proj-iterate}.
	To see injectivity, suppose that $K,K'\in \transmat$ have $\iota(K) = \iota(K')$, and write $\iota(K) = \{K_N\}_{N\in\N}$ and $\iota(K') = \{K'_N\}_{N\in\N}$.
	Then, for any $a,b\in\N_0$, we have, by Lemma~\ref{lem:K-proj-approx-identity},
	\begin{equation*}
	K(a,b) = \lim_{N\to\infty}K_N(a,b) = \lim_{N\to\infty}K'_N(a,b) = K'(a,b),
	\end{equation*}
	as claimed.
\end{proof}

\begin{lemma}\label{lem:VK-K-characterization}
	A VTM $\mathbf{K} = \{K_N\}_{N\in\N}\in\vtransmat$ lies in $\iota(\transmat)$ if and only if $K = \lim_{N\to\infty}K_N$ exists and $\mathbf{K} = \iota(K)$.
\end{lemma}

\begin{proof}
If $\lim_{N\to\infty}K_N$ exists and $\mathbf{K} = \iota(\lim_{N\to\infty}K_N)$, then $\mathbf{K}\in\iota(\transmat)$ trivially.
Conversely, if we have $\mathbf{K} = \iota(K)$ for some $K\in\transmat$, then the result follows from Lemma~\ref{lem:K-proj-approx-identity}.
\end{proof}

Next let us show that the projection $P_N:\transmat_{N+1}\to\transmat_{N}$ for any $N\in\N$ has a particularly simple form which is very amenable to calculations.

\begin{lemma}
For $N\in\N$ and $K\in\transmat_{N+1}$, we have
\begin{equation}\label{eqn:TM-projection}
(P_N(K))(a,b) = \begin{cases}
K(a,b) + \frac{K(a,N+1)K(N+1,b)}{1-K(N+1,N+1)}, &\text{ if } K(N+1,N+1) < 1, \\
K(a,0) + K(a,N+1), &\text{ else if } b = 0, \\
K(a,b), &\text{ else if } b \neq 0, \\
\end{cases}
\end{equation}
for $a,b\in\llbracket 0,N\rrbracket$, as well as $(P_N(K))(a,a) = 1$ for $a\in\N_0\setminus\llbracket 1,N\rrbracket$.
\end{lemma}

\begin{proof}
	First observe that the claimed form of $P_N(K)$ is indeed a transition matrix, and moreover that we have $P_N(K)\in\transmat_{N}$.
	Next, note that $a\in\N_0\setminus \llbracket 1,N\rrbracket$ implies $(P_N(K))(a,a)= 1$ by definition, so we only need to check \eqref{eqn:TM-projection} for $a,b\in\llbracket 1,N\rrbracket$.
	Now note, in particular, that $X_{N+1} = X$.
	Then, if $K(N+1,N+1) < 1$, we have:
	\begin{align*}
	(P_N(K))(a,b) &= \P_{K,a}(X_N(1) = b) \\
	&= \P_{K,a}(I_{X,N}(1) < \infty,X(I_{X,N}(1)) = b) \\
	&= K(a,b) + K(a,N+1)\sum_{k=1}^{\infty}(K(N+1,N+1))^kK(N+1,b) \\
	&= K(a,b) + \frac{K(a,N+1)K(N+1,b)}{1-K(N+1,N+1)} \\
	&= (P_N(K))(a,b),
	\end{align*}
	as needed.
	Similarly, if $K(N+1,N+1) = 1$, then we have
	\begin{align*}
	(P_N(K))(a,0) &= \P_{K,a}(X_N(1) = 0) \\
	&= \P_{K,a}(I_{X,N} = \infty) + \P_{K,a}(I_{X,N} ,\infty,X(I_{X,N}) = 0) \\
	&= K(a,0)+ K(a,N+1) \\
	&= (P_N(K))(a,0)
	\end{align*}
	and, for $b\in\llbracket 1,N\rrbracket$, we have
	\begin{align*}
	(P_N(K))(a,b) &= \P_{K,a}(X_N(1) = b) \\
	&= \P_{K,a}(I_{X,N} <\infty,X_N(I_{X,N}) = b) \\
	&= K(a,b) \\
	&= (P_N(K))(a,b).
	\end{align*}
	This proves the claim.
\end{proof}

\begin{remark}
The description of the map $P_{N}:\transmat_{N+1}\to\transmat_{N}$ can be made even more explicit by regarding elements $\transmat_{N}$ for $N\in\N$ as suitable block matrices.
More specifically, by a slight abuse of notation, we identify an element $K\in\transmat_{N+1}$ for $N\in\N$ as its top-left $(N+2)\times (N+2)$ submatrix; in this setting, if an element $K_{N+1}\in\transmat_{N+1}$ is identified with
\begin{equation}\label{eqn:TM-block-form}
\begin{pmatrix}
1 & 0 & 0 \\
w & A & u \\
q & v^{\text{T}} & p
\end{pmatrix},
\end{equation}
where $A\in [0,1]^{N\times N}$ is a matrix, $u,v,w\in [0,1]^{N}$ are column vectors, and $p,q\in [0,1]$ are real numbers, then the projection $P_{N}:\transmat_{N+1}\to\transmat_{N}$ is exactly
\begin{equation}\label{eqn:K-proj-formula}
\begin{pmatrix}
1 & 0 & 0 \\
w & A & u \\
q & v^{{\text{T}}} & p
\end{pmatrix} \mapsto \begin{cases}
\begin{pmatrix}
1 & 0 \\
w + \frac{q}{1-p}u& A + \frac{1}{1-p}uv^{\text{T}} \\
\end{pmatrix}, &\text{ if } p < 1, \\
& \\
\begin{pmatrix}
1 & 0 \\
w + u& A \\
\end{pmatrix}, &\text{ if } p = 1.
\end{cases}
\end{equation}
It does not appear that there is an analogous simple formula for the composite projection map $P_N\circ P_{N+1}:\transmat_{N+2}\to\transmat_{N}$ for $N\in\N$, nor for the general $P_N:\transmat\to\transmat_{N}$ for $N\in\N$.
\end{remark}

\begin{remark}
Note that, as $p\to 1$, we have
\begin{equation*}
K_3(p) := \begin{pmatrix}
1 & 0 & 0 \\
1-p & 0 & p \\
0 & 1-p & p
\end{pmatrix} \to \begin{pmatrix}
1 & 0 & 0 \\
0 & 0 & 1 \\
0 & 0 & 1 \\
\end{pmatrix} = K_3(0)
\end{equation*}
and
\begin{equation*}
P_2(K_3(p)) = \begin{pmatrix}
1 & 0 \\
1-p & p \\
\end{pmatrix} \to \begin{pmatrix}
1 & 0 \\
0 & 1 \\
\end{pmatrix} \neq \begin{pmatrix}
1 & 0 \\
1 & 0 \\
\end{pmatrix} = P_2(K_3(0)).
\end{equation*}
This example shows that the projections $P_N:\transmat\to\transmat_{N}$ for $N\in\N$ and the inclusion $\iota:\transmat\to\vtransmat$ can fail to be continuous, and also that $\vtransmat$ is not compact.
\end{remark}

\subsection{Virtual Initial Distributions}\label{subsec:VID}

While it was relatively straightforward to extend the notion of transition matrices to that of virtual transition matrices, the process of extending the notion of initial distributions (IDs) to an analogous notion for VMCs is somewhat more complicated.

\begin{definition}
	The set
	\begin{equation*}
	\vinitstate := \left\{\{y_N\}_{N\in\N}\in \prod_{N\in\N}\llbracket 0,N\rrbracket: \begin{matrix}
	\text{for all }N\in\N, \text{ either } \\ y_{N+1} = N+1  \text{ or } y_{N+1} =y_{N} 
	\end{matrix} \right\}
	\end{equation*}
	is called the space of \textit{virtual initial states (VISs)}.
\end{definition}

Observe that $\vinitstate$ is closed in $\prod_{N\in\N}\llbracket 0,N\rrbracket$, hence $\vinitstate$ is itself a compact Polish space.
For for $N\in\N$, we also define the truncated sets
\begin{equation*}
\vinitstate_{N} := \left\{\{y_N\}_{M=1}^{N}\in \prod_{M=1}^{N}\llbracket 0,M\rrbracket: \begin{matrix}
\text{for all }1 \le M < N, \text{ either } \\ y_{M+1} = M+1  \text{ or } y_{M+1} =y_{M} 
\end{matrix} \right\},
\end{equation*}
which are finite sets.
Observe that elements of $\vinitstate$ correspond to infinite sequences consisting of flat stretches and jumps, where a jump at index $N\in\N$ must jump to level $N\in\N$; elements of $\vinitstate_{N}$ for $N\in\N$ are the analogous finite sequences.

\begin{definition}
	A probability measure $\lambda$ on $\mathcal{M}_1(\vinitstate)$ is called \textit{Markovian} if for all $N\in\N$ and $\{a_{M}\}_{M=1}^{N+1} \in \vinitstate_{N+1}$ we have $\lambda(y_{N+1} = a_{N+1}\, |\, y_1 = a_{1},\ldots , y_{N} = a_{N}) = \lambda(y_{N+1} = a_{N+1}\, |\, y_{N} = a_{N})$.
	Then define the set $\vdist_{0} = \left\{\lambda \in \mathcal{M}_1(\vinitstate): \lambda \text{ is Markovian}\right\}$, called the space of \textit{virtual distributions (VDs)}.
\end{definition}

However, it turns out that the elements of $\vdist_{0}$ are in bijection with a simpler space which is much easier to work with.
To state this, consider the following:

\begin{definition}
The set
\begin{equation*}
\vdist_1: = \left\{\{\nu_N\}_{N\in\N}\in\prod_{N\in\N}\mathcal{M}_1(\llbracket 0,N\rrbracket): \nu_{N}(a) \ge \nu_{N+1}(a) \text{ for } N\in\N,a\in\llbracket 0,N\rrbracket\right\}
\end{equation*}
is, by abuse, also called the space of \textit{virtual distributions (VDs)}.
\end{definition}

Observe that elements of $\vdist_1$ are just sequences of marginal distributions, which a priori, do not determine a unique coupling on the entire product space.
The main result of this subsection is the following, which establishes the useful fact that $\vdist_0$ and $\vdist_1$ are two ways of viewing the same object.

\begin{proposition}\label{prop:D0-D1-bijection}
There is a homeomorphism from $\vdist_0$ to $\vdist_1$ given by sending each probability distribution to its sequence of marginal distributions.
\end{proposition}

\begin{proof}
Note that for $\lambda\in \vdist_{0}$ and any $N\in\N$ and $a\in\llbracket 0,N\rrbracket$, we have
\begin{align*}
\lambda(y_{N} = a) &= \lambda(y_{N} = a,y_{N+1} = a) +\lambda(y_{N} = a,y_{N+1} = N+1) \\
&\ge \lambda(y_{N} = a,y_{N+1} = a) \\
&= \lambda(y_{N+1} = a).
\end{align*}
Thus there is a well-defined map $\phi:\vdist_{0}\to \vdist_1$ sending each element $\lambda\in\vdist_0$ to $\phi(\lambda)$ its sequence of marginal distributions in $\vdist_1$.

Next, let us show that $\phi$ is injective.
Take $\lambda,\lambda'\in\vdist_{0}$ and write $\phi(\lambda) = \{\nu_N\}_{N\in\N}$ and $\phi(\lambda') = \{\nu'_N\}_{N\in\N}$, and suppose $\phi(\lambda) = \phi(\lambda')$.
For $N\in\N$ and $\{a_M\}_{M=1}^{N}\in\vinitstate_{N}$ write
\begin{equation*}
E_N(\{a_M\}_{M=1}^{N}) := \{\{y_N\}_{N\in\N}\in\vinitstate: y_1= a_1,\ldots, y_N= a_N  \},
\end{equation*}
and note that
\begin{equation*}
\mathcal{P} := \left\{E_N(\{a_M\}_{M=1}^{N}) : N\in\N, \{a_M\}_{M=1}^{N}\in\vinitstate_{N} \right\}
\end{equation*}
is a $\pi$-system in $\vinitstate$ which generates the Borel $\sigma$-algebra $\mathcal{B}(\vinitstate)$.
So, in order to show $\lambda = \lambda'$ it suffices to show that they agree for all elements of $\mathcal{P}$.
To do this, we'll show by induction on $N\in\N$ that they agree for all elements of
\begin{equation*}
\mathcal{P}_N := \left\{E_N(\{a_M\}_{M=1}^{N}) :\{a_M\}_{M=1}^{N}\in\vinitstate_{N} \right\}.
\end{equation*}
The base case of $N=1$ is immediate since $\vinitstate_{1} = \llbracket 0,1\rrbracket$.
Before we perform the inductive step, we note that, for any $N\in\N$ and $a\in\llbracket 0,N\rrbracket$, we have
\begin{equation}\label{eqn:phi-surj-1}
\begin{split}
\lambda(y_{N+1} = a|y_N = a) &= \frac{\lambda(y_{N+1} = a,y_N = a)}{\lambda(y_N = a)} \\
&= \frac{\lambda(y_{N+1} = a)}{\lambda(y_N = a)} \\
&= \frac{\lambda'(y_{N+1} = a)}{\lambda'(y_N = a)} \\
&= \frac{\lambda'(y_{N+1} = a,y_N = a)}{\lambda'(y_N = a)} = \lambda'(y_{N+1} = a|y_N = a),
\end{split}
\end{equation}
and, similarly,
\begin{equation}\label{eqn:phi-surj-2}
\begin{split}
\lambda(y_{N+1} = N+1|y_N = a) &= \frac{\lambda(y_{N+1} = N+1,y_N = a)}{\lambda(y_N = a)} \\
&= \frac{\lambda(y_N=a)-\lambda(y_{N+1} = a,y_N = a)}{\lambda(y_N = a)} \\
&= \frac{\lambda(y_N=a)-\lambda(y_{N+1} = a)}{\lambda(y_N = a)} \\
&= \frac{\lambda'(y_N=a)-\lambda'(y_{N+1} = N+1)}{\lambda'(y_N = a)} \\
&= \frac{\lambda'(y_N=a)-\lambda'(y_{N+1} = a,y_N = a)}{\lambda'(y_N = a)} \\
&= \frac{\lambda'(y_{N+1} = N+1,y_N = a)}{\lambda'(y_N = a)} \\
&= \lambda'(y_{N+1} = N+1|y_N = a).
\end{split}
\end{equation}
Now for the inductive step.
Take any $N\in\N$ and any $\{a_M\}_{M=1}^{N+1}\in\vinitstate_{N+1}$.
If $a_N = a_{N+1}$, then we use \eqref{eqn:phi-surj-1} and the inductive hypothesis to compute:
\begin{align*}
&\lambda(y_1 = a_1,\ldots, y_N = a_N,y_{N+1} = a_{N}) \\
&= \lambda(y_{N+1}=a_{N}\, |\,y_N = a_N)\lambda(y_1 = a_1,\ldots y_N = a_N) \\
&= \lambda'(y_{N+1}=a_{N}\,|\,y_N = a_N)\lambda'(y_1 = a_1,\ldots y_N = a_N) \\
&= \lambda'(y_1 = a_1,\ldots, y_N = a_N,y_{N+1} = a_{N}).
\end{align*}
Otherwise we have $a_{N+1} =N+1$, so we use \eqref{eqn:phi-surj-2} and the inductive hypothesis to compute:
\begin{align*}
&\lambda(y_1 = a_1,\ldots, y_N = a_N,y_{N+1} = N+1) \\
&= \lambda(y_{N+1}=N+1\,|\,y_N = a_N)\lambda(y_1 = a_1,\ldots, y_N = a_N) \\
&= \lambda'(y_{N+1}=N+1\,|\,y_N = a_N)\lambda'(y_1 = a_1,\ldots, y_N = a_N) \\
&= \lambda'(y_1 = a_1,\ldots, y_N = a_N,y_{N+1} = N+1).
\end{align*}
This completes the induction and shows that $\lambda$ and $\lambda'$ agree on $\mathcal{P}_N$, hence on all of $\mathcal{P}$, hence on all of $\mathcal{B}(\vinitstate)$.
Therefore, $\lambda = \lambda'$, so $\phi$ is injective.

The next step is to show that $\phi$ is surjective.
Take any $\boldsymbol{\nu} = \{\nu_N\}_{N\in\N}\in \vdist_1$, and let us construct a $\lambda\in\vdist_{0}$ such that $\phi(\lambda) = \boldsymbol{\nu}$.
To do this , we recursively define a sequence $\{\lambda_{N}\}_{N\in\N}$ with $\lambda_{N}\in \mathcal{M}_1(\prod_{M=1}^{N}\llbracket 0,M \rrbracket)$ for each $N\in\N$.
To start, we set $\lambda_1 = \nu_1$ on $\mathcal{M}_1(\llbracket 0,1\rrbracket)$.
Then, for $N\in\N$ define $\lambda_{N+1}\in \mathcal{M}_1(\prod_{M=1}^{N+1}\llbracket 0,M \rrbracket)$ via
\begin{align*}
&\lambda_{N+1}(y_1 = a_1,\ldots, y_{N+1} = a_{N+1}) \\
&= \begin{cases}
\frac{\nu_{N+1}(a_{N})}{\nu_{N}(a_{N})}\lambda_{N}(y_1 = a_1,\ldots, y_{N} = a_{N}), &\text{ if } a_{N+1} = a_N, \\
\frac{\nu_{N}(a_{N})-\nu_{N+1}(a_{N})}{\nu_{N}(a_{N})}\lambda_{N}(y_1 = a_1,\ldots, y_{N} = a_{N}), &\text{ if } a_{N+1} = N+1, \\
0, &\text{ if } a_{N+1}\notin \{a_N,N+1\}.
\end{cases}
\end{align*}
Note that the definition of $\vdist_1$ guarantees $\nu_{N}(a_{N})-\nu_{N+1}(a_{N})\ge 0$, hence these are well-defined probabilities.
We claim that $\{\lambda_{N}\}_{N\in\N}$ is projective, that is, that we have $\lambda_{N+1}(\cdot \times \llbracket 0,N+1\rrbracket) = \lambda_{N}(\cdot)$ as measures on $\prod_{M=1}^{N}\llbracket 0,M\rrbracket$; since $\prod_{M=1}^{N}\llbracket 0,M\rrbracket$ is a finite set, it suffices to check that the two measures assign the same value to each point.
For $N=1$ this follows since for any $a_1 \in \llbracket 0,1\rrbracket$ we have
\begin{align*}
\lambda_2(y_1 = a_1) &= \lambda_2(y_1 = a_1,y_2 = a_1) + \lambda_2(y_1 = a_1,y_2 = 2) \\
&\qquad= \frac{\nu_2(a_1)}{\nu_1(a_1)}\nu_1(a_1) + \frac{\nu_1(a_1)-\nu_2(a_1)}{\nu_1(a_1)}\nu_1(a_1) = \nu_1(a_1)
\end{align*}
by construction.
For general $N\in\N$, this follows since for any $\{a_M\}_{M=1}^{N}\in \prod_{M=1}^{N}\llbracket 0,M \rrbracket\setminus\vinitstate_{N}$ we have
\begin{align*}
&\lambda_{N+1}(y_1 = a_1,\ldots, y_{N} = a_{N}) \\
&= \lambda_{N}(y_1 = a_1,\ldots, y_N = a_N) = 0.
\end{align*}
Moreover, for any $N\in\N$ and $\{a_M\}_{M=1}^{N}\in \vinitstate_{N}$ we have
\begin{align*}
&\lambda_{N+1}(y_1 = a_1,\ldots, y_{N} = a_{N}) \\
&\qquad= \lambda_{N+1}(y_1 = a_1,\ldots, y_{N} = a_{N}, y_{N+1} = a_{N+1}) \\
&\qquad\qquad+ \lambda_{N+1}(y_1 = a_1,\ldots, y_{N} = a_{N}, y_{N+1}=N+1) \\
&\qquad= \frac{\nu_{N+1}(a_{N})}{\nu_{N}(a_{N})}\lambda_{N}(y_1 = a_1,\ldots, y_{N} = a_{N}) \\
&\qquad\qquad+ \frac{\nu_{N}(a_{N}) - \nu_{N+1}(a_{N})}{\nu_{N}(a_{N})}\lambda_{N}(y_1 = a_1,\ldots, y_{N} = a_{N}) \\
&\qquad= \lambda_{N}(y_1 = a_1,\ldots, y_N = a_N).
\end{align*}
Thus, by \cite[Theorem~6.14]{Kallenberg}, there is a measure $\lambda$ on $\prod_{N\in\N}\llbracket 0,N\rrbracket$ whose projection onto $\prod_{M=1}^{N}\llbracket 0,M\rrbracket$ is $\lambda_N$ for each $N\in\N$.

Next we claim that $\lambda(\vinitstate) = 1$.
To do this, we first prove $\lambda_{N}(\vinitstate_{N}) = 1$ for each $N\in\N$.
We proceed by induction, and we note that the base case of $N=1$ is immediate by construction.
For the inductive step, let $N\in\N$ be arbitrary, and take any $\{a_M\}_{M=1}^{N+1}\in \prod_{M=1}^{N+1}\llbracket 0,M \rrbracket\setminus\vinitstate_{N+1}$.
If $a_{M+1}\notin \{a_{M},M+1\}$ for some $M\in \llbracket 0,N-1\rrbracket$, then
\begin{align*}
\lambda_{N+1}(y_1=a_1,\ldots, y_{N+1} = a_{N+1}) &\le \lambda_{N+1}(y_1=a_1,\ldots, y_{M+1} = a_{M+1}) \\
&= \lambda_{M+1}(y_1=a_1,\ldots, y_{M+1} = a_{M+1}) = 0,
\end{align*}
by projectivity and the induction hypothesis, and, if $a_{N+1}\notin \{a_{N},N+1\}$, then
\begin{equation*}
\lambda_{N+1}(y_1=a_1,\ldots, y_{N+1} = a_{N+1}) = 0
\end{equation*}
by the definition of $\lambda_{N+1}$.
This completes the induction.
Then note that $\vinitstate = \bigcap_{N\in\N}(\vinitstate_N\times \prod_{M=N+1}^{\infty}\llbracket 0,M \rrbracket)$, so we have
\begin{align*}
\lambda(\vinitstate) &= \lambda\left(\bigcap_{N\in\N}\left(\vinitstate_N\times \prod_{M=N+1}^{\infty}\llbracket 0,M \rrbracket\right)\right) \\
&= \lim_{N\to\infty}\lambda\left(\vinitstate_N\times \prod_{M=N+1}^{\infty}\llbracket 0,M \rrbracket\right) = \lim_{N\to\infty}\boldsymbol{\lambda}_N(\vinitstate_N) = 1,
\end{align*}
as needed.

Next, we claim that $\lambda$ is Markovian.
That is, we need to show
\begin{equation*}
\lambda(y_{N+1}=a_{N+1}|y_1 = a_1,\ldots, y_{N}=a_N) = \lambda(y_{N+1}=a_{N+1}|y_{N}=a_N) 
\end{equation*}
for all $N\in\N$ and $\{a_M\}_{M=1}^{N+1}$.
If $a_{N+1}\notin\{a_N,N+1\}$, then this follows from $\lambda(\vinitstate) = 1$ since both sides are equal to zero.
Thus, if we show the claim for $a_{N+1}=a_N$, then it automatically holds for $a_{N+1} = N+1$ by taking complements.
Therefore it suffices to show the desired identity for $a_{N+1} = a_{N}$, and, to do this, show that both sides are equal to $\nu_{N+1}(a_N)/\nu_{N}(a_N)$.
On the one hand, we compute, using projectivity and the definition of $\{\lambda_{N}\}_{N\in\N}$:
\begin{align*}
&\lambda(y_{N+1}=a_{N}\,|\,y_1 = a_1,\ldots, y_{N}=a_N) \\
&\qquad= \frac{\lambda(y_1 = a_1,\ldots, y_{N} = a_{N}, y_{N+1} = a_{N})}{\lambda(y_1 = a_1,\ldots, y_{N} = a_{N})} \\
&\qquad= \frac{\lambda_{N+1}(y_1 = a_1,\ldots, y_{N} = a_{N}, y_{N+1} = a_{N})}{\lambda_{N}(y_1 = a_1,\ldots, y_{N} = a_{N})} = \frac{\nu_{N+1}(a_N)}{\nu_N(a_N)}.
\end{align*}
On the other hand, we compute, by summing over all $\{a_M'\}_{M = 1}^{N}$ with $a_N'=a_N$ and using projectivity and the definition of $\{\lambda_{N}\}_{N\in\N}$:
\begin{align*}
&\lambda(y_{N+1}=a_{N}\,|\,y_{N}=a_N) \\
&\qquad= \frac{\lambda(y_{N} = a_{N}, y_{N+1} = a_{N})}{\lambda(y_{N} = a_{N})} \\
&\qquad= \frac{\lambda_{N+1}(y_{N} = a_{N}, y_{N+1} = a_{N})}{\lambda_N(y_{N} = a_{N})} \\
&\qquad= \frac{\sum\lambda_{N+1}(y_1 = a_1',\ldots, y_{N} = a_{N}', y_{N+1} = a_{N})}{\sum\lambda_{N}(y_1 = a_1',\ldots, y_{N} = a_{N}')} \\
&\qquad= \frac{\sum\frac{\nu_{N+1}(a_N)}{\nu_N(a_N)}\lambda_{N+1}(y_1 = a_1',\ldots, y_{N} = a_{N}')}{\sum\lambda_{N}(y_1 = a_1',\ldots, y_{N} = a_{N}')} \\
&\qquad= \frac{\nu_{N+1}(a_N)}{\nu_N(a_N)}.
\end{align*}
This shows that $\lambda$ is indeed Markovian.

Finally, we claim that $\phi(\lambda) = \boldsymbol{\nu}$, that is, that
\begin{equation*}
\lambda\left(\prod_{M=1}^{N-1}\llbracket 0,M\rrbracket\times \ \cdot\ \times \prod_{M={N+1}}^{\infty}\llbracket 0,M\rrbracket\right) = \nu_N(\cdot)
\end{equation*}
holds as measures on $\llbracket 0,N \rrbracket$ for each $N\in\N$.
By projectivity, it suffices to show that
\begin{equation*}
\lambda_N\left(\prod_{M=1}^{N-1}\llbracket 0,M\rrbracket\times \ \cdot\right) = \nu_N(\cdot)
\end{equation*}
holds as measures on $\llbracket 0,N \rrbracket$ for each $N\in\N$.
Since this is a finite set, it suffices to check that both measures assign the same value to each element.
We proceed by induction, and again we note that the base case of $N=1$ is immediate.
For the inductive step, let $N\in\N$ be arbitrary.
Note by summing and taking complements that it further suffices to prove the claim only for $a_{N+1} \in \llbracket 0,N\rrbracket$.
To do this, we sum over all $\{a_M'\}_{M=1}^{N}\in\vinitstate_{N}$ with $a_{N}' = a_{N+1}$, and apply the definition of $\lambda_{N+1}$ and the inductive hypothesis:
\begin{align*}
&\lambda_{N+1}(y_{N+1} = a_{N+1}) \\
&= \sum \lambda_{N+1}(y_1=a_1',\ldots, y_N = a_N', y_{N+1}=a_{N+1}) \\
&= \sum \frac{\nu_{N+1}(a_{N+1})}{\nu_{N}(a_{N+1})}\lambda_{N}(y_1=a_1',\ldots, y_{N-1}=a_{N-1}',y_N = a_{N+1}) \\
&= \frac{\nu_{N+1}(a_{N+1})}{\nu_{N}(a_{N+1})}\sum\lambda_{N}(y_1=a_1',\ldots, y_{N-1}=a_{N-1}',y_N = a_{N+1}) \\
&= \frac{\nu_{N+1}(a_{N+1})}{\nu_{N}(a_{N+1})}\lambda_{N}(y_N = a_{N+1}) \\
&= \frac{\nu_{N+1}(a_{N+1})}{\nu_{N}(a_{N+1})}\nu_{N}(a) \\
&= \nu_{N+1}(a_{N+1}).
\end{align*}
This completes the induction and proves $\phi(\lambda) = \boldsymbol{\nu}$.
Thus, $\phi$ is a bijection.

Finally, note that $\phi$ is obviously continuous.
Moreover, $\vdist_0$ is compact and $\vdist_1$ is Hausdorff, so it is classical that $\phi$, being a continuous bijection, is actually a homeomorphism.
\end{proof}

\begin{example}
For $N\in\N$, write $\uniform\llbracket 1,N\rrbracket$ for the uniform measure on $\llbracket 1,N\rrbracket$.
Then consider the sequence $\vuniform = \{\uniform\llbracket 1,N\rrbracket\}_{N\in\N}$.
For any $N\in\N$ we clearly have $\uniform\llbracket 1,N\rrbracket(a) = \frac{1}{N}> \frac{1}{N+1} = \uniform\llbracket 1,N+1\rrbracket(a)$ when $a\in\llbracket 1,N\rrbracket$, and also $\uniform\llbracket 1,N\rrbracket(0) = 0 = \uniform\llbracket 1,N+1\rrbracket(0)$.
Thus we have $\vuniform\in\vdist_1$, and an illustrative example is to exactly identify the measure $\phi^{-1}(\vuniform)\in\vdist_{0}$.

To do this, we note that there is a simple correspondence between $\vinitstate$-valued random variables $\mathbf{a} = \{a_N\}_{N\in\N}$ and sequences of $\{0,1\}$-valued random variables $\{H_N\}_{N\in\N}$, given by
\begin{equation}\label{eqn:nu-H-correspondence}
H_{N+1} = \begin{cases}
0, &\text{ if } a_{N+1} = a_{N}, \\
1, &\text{ if } a_{N+1} = N+1,
\end{cases}\qquad
a_{N+1} = \begin{cases}
a_{N}, &\text{ if } H_N = 0 ,\\
N+1, &\text{ if } H_N = 1.
\end{cases}
\end{equation}
We use this by letting $(\Omega,\F,\P)$ be a probability space on which is defined a sequence $\{H_N\}_{N\in\N}$ of independent random variables with with $H_N = \text{Ber}(\frac{1}{N})$ for all $N\in\N$; then let $\mathbf{a} = \{a_N\}_{N\in\N}$ be defined as above, and set $\lambda = \P\circ\mathbf{a}^{-1}$.
Of course for any $N\in\N$ and $\{a_M\}_{M=1}^{N}\in\vinitstate_{N}$, the values $\lambda(y_{N+1}=N+1\,|\,y_{1}=a_1,\ldots, y_{N} = a_{N})$ and $\lambda(y_{N+1}=N+1|y_{N} = a_{N})$ are both equal to $\P(H_{N+1}) = \frac{1}{N+1}$.
Taking complements shows that $\lambda(y_{N+1}=a_{N}\,|\,y_{1}=a_1,\ldots, y_{N} = a_{N})$ and $\lambda(y_{N+1}=a_{N}\,|\,y_{N} = a_{N})$ are also equal.
Hence we have $\lambda(y_{N+1}=a_{N+1}\,|\,y_{1}=a_1,\ldots, y_{N} = a_{N})$ and $\lambda(y_{N+1}=a_{N+1}\,|\,y_{N} = a_{N})$ for all $N\in\N$ and $\{a_M\}_{M=1}^{N+1}\in\vinitstate_{N+1}$, so $\lambda$ is Markovian.

We claim that $\phi(\lambda) = \{\nu_N\}_{N\in\N}$ equals $\vuniform=\{\uniform\llbracket 1,N\rrbracket\}_{N\in\N}$.
To see this, we use induction on $N\in\N$, wherein the base case of $N=1$ is clear.
For the inductive step, let $N\in\N$ be arbitrary.
Note that by summing and taking complements it suffices to show $\nu_{N+1}(a) = \frac{1}{N+1}$ for all $a\in \llbracket 1,N+1\rrbracket$.
Indeed, this follows by independence:
\begin{align*}
\nu_{N+1}(j) &= \P(a_{N+1} = j) \\
&= \P(a_{N} = j, H_{N+1} = 0) \\
&= \P(a_{N} = j)\P(H_{N+1} = 0) = \frac{1}{N}\cdot \frac{N}{N+1} = \frac{1}{N+1}.
\end{align*}
This completes the induction and proves $\phi(\lambda) = \vuniform$.
We call $\vuniform=\{\uniform\llbracket 1,N\rrbracket\}_{N\in\N}\in\vdist$ the \textit{virtual uniform measure}.
In the following section we will study $\vuniform$ and a collection of related VMCs more carefully.
\end{example}

\subsection{Compatibility}\label{subsec:compatibility}
As we have hinted at above, VDs are important since they will generalize the notion of initial distributions for VMCs.
However, not all VDs can serve as an ``initial distribution'' for a VMC.
The reason for this phenomenon is that, for a VMC $\mathbf{X} = \{X_N\}_{N\in\N}$, each chain $X_N$ for $N\in\N$ progresses at its own rate; one way to inrerpret this is that VMCs ``have no natural time-scale''.
Thus, for $i\in\N$, the slice of values $\{X_N(i)\}_{N\in\N}$ does not, in general, represent a collection of values at the same point in time.
The exception is $i=0$, since all of the chains agree that this is the initial time.
Hence, at time $i=0$ we require extra information about the coupling of the process at all levels.
This leads us to study an important notion of compatibility between VDs and VTMs.

\begin{definition}
Take any $\boldsymbol{\nu} = \{\nu_N\}_{N\in\N}\in\vdist_1$ and $\mathbf{K} = \{K_N\}_{N\in\N}\in\vtransmat$.
We say that $\boldsymbol{\nu}$ is \textit{compatible with} $\mathbf{K}$, or that the pair $(\boldsymbol{\nu},\mathbf{K})$ is \textit{compatible}, if we have 
\begin{equation}\label{eqn:VID-VTM-compatibility}
\nu_N(a) = \begin{cases}
\nu_{N+1}(a) + \nu_{N+1}(N+1)\frac{K_{N+1}(N+1,a)}{1-K_{N+1}(N+1,N+1)}, &\text{ if } K_{N+1}(N+1,N+1) < 1, \\
\nu_{N+1}(0)+\nu_{N+1}(N+1), &\text{ else if } a = 0, \\
\nu_{N+1}(a), &\text{ else if } a \neq 0, \\
\end{cases}
\end{equation}
for all $N\in\N$ and $a\in\llbracket 0,N\rrbracket$.
\end{definition}

\begin{remark}
As in the case of the projection operation for transition matrices, we can make the notion of compatibility slightly more concrete by viewing the elements of $\transmat$ as suitable block matrices.
Again let us for each $N\in\N$ identify $K_{N+1}\in \transmat_{N+1}$ with the block matrix
\begin{equation}\tag{\ref{eqn:TM-block-form}}
\begin{pmatrix}
1 & 0 & 0 \\
w & A & u \\
q & v^{\text{T}} & p
\end{pmatrix},
\end{equation}
for $A\in[0,1]^{N\times N}, u,v,w\in[0,1]^{N}$, and $p,q\in[0,1]$, and let us also identify $\nu_{N+1}$ with the column vector
\begin{equation}
\begin{pmatrix}
z \\
r
\end{pmatrix}
\end{equation}
for $z\in[0,1]^{N+1}$ and $r\in[0,1]$,
Then compatibility says that the column vector $\nu_{N}\in [0,1]^{N+1}$ must satisfy the identity
\begin{equation*}
\nu_N =
\begin{cases}
z+\frac{r}{1-p}
\begin{pmatrix}
q  \\
v
\end{pmatrix}, &\text{ if } p < 1, \\
z + \begin{pmatrix}
r \\ 0
\end{pmatrix}, &\text{ if } p = 1.
\end{cases}
\end{equation*}
\end{remark}

The goal of the next few results is to show that every VMC is uniquely described by a compatible pair in $\vdist_1\times \vtransmat$.
The first direction is simple:

\begin{lemma}\label{lem:VIS-VTM-from-VMC}
	Suppose $\mathbf{X} = \{X_N\}_{N\in\N}$ is a VMC on a probability space $(\Omega,\F,\P)$.
	Then, the sequence $\boldsymbol{\nu} = \{\P\circ X_N(0)^{-1}\}_{N\in\N}$ is a VD, the sequence $\mathbf{K} = \{K_N\}_{N\in\N}$ defined via
	\begin{equation}\label{eqn:VTM-from-VMC}
	K_N(a,b) =\begin{cases}
	\P(X_N(1) = b\,|\,X_N(0)=a), &\text{ if } a,b\in\llbracket 0,N \rrbracket, \\
	1, &\text{ else if } a = b, \\
	0, &\text{ else if } a \neq b, \\
	\end{cases} 
	\end{equation}
	for $N\in\N$ is a VTM, and the pair $(\boldsymbol{\nu},\mathbf{K})$ is compatible.
\end{lemma}

\begin{proof}
First let's show $\boldsymbol{\nu}\in\vdist_1$.
To do this, let $N\in\N$ be arbitrary.
Clearly, $\P\circ X_N(0)^{-1}$ is a probability measure on $\llbracket 0,N\rrbracket$.
Moreover, note that $X_{N+1}(0) < N+1$ implies $I_{X_{N+1},N}(0) = \inf\{i\in\N_0: X_{N+1}(0) \in \llbracket 0,N\rrbracket \} = 0$, hence $X_{N}(0) = X_{N+1}(I_{X_{N+1},N}(0)) = X_{N+1}(0)$.
Thus, for any $a\in\llbracket 0,N\rrbracket$ we have
\begin{align*}
&\P(X_N(0) = a) \\
&= \P(X_N(0) = a,X_{N+1}(0) = a) + \P(X_N(0) = a,X_{N+1}(0) = N+1) \\
&\ge \P(X_N(0) = a,X_{N+1}(0) = a) \\
&= \P(X_{N+1}(0) = a),
\end{align*}
as needed.

Next, we need to show $\mathbf{K}\in\vtransmat$.
To do this, let $N\in\N$ be arbitrary, and note that $K_N$ is clearly a transition matrix on $\N_0$ with state $a$ absorbing for all $a\in\N_0\setminus\llbracket 0,N\rrbracket$.
Thus, it only remains to check \eqref{eqn:TM-projection} for all $a,b\in\llbracket 0,N\rrbracket$.
That this holds for $a=0$ is immediate, so we can assume $a\in\llbracket 1,N\rrbracket$, and in this case we have $\{X_{N}(0) = a\} = \{I_{X_{N+1},N}(0) < \infty, X_{N+1}(I_{X_{N+1},N}(0)) =a\}$.
Then, by the strong Markov property, we have
\begin{align*}
K_{N}(a,b) &=\P(X_N(1) = b \,|\,X_N(0) = i) \\
&=\P((P_N(X_{N+1}))(1) = b \,|\,X_N(0) = i) \\
&=\P((P_N(X_{N+1}))(1) = b \,|\,I_{X_{N+1},N}(0)<\infty,X_N(I_{X_{N+1},N}(0)) = a) \\
&=\P((P_N(X_{N+1}))(1) = b \,|\,X_{N+1}(0) = a).
\end{align*}
Now consider some cases.
If $K_{N+1}(N+1,N+1) < 1$, then we of course have $\P(I_{X_{N+1},N}(1) < \infty|X_{N+1}(0) = a) = 1$, hence
\begin{align*}
K_{N}(a,b) &=\P((P_N(X_{N+1}))(1) = j \,|\,X_{N+1}(0) = a) \\
&=\P(I_{X_{N+1},N}(1) < \infty, X_{N+1}(I_{X_{N+1},N}(1)) = b \,|\,X_{N+1}(0) = a) \\
&=\P(X_{N+1}(I_{X_{N+1},N}(1)) = b \,|\,X_{N+1}(0) = a) \\
&= K_{N+1}(a,b) + K_{N+1}(a,N+1)\sum_{k=1}^{\infty}(K_{N+1}(N+1,N+1))^{k-1} K_{N+1}(N+1,b) \\
&= K_{N+1}(a,b) + \frac{K_{N+1}(a,N+1)K_{N+1}(N+1,b)}{1-K_{N+1}(N+1,N+1)}.
\end{align*}
If instead $K_{N+1}(N+1,N+1) = 1$, then we consider the value of $b\in\llbracket 0,N\rrbracket$.
If $b=0$ then
\begin{align*}
K_{N}(a,0) &= \P((P_N(X_{N+1}))(1) = 0 \,|\,X_{N+1}(0) = a) \\
&= \P(I_{X_{N+1},N}(1) < \infty, X_{N+1}(I_{X_{N+1},N}(1)) = 0\,|\,X_{N+1}(0)=a)\\
&\qquad+\P(I_{X_{N+1},N}(1) = \infty\,|\,X_{N+1}(0)=a) \\
&= K_{N+1}(a,0) + K_{N+1}(a,N+1),
\end{align*}
and if $b\neq 0$ then
\begin{align*}
K_{N}(a,b) &= \P((P_N(X_{N+1}))(1) = b \,|\,X_{N+1}(0) = b) \\
&= \P(I_{X_{N+1},N}(1) < \infty, X_{N+1}(I_{X_{N+1},N}(1)) = b \,|\,X_{N+1}(0)=a)\\
&= K_{N+1}(a,b).
\end{align*}
Therefore, we have $\mathbf{K}\in\vtransmat$.

Finally, we need to check that $\boldsymbol{\nu}$ is compatible with $\mathbf{K}$, which amounts to checking \eqref{eqn:VID-VTM-compatibility} for all $N\in\N$ and $a\in \llbracket 0,N\rrbracket$.
To do this, write
\begin{align*}
&\P(X_N(0) = a ) \\
&=\P(X_{N+1}(0) = a) + \P(X_N(0) = a \,|\, X_{N+1}(0) = N+1)\P(X_{N+1}(0) = N+1),
\end{align*}
and let us compute $\P(X_N(0) = a \,|\, X_{N+1}(0) = N+1)$ by considering the necessary cases.
If $K_{N+1}(N+1,N+1) = 1$, then we clearly have
\begin{equation*}
\P(X_N(0) = a \,|\, X_{N+1}(0) = N+1) = \begin{cases}
1, &\text{ if } a = 0, \\
0, &\text{ if } a \neq 0, \\
\end{cases}
\end{equation*}
hence
\begin{equation*}
\P(X_N(0) = j) = \begin{cases}
\P(X_{N+1}(0) = 0) + \P(X_{N+1}(0) = N+1), &\text{ if } a = 0, \\
\P(X_{N+1}(0) = a), &\text{ if } a \neq 0. \\
\end{cases}
\end{equation*}
Otherwise we have $K_{N+1}(N+1,N+1) <1$, so we can compute
\begin{align*}
&\P(X_N(0) = a \,|\, X_{N+1}(0) = N+1) \\
&= \P(I_{X_{N+1},X}(0) < \infty, X_{N+1}(I_{X_{N+1},X}(0)) = a \,|\, X_{N+1}(0) = N+1) \\
&= \sum_{k=0}^{\infty}(K_{N+1}(N+1,N+1))^k K_{N+1}(N+1,a) \\
&= \frac{K_{N+1}(N+1,a)}{1-K_{N+1}(N+1,N+1)},
\end{align*}
hence
\begin{align*}
&\P(X_N(0) = a ) \\
&=\P(X_{N+1}(0) = a) + \P(X_{N+1}(0) = N+1)\frac{K_{N+1}(N+1,a)}{1-K_{N+1}(N+1,N+1)},
\end{align*}
as needed.
This finishes the proof of the result.
\end{proof}

\begin{definition}
In the setting of Lemma~\ref{lem:VIS-VTM-from-VMC}, we say that $\boldsymbol{\nu}$ and $\mathbf{K}$ are the \textit{virtual initial distribution (VID)} and VTM of $\mathbf{X}$, respectively.
Collectively, we say that $(\boldsymbol{\nu},\mathbf{K})$ are the \textit{canonical data} of $\mathbf{X}$.
\end{definition}

\begin{remark}
It would be slightly more precise to call $\boldsymbol{\nu}$ the \textit{initial virtual distritbution} of $\mathbf{X}$, but we with will use the slightly imprecise term because it sounds more natural (cf. Remark~\ref{rem:VMC-name}).
\end{remark}

The next goal is to establish the converse of Lemma~\ref{lem:VIS-VTM-from-VMC}, namely that for any compatible pair $(\boldsymbol{\nu},\mathbf{K})\in\vdist_1\times\vtransmat$ there is a VMC on some probability space with these as its canonical data.

\begin{proposition}\label{prop:VMC-from-VIS-VTM-existence}
If $(\boldsymbol{\nu},\mathbf{K})\in \vdist_1\times\vtransmat$ is any compatible pair, then there is a VMC with canonical data $(\boldsymbol{\nu},\mathbf{K})$.
\end{proposition}

\begin{proof}
Write $\boldsymbol{\nu} = \{\nu_N\}_{N\in\N}$ and $\mathbf{K} = \{K_N\}_{N\in\N}$.
Then, for each $N\in\N$, let $(\Omega_N,\F_N,\P_N)$ be a probability space on which is defined a MC $X_N = \{X_N(i)\}_{i\in\N_0}$ on $\llbracket 0,N\rrbracket$ with initial distribution $\nu_N$ and transition matrix $K_N$.
Let $\mu_N$ denote the law of $\{P_M(X_N)\}_{M = 1}^{N}$ on the product space $\prod_{M=1}^{N}\chain_M$.

Now let us introduce some notation.
For $N\in\N$, write
\begin{equation*}
\vchain_N := \left\{ \{x_M\}_{M=1}^{N}\in \prod_{M=1}^{N}\chain_M : P_M(x_{M+1}) = x_M \text{ for } 1 \le M < N\right\},
\end{equation*}
and write $\mathcal{B}(\vchain_N)$ for the Borel $\sigma$-algebra on $\vchain_N$.
Next, for $\ell\in\N$ and $a_0,\ldots a_{\ell}\in \N$, write
\begin{equation*}
E_N(a_0,\ldots a_{\ell}) := \left\{\{x_M\}_{M=1}^{N}\in \vchain_N: x_N(0) = a_0,\ldots, x_N(\ell) = a_{\ell} \right\},
\end{equation*}
and define
\begin{equation*}
\mathcal{P}_N := \{E_N(a_0,\ldots a_{\ell}): \ell\in\N, a_0,\ldots, a_{\ell}\in \llbracket 0,N\rrbracket  \}.
\end{equation*}
Observe that $\mathcal{P}_N$ is a $\pi$-system with $\sigma(\mathcal{P}_N) = \mathcal{B}(\vchain_N)$.
Also, $\mathcal{B}( \prod_{M=1}^{N}\chain_M\setminus \vchain_N)$ is a $\pi$-system (in fact, a $\sigma$-algebra), and $\mathcal{P}_N$ and $\mathcal{B}( \prod_{M=1}^{N}\chain_M\setminus \vchain_N)$ are disjoint apart from that they both contain the empty set.
Therefore, $\mathcal{P}_N\cup\mathcal{B}( \prod_{M=1}^{N}\chain_M\setminus \vchain_N)$ is a $\pi$-system which generates $\mathcal{B}( \prod_{M=1}^{N}\chain_M)$

Now we claim that $\{\mu_N\}_{N\in\N}$ is projective, that is, that $\mu_{N+1}(\cdot \times \chain_{N+1}) = \mu_N(\cdot)$ on $\mathcal{B}(\prod_{M=1}^{N}\chain_M)$ for all $N\in\N$.
By the previous paragraph, it is enough to check this for sets in $\mathcal{P}_N\cup\mathcal{B}( \prod_{M=1}^{N}\chain_M\setminus \vchain_N)$; we note that $A\in \mathcal{B}( \prod_{M=1}^{N}\chain_M\setminus \vchain_N)$ has $\mu_{N+1}(A \times \chain_{N+1}) = 0 = \mu_N(A)$, so it only remains to check this for sets in $\mathcal{P}_N$.
More precisely, we need to show that for $N,\ell\in\N$ and $a_0,\ldots a_{\ell}\in\llbracket 0,N\rrbracket$, we have
\begin{align*}
&\mu_{N+1}(E_N(a_0,\ldots, a_{\ell})\times\chain_{N+1}) \\
&= \P_{N+1}((P_N(X_{N+1}))(0) = a_0,\ldots, (P_N(X_{N+1}))(\ell) = a_\ell) \\
&= \P_{N}(X_{N}(0) = a_0,\ldots, X_{N}(\ell) = a_\ell) \\
&=\mu_{N}(E_N(a_0,\ldots, a_{\ell})).
\end{align*}
Since 0 is an absorbing state for both $X_N$ and $P_N(X_{N+1})$, it suffices to assume that $a_0,\ldots, a_{\ell-1}\in\llbracket 1,N\rrbracket$ and $a_{\ell}\in\llbracket 0,N\rrbracket$, since otherwise both sides above are equal to zero.

To do this, we need to check a few different cases, and, in each case, we simply apply the definitions compatibility and of VTM.
If $K_{N+1}(N+1,N+1) < 1$, then we have
\begin{align*}
&\P_{N+1}((P_N(X_{N+1}))(0) = a_0,\ldots, (P_N(X_{N+1}))(\ell) = a_\ell) \\
&= \left(\nu_{N+1}(a_0)+\nu_{N+1}(N+1)\frac{K_{N+1}(N+1,a_0)}{1-K_{N+1}(N+1,N+1)}\right) \\
&\qquad\times\prod_{k=1}^{\ell}\left(K_{N+1}(a_{k-1},a_k) + \frac{K_{N+1}(a_{k-1},N+1)K_{N+1}(N+1,a_k)}{1-K_{N+1}(N+1,N+1)}\right) \\
&=\nu_N(a_0)\prod_{k=1}^{\ell}K_{N}(a_{k-1},a_k) \\
&= \P_{N}(X_{N}(0) = a_0,\ldots, X_{N}(\ell) = a_\ell),
\end{align*}
as needed.
Otherwise $K_{N+1}(N+1,N+1) = 1$, and we need to split into a few more cases.
First suppose $\ell=0$, and let us further consider the value of $a_{0}$.
If $a_{0} = 0$, then
\begin{align*}
\P_{N+1}((P_N(X_{N+1}))(0) = 0) &= \nu_{N+1}(0)+\nu_{N+1}(N+1) \\
&= \nu_N(0) = \P_{N}(X_{N}(0) = 0),
\end{align*}
and, if $a_0\neq 0$, then
\begin{align*}
\P_{N+1}((P_N(X_{N+1}))(0) = a_0) &=\nu_{N+1}(a_0) \\
&= \nu_{N}(a_0) = \P_{N}(X_{N}(0) = a_0),
\end{align*}
as needed.
Now suppose $\ell\ge 1$, and let us further consider the value of $a_{\ell}$.
If $a_\ell = 0$, then
\begin{align*}
&\P_{N+1}((P_N(X_{N+1}))(0) = a_0,\ldots, (P_N(X_{N+1}))(\ell-1) = a_{\ell-1},(P_N(X_{N+1}))(\ell) = 0) \\
&= \nu_{N+1}(a_0)\prod_{k=1}^{\ell-1}K_{N+1}(a_{k-1},a_k)(K_{N+1}(a_{\ell-1},0)+K_{N+1}(a_{\ell-1},N+1)) \\
&=\nu_N(a_0)\prod_{k=1}^{\ell-1}K_{N}(a_{k-1},a_k)K_{N}(a_{\ell-1},0) \\
&= \P_{N}(X_{N}(0) = a_0,\ldots, X_{N}(\ell-1) = a_{\ell-1},X_{N}(\ell) = 0),
\end{align*}
and, if $a_\ell \neq 0$, then
\begin{align*}
&\P_{N+1}((P_N(X_{N+1}))(0) = a_0,\ldots,(P_N(X_{N+1}))(\ell) = a_{\ell}) \\
&= \nu_{N+1}(a_0)\prod_{k=1}^{\ell-1}K_{N+1}(a_{k-1},a_k)K_{N+1}(a_{\ell-1},a_{\ell}) \\
&=\nu_N(a_0)\prod_{k=1}^{\ell-1}K_{N}(a_{k-1},a_k)K_{N}(a_{\ell-1},0) \\
&= \P_{N}(X_{N}(0) = a_0,\ldots, X_{N}(\ell) = a_\ell),
\end{align*}
as needed.
This shows that $\{\mu_N\}_{N\in\N}$ is indeed projective.
Thus, by \cite[Theorem~6.14]{Kallenberg}, there exists a probability measure $\mu$ on $\prod_{N\in\N}\chain_N$ with projections onto $\prod_{M=1}^{N}\chain_M$ given by $\mu_N$ for all $N\in\N$.

Now note that, since $\mu_N(\vchain_N) = 1$ for all $N\in\N$ and $\vchain = \bigcap_{N\in\N}(\vchain_N\times \prod_{M=N+1}^{\infty}\chain_M)$, we have
\begin{align*}
\mu(\vchain) &= \mu\left(\bigcap_{N\in\N}\left(\vchain_N\times \prod_{M=N+1}^{\infty}\chain_M\right)\right) \\
&= \lim_{N\to\infty}\mu\left(\vchain_N\times \prod_{M=N+1}^{\infty}\chain_M\right) = \lim_{N\to\infty}\mu_N(\vchain_N) = 1.
\end{align*}
In other words, $\mu$ is the law of a $\vchain$-valued random variable, and, by construction, its marginal distribution on $\chain_N$ is the law of a MC with initial distribution $\nu_N$ and transition matrix $K_N$, for each $N\in\N$.
This proves that $\mu$ is a VMC with canonical data $(\boldsymbol{\nu},\mathbf{K})$, as claimed.
\end{proof}

Finally, we show that the VMC guaranteed by Proposition~\ref{prop:VMC-from-VIS-VTM-existence} is unique in law.

\begin{proposition}\label{prop:VMC-from-VIS-VTM-uniqueness}
Suppose that $(\Omega,\F,\P)$ and $(\Omega',\F',\P')$ are probability spaces on which are defined a VMCs $\mathbf{X}$ and $\mathbf{X}'$, respectively, with canonical data $(\boldsymbol{\nu},\mathbf{K})$ and $(\boldsymbol{\nu}',\mathbf{K}')$, respectively.
If $(\boldsymbol{\nu},\mathbf{K})=(\boldsymbol{\nu}',\mathbf{K}')$, then we have $\P\circ \mathbf{X}^{-1} = \P'\circ (\mathbf{X}')^{-1}$ as measures on $\vchain$.
\end{proposition}

\begin{proof}
For $M\in\N$ and $B_1\in\mathcal{B}(\chain_1),\ldots,B_M \in \mathcal{B}(\chain_M)$, define
\begin{equation*}
E_M(B_1,\ldots, B_M) := \{\{x_N\}_{N\in\N}\in\vchain: x_1\in B_1,\ldots, x_M\in B_M\}
\end{equation*}
and also
\begin{equation*}
\mathcal{P}_M := \{E_M(B_1,\ldots, B_M): B_1\in\mathcal{B}(\chain_1),\ldots,B_M \in \mathcal{B}(\chain_M) \}.
\end{equation*}
Then note that $\mathcal{P}:=\bigcup_{M\in\N}\mathcal{P}_M$ is a $\pi$-system with $\sigma(\mathcal{P}) = \mathcal{B}(\vchain)$.
Hence, it suffices to show that $\P\circ \mathbf{X}^{-1}$ and $\P'\circ (\mathbf{X}')^{-1}$ agree for all sets in $\mathcal{P}$.

To do this, write $\mathbf{X} = \{X_N\}_{N\in\N}$, and recall that, on $\supp(\P\circ\mathbf{X}^{-1})$, the random variables $\{X_{N}\}_{N\in\llbracket 1,M-1\rrbracket}$ are a measurable function of the random variable $X_M$, for all $M\in\N$; the analogous claim is of course also true for $\mathbf{X}' = \{X_N'\}_{N\in\N}$.
Also write $\boldsymbol{\nu} = \{\nu_N\}_{N\in\N}$ and $\boldsymbol{\nu}' = \{\nu_N'\}_{N\in\N}$, well as $\mathbf{K} = \{K_N\}_{N\in\N}$ and $\mathbf{K}' = \{K_N'\}_{N\in\N}$.
Then take any $N\in\N$ and $B_1\in\mathcal{B}(\chain_1),\ldots,B_N\in\mathcal{B}(\chain_N)$, and note that the claim is immediate if $E_N(B_1,\ldots,B_N)$ is empty.
Otherwise, we have
\begin{align*}
(\P\circ\mathbf{X}^{-1})(E_N(B_1,\ldots,B_N)) &= \P(X_1\in B_1,\ldots, X_N\in B_N) \\
&= \P(X_N\in B_N) \\
&= \P'(X'_N\in B_N) \\
&= \P'(X'_1\in B_1,\ldots, X'_N\in B_N) \\
&= (\P'\circ(\mathbf{X}')^{-1})(E_N(B_1,\ldots,B_N)),
\end{align*}
where $\P(X_N\in B_N) = \P'(X'_N\in B_N)$ follows from the classical result applied to $\nu_N = \nu_N'$ and $K_N = K_N'$.
This finishes the proof.
\end{proof}

Summarizing the partial results shown so far, we come to the main result of this section, which provides a representation theorem describing that each VMC is uniquely described by a compatible pair of a VD and a VTM.

\begin{theorem}\label{thm:VMC-VID-VTM-bijection}
There is a homeomorphism from the space of laws of VMCs to the space of compatible pairs of VID and VTM given by sending each law of a VMC to its canonical data.
\end{theorem}

\begin{proof}
Note that Lemma~\ref{lem:VIS-VTM-from-VMC} implies that there is a map
\begin{align*}
\Phi: &\{\mu\in \mathcal{M}_1(\vchain): \mu \text{ is the law of a VMC} \} \\
&\qquad\qquad\qquad\to \{(\boldsymbol{\nu},\mathbf{K})\in\vdist_1\times\vtransmat : (\boldsymbol{\nu},\mathbf{K}) \text{ is compatible} \}
\end{align*}
which sends each law of a VMC to its canonical data.
Then, Proposition~\ref{prop:VMC-from-VIS-VTM-existence} implies that $\Phi$ is surjective and Proposition~\ref{prop:VMC-from-VIS-VTM-uniqueness} implies that $\Phi$ is injective.
Now we apply the usual trick:
The domain of $\Phi$ is compact (the Markov property is preserved under weak limits), and the range of $\Phi$ is Hausdorff (it is a subset of a metrizable space, hence metrizable itself).
Moreover, $\Phi$ is clearly continuous, and a continuous bijection with these properties is automatically a homeomorphism.

\end{proof}

\subsection{Examples}\label{subsec:Ex-2}
Finally, let us see some illustrative examples of the theory just developed. 
To do this we introduce one last notational convention:
In most interesting examples of VMCs, the VTM $\mathbf{K} = \{K_N\}_{N\in\N}$ has $K_N(a,0) = 0$ for all $N\in\N$ and $a\in\llbracket 1,N \rrbracket$, so displaying the zeroth row and column does not add any information.
So, in order to compactify the presentation, we will use matrices with parentheses as boundaries to denote TMs including the zeroth row and column and matrices with brackets as boundaries to denote TMs excluding the zeroth row and column; in \LaTeX, these correspond to the \texttt{pmatrix} and \texttt{bmatrix} environments, respectively.

\begin{example}
Let $\boldsymbol{\sigma} = \{\sigma_N\}_{N\in\N}\in\vperm$ be a virtual permutation and take any $a\in\N$, and let $\mathbf{X}$ be the VMC on $(\Omega,\F,\P_{L_a})$ constructed in Example~\ref{ex:VMC-from-virtual-perm}.
Its VTM $\mathbf{K}=\{K_N\}_{N\in\N}$ is a \textit{virtual permutation matrix} in the sense that $K_N$ is a permutation matrix for each $N\in\N$, and its VID is exactly $\boldsymbol{\nu} = \{\delta_{a_N}\}_{N\in\N}$ for a suitable sequence $\{a_N\}_{N\in\N}\in\prod_{N\in\N}\llbracket 0,N\rrbracket$.
We write $\mathbf{K}(\boldsymbol{\sigma})$ for the VTM of this VMC, which, in particular, does not depend on $a\in\N_0$.
\end{example}

\begin{example}\label{ex:down-from-infty-again}
	Consider the VMC $\mathbf{X}$ on $(\Omega,\F,\P)$ constructed via Example~\ref{ex:down-from-infty}, where $F$ and $\{L_N\}_{N\in\N}$ are left general.
	We define $q_N = (1-F(L_N))/(1-F(L_{N+1}))$ for $N\in\N$, and then we note that its VTM $\mathbf{K} = \{K_N\}_{N\in\N}$ is exactly given by
	\begin{equation*}
	K_N = \begin{bmatrix}
	0 & 0 & 0 & 0 & \cdots & 0 & 0 & 1 \\
	q_1 & 0 & 0 & 0 & \cdots & 0 & 0 & 1-q_1 \\
	0 & q_2 & 0 & 0 & \cdots & 0 & 0 & 1-q_2 \\
	0 & 0 & q_3& 0 & \cdots & 0 & 0 & 1-q_3 \\
	\vdots & \vdots & \vdots & \vdots & \ddots & \vdots & \vdots & \vdots \\
	0 & 0 & 0 & 0 & \cdots & 0 & 0& 1-q_{N-3} \\
	0 & 0 & 0 & 0 & \cdots & q_{N-2} & 0 & 1-q_{N-2} \\
	0 & 0 & 0 & 0 & \cdots & 0 & q_{N-1} & 1-q_{N-1} \\
	\end{bmatrix}
	\end{equation*}
	for each $N\in\N$, and its VID $\boldsymbol{\nu}$ of $\mathbf{X}$ is just $\boldsymbol{\nu} = \{\delta_N\}_{N\in\N}\in\vdist_1$.
	Note also that $K = \lim_{N\to\infty}K_N$ does not exist in $\transmat$, hence by Lemma~\ref{lem:VK-K-characterization} there is no TM corresponding to this VTM.
\end{example}

\begin{example}\label{ex:splitting-VMC-again}
Consider the VMC $\mathbf{X}$ on $(\Omega,\F,\P_0)$ constructed via Example~\ref{ex:splitting-VMC}.
Its VTM $\mathbf{K} = \{K_N\}_{N\in\N}$ and VID $\boldsymbol{\nu} = \{\nu_N\}_{N\in\N}$ are given by
\begin{equation*}
K_N = \begin{bmatrix}
0 & 0 & 0 & 0 & \cdots & 0 & 0 & 1/2 & 1/2 \\
0 & 0 & 0 & 0 & \cdots & 0 & 0 & 1/2 & 1/2 \\
1 & 0 & 0 & 0 & \cdots & 0 & 0 & 0 & 0 \\
0 & 1 & 0 & 0 & \cdots & 0 & 0 & 0 & 0 \\
\vdots & \vdots & \vdots & \vdots & \ddots & \vdots & \vdots & \vdots \\
0 & 0 & 0 & 0 & \cdots & 0 & 0 & 0 & 0 \\
0 & 0 & 0 & 0 & \cdots & 0 & 0 & 0 & 0 \\
0 & 0 & 0 & 0 & \cdots & 1 & 0 & 0 & 0 \\
0 & 0 & 0 & 0 & \cdots & 0 & 1 & 0 & 0 \\
\end{bmatrix}\qquad \text{ and }\qquad
\nu_N = \begin{bmatrix}
0 \\ 0 \\ 0 \\ 0 \\ \vdots \\ 0 \\ 0 \\ 1/2 \\ 1/2
\end{bmatrix}
\end{equation*}
for all $N\in\N$.
Note, in particular, that neither the VID not the VTM depend on the choice of $\{\ell_N\}_{N\in\N}$.
Also observe $K = \lim_{N\to\infty}K_N$ does not exist in $\transmat$, hence by Lemma~\ref{lem:VK-K-characterization} there is no TM corresponding to this VTM.
\end{example}

\begin{example}\label{ex:VMC-RW}
	For each $N\in\N$, set
	\begin{equation*}
	K_N = \begin{bmatrix}
	1/2 & 1/2 & 0 & \cdots & 0 & 0 & 0 \\
	1/2 & 0 & 1/2 & \cdots & 0 & 0 & 0 \\
	0 & 1/2 & 0 & \cdots & 0 & 0 & 0 \\
	\vdots & \vdots & \vdots & \ddots & \vdots & \vdots & \vdots \\
	0 & 0& 0 & \cdots & 0 & 1/2 & 0 \\
	0 & 0& 0 & \cdots & 1/2 & 0 & 1/2 \\
	0 & 0& 0 & \cdots & 0 & 1/2 & 1/2 \\
	\end{bmatrix},
	\end{equation*}
	which is an element of $[0,1]^{(N+1)\times (N+1)}$.
	It is easy to verify that $\mathbf{K} = \{K_N\}_{N\in\N}$ is indeed a VTM.
	In fact, $\mathbf{K} = \iota(K)$, where
	\begin{equation*}
	K = \begin{bmatrix}
	1/2 & 1/2 & 0 & 0 & 0 & \cdots \\
	1/2 & 0 & 1/2 & 0 & 0 & \cdots \\
	0 & 1/2 & 0 & 1/2 & 0 & \cdots \\
	0 & 0 & 1/2 & 0 & 1/2 &\cdots \\
	0 & 0 & 0 & 1/2 & 0 & \cdots \\
	\vdots & \vdots & \vdots & \vdots & \vdots & \vdots & \ddots \\
	\end{bmatrix}.
	\end{equation*}
	In other words, $\mathbf{K}$ is the VTM representing the simple symmetric random walk on $\N$ where the boundary state 1 holds and reflects with equal probability.
	
	In particular, observe that the VID $\boldsymbol{\nu} = \{\delta_N\}_{N\in\N}$ is compatible with $\mathbf{K}$.
	It follows that there is a VMC $\mathbf{X} = \{X_N\}_{N\in\N}$ with canonical data $(\boldsymbol{\nu},\mathbf{K})$, and it satisfies $X_N(0) = N$ for all $N\in\N$.
	Thus we have $\lim_{N\to\infty}X_N(0) = \infty$ hence $\lim_{N\to\infty}X_N$ does not exist in $\chain$.
	Therefore, $\mathbf{X}\notin \iota(\chain)$ almost surely, by Lemma~\ref{lem:VC-C-characterization}.
	Intuitively speaking, $\mathbf{X}$ is the simple random walk on $\N$ where the state 1 holds and reflects with equal probability, but ``started from infinity''.
	This example demonstrates that even if a VTM corresponds to a classical TM, a compatible VID can be such that the resulting VMC does not correspond to a classical MC.
\end{example}

\section{Some Aspects of Convexity}\label{sec:convexity}

In the classical setting, even though the law of a Markov chain (MC) is uniquely characterized by its transition matrix (TM) and initial distribution (ID), one usually regards the TM as containing the rich information and the ID as playing a secondary role.
We adopt this perspective in the present context, where, although the law of a virtual Markov chain (VMC) is uniquely characterized by its virtual transition matrix (VTM) and virtual initial distribution (VID), we will regard the VTM as containing the important information and the VID as playing a secondary role.

This viewpoint naturally leads us to study a few different problems of convexity in infinite-dimensional spaces related to the study of VMCs.
In Subsection~\ref{subsec:compatibility-revisted}, we consider the compact convex space of all VIDs that are compatible with a given VTM, and in Subsection~\ref{subsec:equi-stat} we consider two different compact convex spaces of VIDs that generalize the notions of equilibrium distribution and stationary distribution for classical MCs; in Subsection~\ref{subsec:Ex-3} we compute some examples of the preceding results.
In Subsection~\ref{subsec:Birkhoff-polytope} we aim to understand the convexity structure of the \textit{virtual Birkhoff polytope}, and we show, in particular, that the classical Birkhoff-von Neumann theorem fails in the virtual setting.

\subsection{Compatibility Revisited}\label{subsec:compatibility-revisted}

Motivated by Theorem~\ref{thm:VMC-VID-VTM-bijection}, we are led to study the space of all VIDs that are compatible with a given VTM.
As we will show in this subsection, this is a compact convex set whose extreme points can, in many cases, be concretely understood.

\begin{definition}
	For $\mathbf{K}\in \vtransmat$, write $\vdist_1(\mathbf{K})$ for the set of elements of $\vdist_1$ which are compatible with $\mathbf{K}$.
\end{definition}

We also introduce some notation to simplify the requisite conditions of VTM and of compatibility.
That is, for $\mathbf{K} = \{K_N\}_{N\in\N} \in\prod_{N\in\N}\transmat_{N}$, $N\in\N$, and $a\in\llbracket 0,N\rrbracket$, we define the constants
\begin{equation*}
C^{\mathbf{K}}_{N,a} := \begin{cases}
\frac{K_{N+1}(N+1,a)}{1-K_{N+1}(N+1,N+1)}, &\text{ if } K_{N+1}(N+1,N+1) < 1, \\
1, &\text{ else if } a = 0, \\
0, &\text{ else if } a \neq 0. \\
\end{cases}
\end{equation*}
Also define $C_{0,0}^{\mathbf{K}} = 0$ by convention.
Then, observe that $\mathbf{K}$ is a VTM if and only if we have $K_{N}(a,b) = K_{N+1}(a,b) + K_{N+1}(a,N+1)C^{\mathbf{K}}_{N,b}$ for all $N\in\N$ and $a\in\llbracket 0,N\rrbracket$ and that a VD $\boldsymbol{\nu} = \{\nu_N\}_{N\in\N}\in\vdist_1$ is compatible with a VTM $\mathbf{K}\in\vtransmat$ iff we have $\nu_N(a) = \nu_{N+1}(a) + \nu_{N+1}(N+1)C^{\mathbf{K}}_{N,a}$ for all $N\in\N$ and $a\in\llbracket 0,N\rrbracket$.
Also note that row-stochasticity implies that we have
\begin{equation}\label{eqn:CK-identity}
\sum_{a=0}^{N}C_{N,a}^{\mathbf{K}} = 1
\end{equation}
for any $N\in\N$.

\begin{lemma}\label{lem:dvist-K-cvx-cpt}
	For $\mathbf{K}\in \vtransmat$, the set $\vdist_1(\mathbf{K})$ is compact, and convex.
\end{lemma}

\begin{proof}
	Observe that we can write	
	\begin{equation*}
	\vdist_1(\mathbf{K}) = \bigcap_{\substack{M\in\N \\ a\in\llbracket 0,M\rrbracket}}\left\{\{\nu_N\}_{N\in\N}\in \vdist_1 : \nu_N(a) = \nu_{N+1}(a) + \nu_{N+1}(N+1)C^{\mathbf{K}}_{N,a} \right\}.
	\end{equation*}
	In other words, $\vdist_1(\mathbf{K})$ is defined as the solution set to countably many linear equations whose coefficients depend on $\mathbf{K}$.
	Of course this implies that $\vdist_1(\mathbf{K})$ is convex.
	Moreover, it implies that $\vdist_1(\mathbf{K})$ is closed in $\vdist_1$, which, since $\vdist_1$ is a compact space, implies that $\vdist_1(\mathbf{K})$ is a compact space.
\end{proof}

By Choquet's theorem, we know that every element of $\vdist_1(\mathbf{K})$ can be written uniquely as a mixture of elements of the extreme points of this set, denoted $\ex{\vdist_1(\mathbf{K})}$.
We are hence motivated to understand the structure of this set of extreme points.

\begin{lemma}\label{lem:VTM-point-mass-VID}
	Let $\mathbf{K}\in\vtransmat$ be a VTM and let $M\in\N_0$ be arbitrary.
	Then, there is a unique VID $\boldsymbol{\nu} = \{\nu_N\}_{N\in\N}\in \vdist_1(\mathbf{K})$ such that $\nu_N =\delta_M$ holds for all $N\in\N$ with $N\ge M$.
\end{lemma}

\begin{proof}
	For existence, let $(\Omega,\F,\P)$ be a probability space on which is defined $X = \{X(i)\}_{i\in\N_0}$ a MC in $\N_0$ with TM $K_M$ and ID $\delta_M$.
	By Proposition~\ref{prop:VMC-MC-characterization} and Lemma~\ref{lem:VIS-VTM-from-VMC}, the VMC $\iota(X)$ has some canonical data $\boldsymbol{\nu} = \{\nu_N\}_{N\in\N}\in\vdist_1$ and $\mathbf{K} = \{K_N\}_{N\in\N}\in\vtransmat$ with $(\boldsymbol{\nu},\mathbf{K})$ compatible.
	Note that, for $N\ge M$, we have $P_N(X) = X$, hence
	\begin{equation*}
	\nu_N(a) = \P((P_N(X))= a) = \P(X = a) = \delta_M(a)
	\end{equation*}
	for all $a\in\llbracket 0,N\rrbracket$, as desired.
	Also, by construction, we have $\mathbf{K} = \iota(K_M)$.
	Thus, it only remains to show that $(\mathbf{\nu},\iota(K_M))$ being compatible implies that $(\mathbf{\nu},\mathbf{K})$ is compatible.
	To do this, write $\iota(K_M) = \{K_N'\}_{N\in\N}$, and note that, for $N\in\N$, we have that $N\ge M$ implies $K_N' = K_M$ and that $N\le M$ implies $K_N' = K_N$.
	Now let $N\in\N$ and $a\in\llbracket 0,N\rrbracket$ be arbitrary.
	If $N\le M$ then we have $C_{N,a}^{\mathbf{K}} = C_{N,a}^{\iota(K_M)}$, hence
	\begin{align*}
	\nu_N(a) &= \nu_{N+1}(a) + \nu_{N+1}(N+1)C^{\iota(K_M)}_{N,a} \\
	&= \nu_{N+1}(a) + \nu_{N+1}(N+1)C^{\mathbf{K}}_{N,a}.
	\end{align*}
	If instead $N\ge M$ then we have $\nu_{N+1}(N+1) = 0$, hence
	\begin{equation*}
	\nu_N(a) = \delta_{M}(a) = \nu_{N+1}(a) + \nu_{N+1}(N+1)C^{\mathbf{K}}_{N,a}.
	\end{equation*}
	Therefore, existence is proved.
	
	For uniqueness, take $\boldsymbol{\nu},\boldsymbol{\nu}'\in\vdist_1(\mathbf{K})$ and write $\boldsymbol{\nu} = \{\nu_N\}_{N\in\N}$ and $\boldsymbol{\nu}' = \{\nu_N'\}_{N\in\N}$, and suppose that $\nu_N = \nu_N' = \delta_N$ holds for all $N\in\N$ with $N\ge M$.
	Then applying the compatibility relation $\nu_N(a) = \nu_{N+1}(a) + \nu_{N+1}(N+1)C^{\mathbf{K}}_{N,a}$ for all $N\in\N$ and $a\in\llbracket 0,N\rrbracket$ inductively shows that we also have $\nu_N = \nu_N'$ for all $N\in\N$ with $N\le M$.
	Thus $\nu_N = \nu_N'$ for all $N\in\N$, hence $\boldsymbol{\nu} = \boldsymbol{\nu}'$.	
\end{proof}

For any $\mathbf{K}\in\vtransmat$ and $M\in\N_0$, write $\delta_{M}^{\mathbf{K}}$ for the element of $\vdist_1(\mathbf{K})$ guaranteed by the result above.
Observe also that, if $\mathbf{K}$ arises from the setting of Lemma~\ref{lem:construct-VMC-from-MP}, then we simply have $\delta_M^{\mathbf{K}} = \{\P_{L_M}\circ X_{N}(0)^{-1}\}_{N\in\N}$ for $M\in\N$.
The next result confirms the intuition that $\delta_{M}^{\mathbf{K}}$ are extreme points of $\vdist_1(\mathbf{K})$ for all $M\in\N_0$.

\begin{lemma}
	For arbitrary $\mathbf{K}\in\vtransmat$, we have
	\begin{equation*}
	\ex{\vdist_1(\mathbf{K})} \supseteq \{\delta^{\mathbf{K}}_N: N\in\N_0 \}.
	\end{equation*}
\end{lemma}

\begin{proof}
Take any $M\in\N$ and suppose that $\boldsymbol{\nu},\boldsymbol{\nu}'\in\vdist_1(\mathbf{K})$ and $\alpha\in (0,1)$ are such that we have $(1-\alpha)\boldsymbol{\nu} + \alpha\boldsymbol{\nu}' = \delta_{M}^{\mathbf{K}}$.
Now write $\boldsymbol{\nu} = \{\nu_N\}_{N\in\N}$ and $\boldsymbol{\nu}' = \{\nu_N'\}_{N\in\N}$.
Note that for any $N\in\N$ with $N\ge M$ we have $(1-\alpha)\nu_N + \alpha\nu_N' = \delta_N$, hence $\nu_N = \nu_N' = \delta_N$.
Thus by Lemma~\ref{lem:VTM-point-mass-VID} we have $\boldsymbol{\nu} = \boldsymbol{\nu}' = \delta_{N}^{\mathbf{K}}$, as claimed.
\end{proof}

For most interesting examples of $\mathbf{K}\in\vtransmat$, it is intuitive that $\vdist_1(\mathbf{K})$ should have more extreme points than just those included in $\{\delta^{\mathbf{K}}_N: N\in\N_0\}$. 
(Indeed, consider Example~\ref{ex:down-from-infty-again}:
The VMC $\mathbf{X}$ starts ``from infinity'', and this does not appear to be a mixture of any other starting states.)
In order to get a more complete understanding of such extreme points of $\vdist_1(\mathbf{K})$, we introduce an additional characterization of VIDs.

\begin{definition}
For any $\mathbf{K}\in\vtransmat$, we set
\begin{equation*}
\vdist_2(\mathbf{K}) := \left\{\{p_a\}_{a\in\N_0}\in [0,1]^{\N_0}: p_a \ge \sum_{M=a}^{\infty}C^{\mathbf{K}}_{M,a}p_{M+1} \text{ for } a\in\N_0, \text{ and } p_0 = 1 \right\}.
\end{equation*}
\end{definition}

\begin{lemma}
For any $\mathbf{K}\in\vtransmat$, the set $\vdist_2(\mathbf{K})$ is compact and convex.
\end{lemma}

\begin{proof}
Observe that $\vdist_2(\mathbf{K})$ is just the solution set to countably many non-strict inequalities, hence it is closed in $[0,1]^{\N_0}$.
But $[0,1]^{\N_0}$ is compact by Tychonoff, hence $\vdist_2(\mathbf{K})$ is compact.
Moreover, the inequalities are linear, so $\vdist_2(\mathbf{K})$ is convex.
\end{proof}

\begin{proposition}\label{prop:D1K-D2K-bijection}
For any $\mathbf{K}\in\vtransmat$, the map $\psi_{\mathbf{K}}:\vdist_{2}(\mathbf{K})\to \vdist_1(\mathbf{K})$ defined via $\psi_{\mathbf{K}}(\{p_a\}_{a\in\N_0}) = \{\nu_{N}\}_{N\in\N}$ with
\begin{equation*}
\nu_N(a) = p_a - \sum_{M=a}^{N-1}C_{M,a}^{\mathbf{K}}p_{M+1}
\end{equation*}
for $N\in\N$ and $a\in\llbracket 0,N\rrbracket$ is a well-defined linear homeomorphism.
Its inverse $\psi_{\mathbf{K}}^{-1}:\vdist_1(\mathbf{K})\to \vdist_{2}(\mathbf{K})$ is given by $\psi_{\mathbf{K}}^{-1}(\{\nu_{N}\}_{N\in\N}) = (\{p_a\}_{a\in\N_0})$ with $p_0 = 1$ and $p_a = \nu_a(a)$ for $a\in \N$.
\end{proposition}

\begin{proof}
Let us show that $\psi_{\mathbf{K}}$ is well-defined.
To do this, first take arbitrary $N\in\N$ and $a\in\llbracket 0,N\rrbracket$, and note that, on the one hand, we have
\begin{equation*}
\nu_N(a) = p_a - \sum_{M=a}^{N-1}C_{M,a}^{\mathbf{K}}p_{M+1} \ge p_a - \sum_{M=a}^{\infty}C_{M,a}^{\mathbf{K}}p_{M+1} \ge 0,
\end{equation*}
while, on the other hand, we have
\begin{equation*}
\nu_N(a) = p_a - \sum_{M=a}^{N-1}C_{M,a}^{\mathbf{K}}p_{M+1} \le p_a \le 1.
\end{equation*}
This shows $0\le \nu_N(a) \le 1$.
Next note that, we have
\begin{equation*}
\nu_N(a) = p_a - \sum_{M=a}^{N-1}C_{M,a}^{\mathbf{K}}p_{M+1} \ge p_a - \sum_{M=a}^{N}C_{M,a}^{\mathbf{K}}p_{M+1} = \nu_{N+1}(a),
\end{equation*}
hence $\nu_{N}(a)\ge \nu_{N+1}(a)$ as needed.
Finally, note that for any $N\in\N$ we have, by switching the order of summation and applying \eqref{eqn:CK-identity}:
\begin{align*}
\sum_{a=0}^{N}\nu_N(a) &=
\sum_{a=0}^{N}\left(p_a - \sum_{M=a}^{N-1}C_{M,a}^{\mathbf{K}}p_{M+1}\right) \\
&= \sum_{a=0}^{N}p_a - \sum_{a=0}^{N}\sum_{M=a}^{N-1}C_{M,a}^{\mathbf{K}}p_{M+1} \\
&= \sum_{a=0}^{N}p_a - \sum_{M=0}^{N-1}p_{M+1}\sum_{a=0}^{M}C_{M,a}^{\mathbf{K}} \\
&= \sum_{a=0}^{N}p_a - \sum_{M=0}^{N-1}p_{M+1} = p_0 = 1.
\end{align*}
Thus, $\psi_{\mathbf{K}}:\vdist_2(\mathbf{K})\to \vdist_1(\mathbf{K})$ is indeed a well-defined map.
It is also immediate from the definition that this map is linear.

Now let us show that $\psi_{\mathbf{K}}$ is injective.
Suppose $\{p_a\}_{a\in\N_0},\{p_a'\}_{a\in\N_0}\in\vdist_2(\mathbf{K})$ have $\psi_{\mathbf{K}}(\{p_a\}_{a\in\N_0}) = \psi_{\mathbf{K}}(\{p_a'\}_{a\in\N_0})$.
Write $\psi_{\mathbf{K}}(\{p_a\}_{a\in\N_0}) = \{\nu_N\}_{N\in\N}$ and $\psi_{\mathbf{K}}(\{p'_a\}_{a\in\N_0}) = \{\nu_N'\}_{N\in\N}$, and note that we have $\nu_N(N) = p_N$ and $\nu_N'(N) = p_N'$ for all $N\in\N$.
Hence, we have $p_a = p_a'$ for all $a\in\N$, and we of course also have $p_0= p_0' = 1$.
Thus, $\{p_a\}_{a\in\N_0}=\{p_a'\}_{a\in\N_0}$ so $\psi_{\mathbf{K}}$ is injective.

To see that $\psi_{\mathbf{K}}$ is surjective, take any $\{\nu_N\}_{N\in\N}\in\vdist_1(\mathbf{K})$, and define $\{p_a\}_{a\in\N_0}$ via $p_0 = 1$ and $p_a = \nu_{a}(a)$ for $a\in\N$.
Then take any $N\in\N$ and $a\in\llbracket 1,N\rrbracket$, and compute, using compatibility:
\begin{align*}
p_a - \sum_{M=a}^{N-1}C_{M,a}^{\mathbf{K}}p_{M+1} &= \nu_a(a) - \sum_{M=0}^{N-1}C_{M,a}^{\mathbf{K}}\nu_{M+1}(M+1) \\
&= \nu_a(a) - \sum_{M=a}^{N-1}(\nu_{M}(a)-\nu_{M+1}(a)) = \nu_N(a).
\end{align*}
By summing over these and taking complements, we also get
\begin{equation*}
p_0 - \sum_{M=0}^{N-1}C_{M,0}^{\mathbf{K}}p_{M+1} = \nu_N(0),
\end{equation*}
hence $\psi_{\mathbf{K}}(\{p_a\}_{a\in\N_0}) = \{\nu_N\}_{N\in\N}$.
Therefore, $\psi_{\mathbf{K}}$ is surjective.
This proves that $\psi_{\mathbf{K}}$ is in fact a bijection.

Observe now that the inverse map $\psi_{\mathbf{K}}^{-1}:\vdist_1(\mathbf{K})\to\vdist_2(\mathbf{K})$ is given exactly by $\psi_{\mathbf{K}}^{-1}(\{\nu_N\}_{N\in\N})= \{p_a\}_{a\in\N_0}$ where $p_0 = 1$ and $p_a = \nu_a(a)$ for $a\in\N$, and that this map is clearly continuous.
Since $\vdist_1(\mathbf{K})$ is compact by Lemma~\ref{lem:dvist-K-cvx-cpt} and $\vdist_2(\mathbf{K})$ is Hausdorff, it is classical that $\psi_{\mathbf{K}}^{-1}$ is in fact a homeomorphism, hence that $\psi_{\mathbf{K}}$ is a homeomorphism.
This finishes the proof.
\end{proof}

Let us build the following diagram to succintly summarize the three different perspectives on VIDs that we have developed thus far:
\begin{center}
	\begin{tikzcd}
		\mathcal{M}_1(\vinitstate) \arrow[d, phantom, "\supseteq", sloped] & \prod_{N\in\N}\mathcal{M}_1(\llbracket 0,N\rrbracket) \arrow[d, phantom, "\supseteq", sloped] \arrow[phantom]{d}[xshift=4ex]{\text{\tiny conv.}} & \left[0,1\right]^{\N_0} \arrow[dd, phantom, "\supseteq", sloped] \arrow[phantom]{dd}[xshift=4ex]{\text{\tiny conv.}}\\
		\vdist_0 \arrow[r, leftrightarrow, "\text{\tiny homeo.}"]  & \vdist_1 \arrow[d, phantom, "\supseteq", sloped] \arrow[phantom]{d}[xshift=4ex]{\text{\tiny conv.}} \\
		& \vdist_1(\mathbf{K}) \arrow[r, leftrightarrow, "\text{\tiny lin. homeo.}"]& \vdist_2(\mathbf{K})
	\end{tikzcd}
\end{center}
In more detail, the space $\vdist_0$ describes the law of an entire Markovian probability distribution on $\vinitstate$, and the space $\vdist_1$ describes only the marginal sequences of such probability distributions; however, we saw in Proposition~\ref{prop:D0-D1-bijection} that these spaces were, in fact, homeomorphic.
Now consider fixing some $\mathbf{K}\in\vtransmat$.
The space $\vdist_1(\mathbf{K})$ describes the (convex) subspace of $\vdist_1$ consisting of the sequences of probability distrubutions which are compatible with $\mathbf{K}$, and the space $\vdist_0(\mathbf{K})$ describes each marginal distribution with a single real number; however, we saw in Proposition~\ref{prop:D1K-D2K-bijection} that these spaces were, in fact, linearly homeomorphic.
(The latter correspondence is predicated on the observation that knowing that $\boldsymbol{\nu}$ is compatible with $\mathbf{K}$ renders much of its information as redundant.)
	
	Now we return to the task at hand of understanding the extreme points of $\vdist_1(\mathbf{K})$.
	The importance of Proposition~\ref{prop:D1K-D2K-bijection} is that it shows that it suffices to find the extreme points of $\vdist_2(\mathbf{K})$ and then simply compute their image under $\psi_{\mathbf{K}}^{-1}$.
	Since $\vdist_2(\mathbf{K})$ is just a subset of $[0,1]^{\N_0}$ which is defined by countably many linear inequalities, its extreme point arise by setting ``as many of them as possible'' to equalities.
	While this is by no means easy in general, the plan can be executed in many simple examples.
	The explicit tools for doing so are the following results:
	
	\begin{lemma}\label{lem:D2K-D1K-extreme}
	The maps
	\begin{equation*}
	\psi_{\mathbf{K}}:\ex{\vdist_{2}(\mathbf{K})}\to \ex{\vdist_{1}(\mathbf{K})}
	\end{equation*}
	and
	\begin{equation*}
	\psi_{\mathbf{K}}:\overline{\ex{\vdist_{2}(\mathbf{K})}}\to \overline{\ex{\vdist_{1}(\mathbf{K})}}
	\end{equation*}
	are well-defined homeomorphisms.
	\end{lemma}

	\begin{proof}
	First let us show that $\psi_{\mathbf{K}}:\ex{\vdist_{2}(\mathbf{K})}\to \ex{\vdist_{1}(\mathbf{K})}$ is well-defined.
	Indeed, suppose that $p = \{p_a\}_{a\in\N_0}\in \ex{\vdist_{2}(\mathbf{K})}$ and that $\boldsymbol{\nu},\boldsymbol{\nu}'\in\vdist_1(\mathbf{K})$ and $\alpha\in (0,1)$ have $(1-\alpha)\boldsymbol{\nu} + \alpha\boldsymbol{\nu}' = \psi_{\mathbf{K}}(p)$.
	Then take $\psi_{\mathbf{K}}^{-1}$ of both sides and use the linearity guaranteed by Proposition~\ref{prop:D1K-D2K-bijection} to get $(1-\alpha)\psi_{\mathbf{K}}^{-1}(\boldsymbol{\nu}) + \alpha\psi_{\mathbf{K}}^{-1}(\boldsymbol{\nu}') = p$.
	Since $p\in\ex{\vdist_{2}(\mathbf{K})}$, this implies $\psi_{\mathbf{K}}^{-1}(\boldsymbol{\nu}) = \psi_{\mathbf{K}}^{-1}(\boldsymbol{\nu}')$, hence $\boldsymbol{\nu} = \boldsymbol{\nu}'$.
	Therefore, $\psi_{\mathbf{K}}(p)\in\ex{\vdist_{1}(\mathbf{K})}$, so $\psi_{\mathbf{K}}:\ex{\vdist_{2}(\mathbf{K})}\to \ex{\vdist_{1}(\mathbf{K})}$ is indeed a well-defined map.
	Since $\psi_{\mathbf{K}}:\vdist_{2}(\mathbf{K})\to \vdist_{1}(\mathbf{K})$ is a homeomorphism by Proposition~\ref{prop:D1K-D2K-bijection}, it follows that $\psi_{\mathbf{K}}:\ex{\vdist_{2}(\mathbf{K})}\to \ex{\vdist_{1}(\mathbf{K})}$ is also a homeomorphism.
	Moreover, $\psi_{\mathbf{K}}:\vdist_{2}(\mathbf{K})\to \vdist_{1}(\mathbf{K})$ being a homeomorphism implies that the same result extends to the closures.
	\end{proof}

	\begin{corollary}\label{cor:psiK-deltaK-property}
	For any $\mathbf{K}\in\vtransmat$ and $N\in\N$, the element $\psi_{\mathbf{K}}^{-1}(\delta_N^{\mathbf{K}})$ is the unique $\{p_a\}_{a\in\N_0}\in \vdist_2(\mathbf{K})$ satisfying $p_N = 1$ and $p_{a} = 0$ for all $a > N$.
	\end{corollary}

\begin{proof}
Immediate from Lemma~\ref{lem:VTM-point-mass-VID} and from the characterization of the map $\psi_{\mathbf{K}}^{-1}:\vdist_1(\mathbf{K})\to\vdist_2(\mathbf{K})$ given in the proof of Proposition~\ref{prop:D1K-D2K-bijection}.
\end{proof}

At the end of this section we will see some examples of using these tools to compute the extreme points of $\vdist_1(\mathbf{K})$ for some specific choices of $\mathbf{K}$.

\subsection{Equilibrium and Stationary Distributions}\label{subsec:equi-stat}
Fix $N\in\N$ and let $K\in\transmat_N$ be any TM.
We say that $\nu\in\mathcal{M}_1(\llbracket 0,N\rrbracket)$ is a \textit{stationary distribution} if we have $\nu K = \nu$.
Of course, the name comes from the fact that, if $X$ is the MC with canonical data $(\nu,K)$, then it is a stationary stochastic process.
It is classical that, if $K\in\transmat_N$ is irreducible and aperiodic, then there exists a unique stationary distribution $\nu$, and it moreover has the property that if $X$ is a MC with canonical data $(\nu',K)$ for any $\nu'\in\mathcal{M}_1(\llbracket 0,N\rrbracket)$, then the law of $X(i)$ converges to $\nu$ in distribution as $i\to\infty$.
For this latter reason, stationary distributions are sometimes called \textit{equilibrium distributions}.
In this subsection we explore analogous notions of equilibrium and stationary distributions for VTMs.
Interestingly, in the virtual case it becomes important to distinguish between these notions.

For any $\mathbf{K}\in\vtransmat$, let us say that $\boldsymbol{\nu} = \{\nu_N\}_{N\in\N}\in\vdist_1$ is an \textit{equilibrium distribution for $\mathbf{K} = \{K_N\}_{N\in\N}\in\vtransmat$} if we have $\nu_N K_N = \nu_N$ for all $N\in\N$, and let us say that it is a \textit{stationary distribution for $\mathbf{K}$} if it is an equilibrium distribution for $\mathbf{K}$ which is also compatible with $\mathbf{K}$.
We write $\vdist_1^{\text{eq}}(\mathbf{K})$ for the space of VIDs that are equilibrium distributions for $\mathbf{K}$, and write $\vdist_1^{\text{st}}(\mathbf{K})$ for the space of VIDs that are stationary distributions for $\mathbf{K}$.

\begin{lemma}
	For any $\mathbf{K}\in\vtransmat$, the spaces $\vdist_1^{\text{eq}}(\mathbf{K})$ and $\vdist_1^{\text{st}}(\mathbf{K})$ are compact and convex.
\end{lemma}

\begin{proof}
	By Lemma~\ref{lem:dvist-K-cvx-cpt} and the fact $\vdist_1^{\text{st}}(\mathbf{K}) = \vdist_1^{\text{eq}}(\mathbf{K})\cap \vdist_1(\mathbf{K})$, it suffices to show that $\vdist_1^{\text{eq}}(\mathbf{K})$ is compact and convex.
	To do this, we simply write
	\begin{equation*}
	\vdist_1^{\text{eq}}(\mathbf{K}) = \bigcap_{N\in\N}\left\{\{\nu_N\}_{N\in\N}\in \vdist_1 : \nu_N K_N= \nu_N \right\},
	\end{equation*}
	and the result follows.
\end{proof}

Let us now give two results which justify the nomenclature of the terms just introduced.
We say that a VTM $\mathbf{K} = \{K_N\}_{N\in\N}\in\vtransmat$ is \textit{irreducible} if, for each $N\in\N$ we have $K_N(i,0) = 0$ for all $i\in \llbracket 1,N\rrbracket$ and and that $\llbracket 1,N\rrbracket$ is a communicating class of $K_N$.
We say $\mathbf{K}$ is \textit{aperiodic} if $K_N$ is aperiodic for each $N\in\N$.
The proofs of both results are immediate from the definitions and from the analagous results for classical MCs.

\begin{lemma}\label{lem:VTM-irred-aperiodic}
	If $\mathbf{K} = \{K_N\}_{N\in\N}\in\vtransmat$ is an irreducible and aperiodic VTM, then there exists a unique equilibrium distribution $\boldsymbol{\nu} = \{\nu_N\}_{N\in\N}\in\vdist_1$ for $\mathbf{K}$.
	Moreover, if $\mathbf{X} = \{X_N\}_{N\in\N}$ is any VMC with VTM $\mathbf{K}$, then for each $N\in\N$, the law of $X_N(i)$ converges to $\nu_{N}$ in distribution as $i\to\infty$.
\end{lemma}

\begin{lemma}
	Suppose that $\mathbf{K}\in\vtransmat$ is a VTM with $\boldsymbol{\nu}\in\vdist_1$ a stationary distribution.
	Then, the VMC $\mathbf{X} = \{X_N\}_{N\in\N}$ with canonical data $(\boldsymbol{\nu},\mathbf{K})$ has the property that $X_N$ is a stationary stochastic process for each $N\in\N$.
\end{lemma}

\subsection{Examples}\label{subsec:Ex-3}
Let us see some examples of the ideas developed in this section.

\begin{example}
	Consider the VMC $\mathbf{X}$ on $(\Omega,\F,\P)$ constructed via Example~\ref{ex:down-from-infty}, whose VTM $\mathbf{K}$ is given in Example~\ref{ex:down-from-infty-again}, where $F$ and $\{L_N\}_{N\in\N}$ are left general.
	
	As a first step, let us compute the elements of the set $\vdist_1(\mathbf{K})$.
	By Choquet's theorem, it suffices to find the extreme points of this set.
	To do this, note that we have $C_{N,N}^{\mathbf{K}} = 1$ for all $N\in\N_0$, hence $\vdist_2(\mathbf{K})$ is exactly the space of all non-decreasing sequences in $[0,1]$ with $p_0 = 1$.
	By Lemma~\ref{lem:decreasing-seq-space-ex-2}, we know that the extreme points of this space are exactly the sequences $\{1^N0^{\infty}:N\in\N\}\cup\{1^{\infty}\}$.
	Moreover, we have $\psi_{\mathbf{K}}(1^{N+1}0^{\infty}) = \delta_{N}^{\mathbf{K}}$ for all $N\in\N_0$ by Lemma~\ref{cor:psiK-deltaK-property}, and also $\lim_{N\to\infty}\delta_N^{\mathbf{K}} = \psi_{\mathbf{K}}(1^{\infty})$ which is extreme by Lemma~\ref{lem:D2K-D1K-extreme}.
	Of course, $\lim_{N\to\infty}\delta_N^{\mathbf{K}}$ is just the VID introduced in Example~\ref{ex:down-from-infty-again}.
	Therefore, we have
	\begin{equation*}
	\ex{\vdist_1(\mathbf{K})} = \{\delta_N^{\mathbf{K}} : N\in\N_0\}\cup\left\{\lim_{N\to\infty}\delta_N^{\mathbf{K}}\right\}
	\end{equation*}
	In other words, this VTM has an extremal VID corresponding to starting in each finite level $N\in\N$ as well as a unique extremal VID corresponding to starting at infinity.
	
	Next let us find the points of $\vdist_1^{\text{eq}}(\mathbf{K})$.
	As long as $q_N > 0$ for all $N\in\N$, it follows that $\mathbf{K}$ is irreducible and aperiodic, hence by Lemma~\ref{lem:VTM-irred-aperiodic} that $\#(\vdist_1^{\text{eq}}(\mathbf{K})) = 1$.
	An easy calculation shows the unique element $\boldsymbol{\nu} = \{\nu_N\}_{N\in\N}\in\vdist_1^{\text{eq}}(\mathbf{K})$ is given by
	\begin{equation*}
	\nu_N(a) = z_N^{-1}\prod_{M=1}^{a-1}q_M^{-1} \qquad \text{ where }\qquad
	z_N = \sum_{a=1}^{N}\prod_{M=1}^{a-1}q_M^{-1},
	\end{equation*}
	for $N\in\N$ and $a\in\llbracket 0,N\rrbracket$.
	
	Finally, we find $\vdist_1^{\text{st}}(\mathbf{K})$.
	Note that for $N\in\N$ and $a=1$ we have
	\begin{align*}
	\nu_N(1) &= z_{N}^{-1} > z_{N+1}^{-1} = \nu_{N+1}(1) + \nu_{N+1}(N+1)C_{N,1}^{\mathbf{K}}
	\end{align*}
	which contradicts \eqref{eqn:VID-VTM-compatibility}, hence $\boldsymbol{\nu}$ is not compatible with $\mathbf{K}$.
	Therefore, $\vdist_1^{\text{st}}(\mathbf{K})$ is empty.
	In particular, we see that $\mathbf{K}$ admits equilibrium distributions but no stationary distributions.
\end{example}

\begin{example}
	Consider the VMC $\mathbf{X}$ on $(\Omega,\F,\P_0)$ constructed via Example~\ref{ex:splitting-VMC}, whose VTM $\mathbf{K}$ is given in Example~\ref{ex:splitting-VMC-again}.
	Recall that $\{\ell_N\}_{N\in\N}$ is arbitrary and does not affect the VTM.
	
	To begin, let us compute $\vdist_1(\mathbf{K})$.
	Note that we have $C_{1,1}^{\mathbf{K}} = 1$ as well as $C_{N,N-1}^{\mathbf{K}} = 1$ for all $N\in\N$ with $N\ge 2$, hence
	\begin{equation*}
	\vdist_2(\mathbf{K}) = \{\{p_a\}_{a\in\N_0}\in[0,1]^{\N_0}: p_1 \ge p_2+p_3 \text{ and } p_a \ge p_{a+2} \text{ for } a\ge 2 \}.
	\end{equation*}
	Now we find its extreme points.
	If $p_1 = 0$ then this forces $p_2=p_3 = 0$, hence also $p_a = 0$ for all $a\in\N$, and this point is clearly extreme.
	If $p_1 = 1$, then the resulting point is extreme only when at most one of $p_2$ or $p_3$ is equal to 1; hence the terms of some parity are identically equal to zero, while the terms of the other parity form a non-decreasing sequence whose extremality we can characterize via the results of Appendix~\ref{app:monotone}.
	If the even indices are identically zero, then the extreme points of the odd indices are characterized by Lemma~\ref{lem:decreasing-seq-space-ex}, and if the odd indices are identically zero, then the extreme points of the even indices are characterized by Lemma~\ref{lem:decreasing-seq-space-ex-2}.
	Moreover, it is easy to show that there are no other extreme points.
	Transforming these back to $\vdist_1(\mathbf{K})$ via $\psi_{\mathbf{K}}$, we find
	\begin{equation*}
	\ex{\vdist_1(\mathbf{K})} = \{\delta_N^{\mathbf{K}} : N\in\N_0\}\cup\left\{\lim_{N\to\infty}\delta_{2N}^{\mathbf{K}},\lim_{N\to\infty}\delta_{2N+1}^{\mathbf{K}}\right\}.
	\end{equation*}
	Of course, the VIDs $\lim_{N\to\infty}\delta_{2N}^{\mathbf{K}}$ and $\lim_{N\to\infty}\delta_{2N+1}^{\mathbf{K}}$ both correspond to ``starting at infinity'' but in two different senses.
	
	This point deserves further elaboration.
	Recall the representation of this VID given in Example~\ref{ex:splitting-VMC}, which states that the sample paths of this VMCs represent repeatedly choosing, uniformly at random, one of two different paths coming down from infinity.
	These two VIDs then correspond respectively to starting at the two different ``trailheads''; geometrically speaking, $\lim_{N\to\infty}\delta_{2N}^{\mathbf{K}}$ and $\lim_{N\to\infty}\delta_{2N+1}^{\mathbf{K}}$ correspond to starting infinitesimally left of the origin and infinitesimally right of the origin, respectively.
	At the origin, the corresponding VID is in fact
	\begin{equation*}
	\{\P_0\circ X_N(0)^{-1} \}_{N\in\N} = \frac{1}{2}\left(\lim_{N\to\infty}\delta_{2N}^{\mathbf{K}} + \lim_{N\to\infty}\delta_{2N+1}^{\mathbf{K}}\right).
	\end{equation*}
	This dispels the plausible but false notion that, for any VTM $\mathbf{K}\in\vtransmat$ arising in the setting of Lemma~\ref{lem:construct-VMC-from-MP}, all the elements of $\ex{\vdist_1(\mathbf{K})}\setminus \{\delta_N^{\mathbf{K}}: N\in\N_0\}$ are of the form $\{\P_x\circ X_N(0)^{-1} \}_{N\in\N}$ for some $x\in S$ which is an accumulation point of $\{L_N\}_{N\in\N}$.
	
	Next let us find the points of $\vdist_1^{\text{eq}}(\mathbf{K})$.
	It is clear that $\mathbf{K}$ is irreducible and aperiodic, hence by Lemma~\ref{lem:VTM-irred-aperiodic} that $\#(\vdist_1^{\text{eq}}(\mathbf{K})) = 1$.
	Then note that $K_N$ is doubly-stochastic for each $N\in\N$, hence $\vdist_1^{\text{eq}}(\mathbf{K}) = \{\vuniform\}$.
	
	To see that $\vdist_1^{\text{st}}(\mathbf{K})$ is empty, it suffices to show that $\vuniform$ is not compatible with $\mathbf{K}$.
	Indeed, for any $N\in\N$ with $N\ge 2$ and $a=N-1$ we have
	\begin{equation*}
	\nu_N(N-1) = \frac{1}{N} \neq \frac{2}{N+1} = \nu_{N+1}(N-1) + \nu_{N+1}(N+1)C^{\mathbf{K}}_{N,N-1},
	\end{equation*}
	whence the result.
\end{example}

\begin{example}\label{ex:uniform-VTM}
	Define the VTM $\mathbf{K} = \{K_N\}_{N\in\N}$ via
	\begin{equation*}
	K_N = \begin{bmatrix}
	1/N & 1/N & \cdots & 1/N & 1/N \\
	1/N & 1/N & \cdots & 1/N & 1/N \\
	\vdots & \vdots & \ddots & \vdots & \vdots \\
	1/N & 1/N & \cdots & 1/N & 1/N \\
	1/N & 1/N & \cdots & 1/N & 1/N \\
	\end{bmatrix}
	\end{equation*}
	for all $N\in\N$.
	Since $K_N(a,b) \to 0$ as $N\to\infty$ for all $a\in\N$ and $b\in\N_0$, we have by Lemma~\ref{lem:VK-K-characterization} that $\mathbf{K}\notin \iota(K)$.
	Intuitively speaking, $\mathbf{K}$ is the VTM representing the ``random walk on the infinite clique''.
	
	This example is somewhat more complicated than the previous examples in that we can only get a partial understanding of $\vdist_1(\mathbf{K})$.
	To do this, note that we have $C_{N,a}^{\mathbf{K}} = 1/N$ for all $N\in\N$ and $a\in \llbracket 1,N\rrbracket$ also $C_{0,0}^{\mathbf{K}} = 1$ by convention.
	Hence:	
	\begin{align*}
	\vdist_2(\mathbf{K}) &:= \left\{\{p_a\}_{a\in\N_0}\in [0,1]^{\N_0}: \begin{matrix}
	p_0 = 1 \ge \sum_{M=1}^{\infty}M^{-1}p_{M+1}, \text{ and}\\ p_a \ge \sum_{M=a}^{\infty}M^{-1}p_{M+1} \text{ for } a\in\N
	\end{matrix}\right\}
	\end{align*}
	In principle one can use the map $\psi_{\mathbf{K}}$ to transform this to an understanding of $\vdist_1(\mathbf{K})$, but, unfortunately, we are not able to characterize the extreme points of this set.
	
	A much easier question is to understand the equilibrium and stationary distributions for $\mathbf{K}$.
	Note that $\mathbf{K}$ is clearly irreducible and aperiodic, and moreover that, for each $N\in\N$, the matrix $K_N$ is doubly-stochastic.
	Hence we have $\vdist_{1}^{\text{eq}}(\mathbf{K}) = \{\vuniform\}$.
	One easily checks that $\vuniform$ is in fact compatible with $\mathbf{K}$, hence also that $\vdist_{1}^{\text{st}}(\mathbf{K}) = \{\vuniform\}$.
\end{example}

\subsection{The Virtual Birkhoff Polytope}\label{subsec:Birkhoff-polytope}

As we saw in Example~\ref{ex:VMC-from-virtual-perm}, VMCs can be seen as a generalization of virtual permutations.
We also know the classical Birkhoff-von Neumann theorem that, in a fixed dimension, the (closed) convex hull of the TMs of all permutation matrices is exactly the space of doubly-stochastic matrices.
Hence, it is natural to ask whether, in some suitable sense, the closed convex hull of the VTMs of all virtual permutations coincides with the space of ``doubly-stochastic'' VTMs.
As we show in this section, the natural generalization is false, but we are able to gain some understanding of the convexity structure of this space.

For each $N\in\N$, write $\birkhoff_{N}\subseteq \transmat_{N}$ for the subspace of all doubly-stochastic transition matrices (\textit{DSTMs}) on $\llbracket 1,N\rrbracket$; these are often called the \textit{Birkhoff polytopes}.
Also write $\birkhoff\subseteq \transmat$ for the subspace of all doubly-stochastic transition matrices (\textit{DSTMs}).
We say that a VTM $\mathbf{K} = \{K_N\}_{N\in\N}\in\vtransmat$ is \textit{doubly-stochastic} if $K_N$ is doubly-stochastic for each $N\in\N$, or also that $\mathbf{K}$ is a \textit{DSVTM}, and we write $\vbirkhoff\subseteq \vtransmat$ for the space of all DSVTMs; this is called the \textit{virtual Birkhoff polytope} although, as we will see, it is not a polytope at all.
Note that $\vbirkhoff$ is equivalently characterized as the space of all VTMs which admit the virtual uniform measure $\vuniform$ as an equilibrium distribution.

Note, in particular, that the first row and column of $K_N$ can be ignored for each $N\in\N$, hence can always use the \texttt{bmatrix} environments to describe them as block matrices.
Moreover, the projection operation $P_N:\birkhoff_{N+1}\to \birkhoff_{N}$ can be written as:
\begin{equation}\label{eqn:DS-K-proj-formula}
\begin{bmatrix}
A & u \\
v^{{\text{T}}} & p
\end{bmatrix} \mapsto \begin{cases}
\begin{bmatrix}
A + \frac{1}{1-p}uv^{\text{T}} \\
\end{bmatrix} &\text{ if } p < 1 \\
& \\
\begin{bmatrix}
\ A\  \\
\end{bmatrix}  &\text{ if } p = 1
\end{cases},
\end{equation}
which follows from \eqref{eqn:K-proj-formula}, since \eqref{eqn:TM-block-form} being doubly-stochastic necessarily implies $w=0$ and $q=0$, and since $p=0$ implies $u=v=0$.
Now we develop the requisite properties of these projection maps.

\begin{lemma}\label{lem:DS-VTM-proj-cts}
	For any $N\in\N$, the map $P_N:\birkhoff_{N+1}\to \birkhoff_{N}$ is continuous.
\end{lemma}

\begin{proof}
	Suppose that $\{K_n\}_{n\in\N}$ and $K$ in $\birkhoff_{N+1}$ have $K_n\to K$.
	Using the block form \eqref{eqn:TM-block-form}, write $A_n,u_n,v_n$, and $p_n$ for the blocks of $K_n$ for each $n\in\N$, and write $A,u,v$, and $p$ for the blocks of $K$.
	Then note that we have $A_n\to A,u_n\to u,v_n\to v$, and $p_n\to p$ in their respective topologies as $n\to\infty$.
	If $p<1$, then we have $p_n<1$ for sufficiently large $n\in\N$, hence we clearly have $P_N(K_n)\to P_n(K)$ as $n\to\infty$.
	In instead $p=1$, then we recall the fact that all matrix norms on $\transmat_N$ are equivalent; hence it suffices to show $\|P_N(K_n)- P_n(K)\|_{\text{F}}\to 0$ as $n\to\infty$ where $\|\cdot\|_{\text{F}}$ denotes the Frobenius norm.
	To do this, we simply bound
	\begin{align*}
	\|P_N(K_n)- P_n(K)\|_{\text{F}} &\le \|A_n-A\|_{\text{F}}+\frac{1}{1-p_n}\|u_nv_n^{\text{T}}\|_{\text{F}} \\
	&= \|A_n-A\|_{\text{F}}+\frac{1}{1-p_n}\|u_n\|_2 \|v_n\|_2 \\
	&\le \|A_n-A\|_{\text{F}}+ \frac{1}{1-p_n}\|u_n\|_1 \|v_n\|_1 \\
	&\le \|A_n-A\|_{\text{F}}+ 1- p_n.
	\end{align*}
	Since the right side goes to zero as $n\to\infty$, the result is proved.
\end{proof}

\begin{lemma}
	The map $\iota:\birkhoff\to\vbirkhoff$ is a continuous injection.
\end{lemma}

\begin{proof}
	That $\iota:\birkhoff\to\vbirkhoff$ is an injection follows from Lemma~\ref{lem:iota-K-injection} in that $\iota:\transmat\to\vtransmat$ is an injection, and that $\iota:\birkhoff\to\vbirkhoff$ is continuous follows from Lemma~\ref{lem:DS-VTM-proj-cts}.
\end{proof}

\begin{lemma}
	The space $\vbirkhoff$ is compact.
\end{lemma}

\begin{proof}
	Immediate from Lemma~\ref{lem:DS-VTM-proj-cts}.
\end{proof}

\begin{example}
	Consider
	\begin{equation*}
	K' = \begin{bmatrix}
	1 & 0 & 0 \\
	0 & 0 & 1 \\
	0 & 1 & 0
	\end{bmatrix} \qquad K'' = \begin{bmatrix}
	0 & 0 & 1 \\
	0 & 1 & 0 \\
	1 & 0 & 0
	\end{bmatrix}
	\end{equation*}
	in $\transmat_{3}$, and note these correspond to some elements $\mathbf{K} := \{K_N\}_{N\in\N}:= \iota(K)$ and $\mathbf{K}' := \{K_N'\}_{N\in\N} := \iota(K')$ in $\vtransmat$.
	Since we have
	\begin{equation*}
	P_2(K') = P_2(K'') = \begin{bmatrix}
	1 & 0 \\
	0 & 1
	\end{bmatrix},
	\end{equation*}
	and also that $K_M = K_N$ and $K_M' = K_N'$ for all $M\ge N$, it follows that $\mathbf{K},\mathbf{K}'\in \vbirkhoff$.
	However, the matrix $K = \frac{1}{2}(K'+K'')$ satisfies
	\begin{equation*}
	K = \begin{bmatrix}
	1/2 & 0 & 1/2 \\
	0 & 1/2 & 1/2 \\
	1/2 & 1/2 & 0 \\
	\end{bmatrix} \qquad P_2(K) = \begin{bmatrix}
	3/4 & 1/4 \\
	1/4 & 3/4
	\end{bmatrix}.
	\end{equation*}
	Since $P_2(\frac{1}{2}(K'+K'')) \neq \frac{1}{2}(P_2(K')+P_2(K''))$, it follows that $\vbirkhoff$ is not convex.
\end{example}

Despite the fact that $\vbirkhoff$ is not convex as in the classical case, we can still understand some of its convexity structure.
The following result characterizes exactly when the line segment between two points of $\vbirkhoff$ is fully contained in $\vbirkhoff$, and it implies that the relative interior of any such segment is either fully contained in or fully disjoint from $\vbirkhoff$.

\begin{proposition}\label{prop:vBirkhoff-convexity}
Take any $\mathbf{K},\mathbf{K}'\in\vbirkhoff$ and write $\mathbf{K} = \{K_N\}_{N\in\N}$ and $\mathbf{K}'= \{K_N'\}_{N\in\N}$.
Then, the set of $\alpha\in [0,1]$ for which we have $(1-\alpha)\mathbf{K} + \alpha \mathbf{K}' \in \vbirkhoff$ is $[0,1]$ if for all $N\in\N$ we have at least one of the properties
\begin{enumerate}
\item[(i)] $K_{N+1}(N+1,N+1) = 1$,
\item[(ii)] $K_{N+1}'(N+1,N+1) = 1$,
\item[(iii)] $K_{N+1}(N+1,N+1) < 1$ and $K_{N+1}(N+1,N+1) < 1$, and
\begin{equation*}
\frac{K_{N+1}(a,N+1)}{1-K_{N+1}(N+1,N+1)} = \frac{K_{N+1}'(a,N+1)}{1-K_{N+1}'(N+1,N+1)}
\end{equation*}
for all $a\in\llbracket 1,N\rrbracket$, or
\item[(iv)] $K_{N+1}(N+1,N+1) < 1$ and $K_{N+1}(N+1,N+1) < 1$, and
\begin{equation*}
\frac{K_{N+1}(N+1,b)}{1-K_{N+1}(N+1,N+1)} = \frac{K_{N+1}'(N+1,b)}{1-K_{N+1}'(N+1,N+1)}
\end{equation*}
for all $b\in\llbracket 1,N\rrbracket$,
\end{enumerate}
and it is $\{0,1\}$ if there is some $N\in\N$ for which we have none of these properties.
\end{proposition}

\begin{proof}
To begin, let us define, for $\mathbf{K} = \{K_N\}_{N\in\N}\in\prod_{N\in\N}\birkhoff_{N},N\in\N$, and $a\in\llbracket 0,N\rrbracket$, the values
\begin{equation*}
F^{\mathbf{K}}_{N,a,b} := \begin{cases}
\frac{K_{N+1}(a,N+1)K_{N+1}(N+1,b)}{1-K_{N+1}(N+1,N+1)}, &\text{ if } K_{N+1}(N+1,N+1) < 1, \\
0, &\text{ if } K_{N+1}(N+1,N+1) = 1.
\end{cases}
\end{equation*}
Then note that $\mathbf{K}$ is in $\vbirkhoff$ if and only if for all $N\in\N$ and $a,b\in\llbracket 1,N\rrbracket$ we have $K_{N}(a,b)= K_{N+1}(a,b) + F^{\mathbf{K}}_{N,a,b}$.
Thus for $\mathbf{K} = \{K_N\}_{N\in\N}$ and $\mathbf{K}'= \{K_N'\}_{N\in\N}$ in $\vbirkhoff$ and $\alpha\in [0,1]$, we have
\begin{align*}
K_{N}(a,b) &= K_{N+1}(a,b) + F^{\mathbf{K}}_{N,a,b}, \text{ and }\\
K_{N}'(a,b) &= K_{N+1}'(a,b) + F^{\mathbf{K}'}_{N,a,b},
\end{align*}
for all $N\in\N$ and $a,b\in\llbracket 1,N\rrbracket$, and we want to understand when we have
\begin{align*}
&(1-\alpha)K_{N}(a,b)+ \alpha K_{N}'(a,b) \\
&\qquad= (1-\alpha)K_{N+1}(a,b) + \alpha K_{N+1}(a,b) + F^{(1-\alpha)\mathbf{K}+\alpha \mathbf{K}'}_{N,a,b}
\end{align*}
for all $N\in\N$ and $a,b\in\llbracket 1,N\rrbracket$.
Plugging the first two equations into the third, we see that $(1-\alpha)\mathbf{K}+\alpha \mathbf{K}'\in\vbirkhoff$ if and only if
\begin{equation}\label{eqn:birkhoff-convexity-condition}
(1-\alpha)F^{\mathbf{K}}_{N,a,b} + \alpha F^{\mathbf{K}'}_{N,a,b}= F^{(1-\alpha)\mathbf{K}+\alpha \mathbf{K}'}_{N,a,b}
\end{equation}
for all $N\in\N$ and $a,b\in\llbracket 1,N\rrbracket$.
We use this characterization to complete the two directions of the proof.

Now let us show that, for each $N\in\N$, any of the properties enumerated above implies \eqref{eqn:birkhoff-convexity-condition} for all $a,b\in\llbracket 1,N\rrbracket$.
First suppose both (i) and (ii), that $K_{N+1}(N+1,N+1) = K_{N+1}'(N+1,N+1) = 1$, and note that this implies
\begin{equation*}
F^{(1-\alpha)\mathbf{K}+\alpha \mathbf{K}'}_{N,a,b} = F^{\mathbf{K}}_{N,a,b} = F^{\mathbf{K}'}_{N,a,b} = 0,
\end{equation*}
as needed.
Now suppose that (i) holds but (ii) fails, that is, that $K_{N+1}(N+1,N+1) =1$ and $K_{N+1}'(N+1,N+1) < 1$.
Then we have $(1-\alpha)K_{N+1}(N+1,N+1) +\alpha K_{N+1}'(N+1,N+1) < 1$, and also $K_{N+1}(a,N+1) = K_{N+1}(N+1,b) =0$ by double-stochasticity.
Hence we can compute:
\begin{align*}
&F^{(1-\alpha)\mathbf{K}+\alpha \mathbf{K}'}_{N,a,b} \\
&= \frac{((1-\alpha)\cdot 0+\alpha K_{N+1}'(a,N+1))((1-\alpha)\cdot 0+\alpha K_{N+1}'(N+1,b))}{1-(1-\alpha)\cdot1-\alpha K_{N+1}'(N+1,N+1)} \\
&= \frac{\alpha^2 K_{N+1}'(a,N+1)K_{N+1}'(N+1,b)}{\alpha(1- K_{N+1}'(N+1,N+1))} \\
&= \frac{\alpha K_{N+1}'(a,N+1)K_{N+1}'(N+1,b)}{1- K_{N+1}'(N+1,N+1)} = \alpha F_{N,a,b}^{\mathbf{K}'},
\end{align*}
as needed.
The same argument works for the case that (i) fails and (ii) holds.
Now suppose that both (i) and (ii) fail.
Then, \eqref{eqn:birkhoff-convexity-condition} is equivalent to
\begin{equation}\label{eqn:algebra-miracle}
(1-\alpha)\frac{xy}{1-z} + \alpha\frac{x'y'}{1-z'} = \frac{((1-\alpha)x+\alpha x')((1-\alpha)y+\alpha y')}{1-(1-\alpha)z-\alpha z'}
\end{equation}
for the assignment
\begin{align*}
x &= K_{N+1}(a,N+1), \qquad x' = K_{N+1}'(a,N+1) \\
y &= K_{N+1}(N+1,b), \qquad y' = K_{N+1}'(N+1,b) \\
z &= K_{N+1}(N+1,N+1), \qquad z' = K_{N+1}'(N+1,N+1).
\end{align*}
Now a small arithmetic miracle occurs, as it turns out that \eqref{eqn:algebra-miracle} is equivalent to
\begin{equation}\label{eqn:algebra-miracle-result}
\left(\frac{x}{1-z} - \frac{x'}{1-z'}\right)\left(\frac{y}{1-z} - \frac{y'}{1-z'}\right) = 0,
\end{equation}
which, remarkably, does not depend on $\alpha$.
Now note that either (iii) or (iv) imply that at least one factor in \eqref{eqn:algebra-miracle-result} is zero, hence \eqref{eqn:algebra-miracle} is satisfied, as needed.

Conversely, it is easy to show for each $N\in\N$ that \eqref{eqn:birkhoff-convexity-condition} holding for all $a,b\in\llbracket 1,N\rrbracket$ implies one of the enumerated properties:
If both (i) and (ii) fail, then, as we remarked above, \eqref{eqn:birkhoff-convexity-condition} is equivalent to \eqref{eqn:algebra-miracle-result}, hence at least one of the factors must be equal to zero.
This shows that for all $a,b\in\llbracket 1,N\rrbracket$ we must have either
\begin{equation*}
\frac{K_{N+1}(a,N+1)}{1-K_{N+1}(N+1,N+1)} = \frac{K_{N+1}'(a,N+1)}{1-K_{N+1}'(N+1,N+1)}
\end{equation*}
or
\begin{equation*}
\frac{K_{N+1}(N+1,b)}{1-K_{N+1}(N+1,N+1)} = \frac{K_{N+1}'(N+1,b)}{1-K_{N+1}'(N+1,N+1)},
\end{equation*}
which in turn this implies that at least one of (iii) or (iv) holds.
\end{proof}

The following is a direct translation of the conditions (i) to (iv) in terms of the analogous conditions for VTMs of virtual permutations.

\begin{corollary}\label{cor:vperm-convexity}
	Take any virtual permutations $\boldsymbol{\sigma},\boldsymbol{\sigma}'\in\vperm$, and write $\boldsymbol{\sigma} = \{\sigma_N\}_{N\in\N}$ and $\boldsymbol{\sigma}' = \{\sigma_N'\}_{N\in\N}$.
	Then the set of $\alpha\in [0,1]$ for which we have $(1-\alpha)\mathbf{K}(\boldsymbol{\sigma})+\alpha\mathbf{K}(\boldsymbol{\sigma}')\in\vbirkhoff$ is $[0,1]$ if for all $N\in\N$ we have at least one of
	\begin{enumerate}
		\item[(i)] $\sigma_{N}(N) = N$,
		\item[(ii)] $\sigma_{N}'(N) = N$,
		\item[(iii)] $(\sigma_N)^{-1}(N) = (\sigma_N')^{-1}(N)$, or
		\item[(iv)] $\sigma_N(N) = \sigma_N'(N)$,
	\end{enumerate}
	and it is $\{0,1\}$ if there is some $N\in\N$ for which we have none of these properties.
\end{corollary}

\begin{example}
	To see that Corollary~\ref{cor:vperm-convexity} can give non-trivial information, let $\boldsymbol{\sigma} = \{\sigma_N\}_{N\in\N}$ and $\boldsymbol{\sigma}' = \{\sigma_N'\}_{N\in\N}$ in $\vperm$ be the virtual permutations whose cycle structures are given by
	\begin{equation*}
	\sigma_N = \begin{cases}
	(1)(23)(45)\cdots (N-1 \, N), &\text{ if } N \text{ odd}, \\
	(1)(23)(45)\cdots (N-2 \, N-1)(N), &\text{ if } N \text{ even}, \\
	\end{cases}
	\end{equation*}
	and
	\begin{equation*}
	\sigma_N' = \begin{cases}
	(12)(34)(56)\cdots (N-2 \, N-1)(N), &\text{ if } N \text{ odd}, \\
	(12)(34)(56)\cdots (N-1 \, N), &\text{ if } N \text{ even}, \\
	\end{cases}
	\end{equation*}
	for all $N\in\N$ with $N\ge 2$.
	Note that (i) is true for odd $N$ and that (ii) is true for even $N$, and therefore that we have $(1-\alpha)\mathbf{K}(\boldsymbol{\sigma})+\alpha\mathbf{K}(\boldsymbol{\sigma}')\in\vbirkhoff$ for all $\alpha\in [0,1]$.
	In particular, observe that $\frac{1}{2}\mathbf{K}(\boldsymbol{\sigma}) + \frac{1}{2}\mathbf{K}(\boldsymbol{\sigma}')$ is exactly the VTM given in Example~\ref{ex:VMC-RW}, that is, the VTM of the random walk on $\N$ for which the boundary state 1 reflects and holds with equal probability.
\end{example}

Finally, recall that the \textit{kernel} of a set $S$ in a real vector space is the set of all $x\in S$ such that for all $x'\in S$ and $\alpha\in [0,1]$ we have $(1-\alpha)x+\alpha x' \in S$; this is denoted $\ker(S)$.
Then, $S$ is convex if and only if $\ker(S) = S$ and $S$ is called \textit{star-shaped} if $\ker(S)$ is non-empty.
Hence, the kernel encodes some information about the convexity structure of possibly non-convex sets.
Our last result shows that the kernel of the virtual Birkhoff polytope is a sinlgeton.

\begin{theorem}\label{thm:virtual-Birkhoff-von-Neumann}
We have $\ker(\vbirkhoff) = \{\identity\,\}$, for $\identity = \{I_N\}_{N\in\N}$ the VTM consisting of the identity matrix $I_N$ for each level $N\in\N$.
\end{theorem}

\begin{proof}
That $\identity$ is in $\ker(\vbirkhoff)$ follows from Proposition~\ref{prop:vBirkhoff-convexity} since (i) is always satisfied.
For the converse, suppose $\mathbf{K}\in\ker(\vbirkhoff)$, and write $\mathbf{K} = \{K_N\}_{N\in\N}$.
Note that it suffices to show $K_{N+1}(N+1,N+1)= 1$ for all $N\in\N$.
To do this, take any $N\in\N$ with $N\ge 2$.

We now define three elements of $\vbirkhoff$:
First set 
\begin{equation*}
K^1_{N+1} := \begin{bmatrix}
0 & 1 & 0 & 0 & \cdots & 0 & 0 & 0 & 0 \\
0 & 0 & 0 & 0 & \cdots & 0 & 0 & 0 & 1 \\
0 & 0 & 1 & 0 & \cdots & 0 & 0 & 0 & 0 \\
0 & 0 & 0 & 1 & \cdots & 0 & 0 & 0 & 0 \\
\vdots & \vdots & \vdots & \vdots & \ddots & \vdots & \vdots & \vdots & \vdots \\
0 & 0 & 0 & 0 & \cdots & 1 & 0 & 0 & 0 \\
0 & 0 & 0 & 0 & \cdots & 0 & 1 & 0 & 0 \\
0 & 0 & 0 & 0 & \cdots & 0 & 0 & 1 & 0 \\
1 & 0 & 0 & 0 & \cdots & 0 & 0 & 0 & 0 \\
\end{bmatrix}
\end{equation*}
in $\transmat_{N+1}$, and define $\mathbf{K}^1 = \iota(K_{N+1}^1)$.
(Observe that $\mathbf{K}^1$ is just the virtual permutation matrix corresponding to the classical permutation with cycle structure $(1\, 2\, N+1)$.)
Then we have
\begin{equation*}
\frac{K_{N+1}^1(1,N+1)}{1-K_{N+1}^1(N+1,N+1)} = \frac{K_{N+1}^1(N+1,2)}{1-K_{N+1}^1(N+1,N+1)} = 0.
\end{equation*}
Second, define $K_{N+1}^2\in\transmat_{N+1}$ via
\begin{equation*}
K^2_{N+1} := \begin{bmatrix}
0 & 1/2 & 0 & 0 & \cdots & 0 & 0 & 0 & 1/2 \\
0 & 0 & 1 & 0 & \cdots & 0 & 0 & 0 & 0 \\
0 & 0 & 0 & 1 & \cdots & 0 & 0 & 0 & 0 \\
0 & 0 & 0 & 0 & \cdots & 0 & 0 & 0 & 0 \\
\vdots & \vdots & \vdots & \vdots & \ddots & \vdots & \vdots & \vdots & \vdots \\
0 & 0 & 0 & 0 & \cdots & 0 & 1 & 0 & 0 \\
0 & 0 & 0 & 0 & \cdots & 0 & 0 & 1 & 0 \\
1 & 0 & 0 & 0 & \cdots & 0 & 0 & 0 & 0 \\
0 & 1/2 & 0 & 0 & \cdots & 0 & 0 & 0 & 1/2 \\
\end{bmatrix}
\end{equation*}
and set $\mathbf{K}^2 = \iota(K_{N+1}^2)$.
Then note that we have
\begin{equation*}
\frac{K_{N+1}^2(1,N+1)}{1-K_{N+1}^2(N+1,N+1)} = \frac{K_{N+1}^2(N+1,2)}{1-K_{N+1}^2(N+1,N+1)} = 1.
\end{equation*}
Third, consider $\mathbf{K}^3$ to be the VTM from Example~\ref{ex:uniform-VTM}, and note that we have
\begin{equation*}
\frac{K_{N+1}^3(1,N+1)}{1-K_{N+1}^3(N+1,N+1)} = \frac{K_{N+1}^3(N+1,2)}{1-K_{N+1}^3(N+1,N+1)} = \frac{1}{N}.
\end{equation*}
Now we apply Proposition~\ref{prop:vBirkhoff-convexity} to the pair $(\mathbf{K},\mathbf{K}^i)$ for $i\in\{1,2,3\}$ and we consider the value of $N\in\N$ that we have fixed.
Since (ii) always fails by construction, we must have at least one of (i), (iii), or (iv), holding for each $i\in\{1,2,3\}$.
If (i) were not true, then, by the pigeonhole principle, at least one of (iii) or (iv) would have to hold for two different $i\in\{1,2,3\}$, but this is a contradiction since then one of the values
\begin{equation*}
\frac{K_{N+1}(1,N+1)}{1-K_{N+1}(N+1,N+1)} \qquad\text{or}\qquad \frac{K_{N+1}(N+1,2)}{1-K_{N+1}(N+1,N+1)}
\end{equation*}
would be equal to two distinct values in $\{0,1,\frac{1}{N}\}$.
Therefore, (i) holds, that is, $K_{N+1}(N+1,N+1) = 1$ for all $N\in\N$ with $N\ge 2$.

It still remains to show that $K_2(2,2) = 1$.
To do this, define
\begin{equation*}
K_2(\theta)=\begin{bmatrix}
\theta & 1-\theta \\
1-\theta & \theta \\
\end{bmatrix}
\end{equation*}
for all $\theta\in (0,1)$, and set $\mathbf{K}(\theta) = \iota(K_2(\theta))$.
Again by Proposition~\ref{prop:vBirkhoff-convexity}, we must have that one of the enumerated properties holds for each $\theta \in (0,1)$, and again we see that (ii) always fails.
Thus, if (i) were not true, then we could choose three distinct values of $\theta\in (0,1)$, and the pigeonhole principle would imply that at least one (iii) or (iv) would have to hold for two distinct values of $\theta\in (0,1)$, which is again a contradiction.
Thus, (i) holds, and we have proven $\mathbf{K} = \identity$.
\end{proof}

\appendix

\section{The Space of Monotone Sequences}\label{app:monotone}

In this appendix we prove some results about extremality in convex spaces of monotone sequences living in a given bounded set of reals.
We believe such results are likely well-known, but we could not find a reference.
To begin, define
\begin{equation*}
D = \left\{\{x_k\}_{k\in\N}\in[0,1]^{\N}: x_k \ge x_{k+1} \text{ for all } k\in\N \right\},
\end{equation*}
and observe that $D$ is clearly convex.
Also define
\begin{equation*}
D' = \left\{\{x_k\}_{k\in\N}\in D: x_1 = 1 \right\},
\end{equation*}
which is also convex.
When $[0,1]^{\N}$ is endowed with the product topology which is compact by Tychonoff, then $D$ and $D'$ are compact convex sets.
Hence, one naturally inquires about the extreme points of $D$ and $D'$, which are easy to describe.

For each $m\in\N$, define the sequence $1^m0^{\infty}:= \{x_k\}_{k\in\N}$ via $x_k = \ind\{k\le m\}$ for all $k\in\N$.
Also define $1^{\infty}:= \{x_k\}_{k\in\N}$ via $x_k =1$ for all $k\in\N$ and $0^{\infty}:= \{x_k\}_{k\in\N}$ via $x_k =0$ for all $k\in\N$.

\begin{lemma}\label{lem:decreasing-seq-space-ex}
$\ex{D} = \{0^{\infty}\}\cup\{1^m0^{\infty}: m\in\N\}\cup\{1^{\infty}\}$.
\end{lemma}

\begin{proof}
	That $D$ is compact and convex is clear, as is the statement
	\begin{equation*}
	\ex{D} \supseteq \{0^{\infty}\}\cup\{1^m0^{\infty}: m\in\N\}\cup\{1^{\infty}\},
	\end{equation*}
	so only the reverse inclusion remains to be shown.
	Indeed, suppose that $\{x_k\}_{k\in\N}$ is not in the right side.
	Then the value $k_{\ast} = \inf\{k\in\N: x_k < 1 \}$ is finite and satisfies $x_{k_{\ast}} \in (0,1)$.
	Now define
	\begin{equation*}
	k_{\ast\ast} := \inf\{k\in\N: k \ge k_{\ast} \text{ and } x_{k}> x_{k+1} \}.
	\end{equation*}
	If $k_{\ast\ast} = \infty$, then we set $\varepsilon := \min\{x_{k_{\ast}},1-x_{k_{\ast}}\} > 0$ and define $x^{\pm} = \{x_k^{\pm}\}_{k\in\N_0}$ via
	\begin{equation*}
	x^{\pm}_k = \begin{cases}
	1, &\text{ if } k < k_{\ast}, \\
	x_{k_{\ast}}\pm\frac{1}{2}\varepsilon, &\text{ if } k \ge k_{\ast}. \\
	\end{cases}
	\end{equation*}
	Notice that we have $x\notin\{x^+,x^-\}$ and $x = \frac{1}{2}(x^++x^-)$, thus $x$ is not extreme.
	Otherwisewe have $k_{\ast\ast} < \infty$ and $x_{k_{\ast\ast}}\in (x_{k_{\ast\ast}+1},1)$, so we can set
	\begin{equation*}
	\varepsilon := \min\{x_{k_{\ast}},1-x_{k_{\ast}},x_{k_{\ast\ast}}-x_{k_{\ast\ast}+1},1-x_{k_{\ast\ast}}\} > 0,
	\end{equation*}
	and we can define $x^{\pm} = \{x_k^{\pm}\}_{k\in\N_0}$ via
	\begin{equation*}
	x^{\pm}_k = \begin{cases}
	1, &\text{ if } k < k_{\ast}, \\
	x_{k_{\ast}}\pm\frac{1}{2}\varepsilon, &\text{ if } k_{\ast}\le k \le k_{\ast\ast}, \\
	x_k, &\text{ if } k > k_{\ast\ast}.
	\end{cases}
	\end{equation*}
	Again we have $x\notin\{x^+,x^-\}$ and $x = \frac{1}{2}(x^++x^-)$, so $x$ is not extreme.
	Therefore, the result is proved.
\end{proof}

The same argument shows that we also have:

\begin{lemma}\label{lem:decreasing-seq-space-ex-2}
	$\ex{D'} = \{1^m0^{\infty}: m\in\N\}\cup\{1^{\infty}\}$.
\end{lemma}

Since $1^m0^{\infty}\to 1^{\infty}$ in the product topology as $m\to\infty$, the compact convex sets $D$ and $D'$ both have closed sets of extreme points.

\end{document}